\documentclass[11pt]{article}
\usepackage{geometry}
\usepackage{pst-plot, pstricks}
\usepackage{hyperref}
\usepackage[utf8]{inputenc}
\usepackage{amsmath}
\usepackage{amssymb}
\usepackage{amscd}
\usepackage{textcomp}
\usepackage{amsthm}
\theoremstyle{plain}
\newtheorem{theorem}{Theorem}[section]

\newtheorem{lemma}[theorem]{Lemma}

\newtheorem{corollary}[theorem]{Corollary}
\newtheorem{proposition}[theorem]{Proposition}

\theoremstyle{definition}
\newtheorem{definition}[theorem]{Definition}

\newtheorem{example}[theorem]{Example}
\newtheorem{remark}[theorem]{Remark}

\newtheorem{question}[theorem]{Question}

\theoremstyle{remark}

\usepackage{xy}
\xyoption{all}

\newdir^{ (}{{}*!/-3pt/\dir^{(}}    
\newdir^{  }{{}*!/-3pt/\dir^{}}    
\newdir_{ (}{{}*!/-3pt/\dir_{(}}    
\newdir_{  }{{}*!/-3pt/\dir_{}}

\newcounter{zahl}

\def\theenumi{(\alph{enumi})}

\def\p@enumii{\theenumi}

\newcommand{\DS}{\displaystyle}
\newcommand{\TS}{\textstyle}
\newcommand{\SC}{\scriptstyle}
\newcommand{\SSC}{\scriptscriptstyle}

\DeclareMathOperator{\End}{End}

\DeclareMathOperator{\Frob}{Frob}

\DeclareMathOperator{\GL}{GL}

\DeclareMathOperator{\Koh}{H}
\DeclareMathOperator{\CKoh}{\check H}
\DeclareMathOperator{\Hom}{Hom}

\DeclareMathOperator{\Id}{Id}

\DeclareMathOperator{\Lie}{Lie}

\DeclareMathOperator{\Pic}{Pic}

\DeclareMathOperator{\QEnd}{QEnd}
\DeclareMathOperator{\QHom}{QHom}
\DeclareMathOperator{\Quot}{Frac}

\DeclareMathOperator{\Res}{Res}

\DeclareMathOperator{\Spec}{Spec}

\DeclareMathOperator{\Sym}{Sym}
\DeclareMathOperator{\Tor}{Tor}

\DeclareMathOperator{\Var}{V}

\DeclareMathOperator{\ann}{ann}

\DeclareMathOperator{\coker}{coker}

\newcommand{\et}{{\rm \acute{e}t}}
\newcommand{\fpqc}{{\it fpqc}}
\newcommand{\fppf}{{\it fppf}}

\DeclareMathOperator{\id}{\,id}
\DeclareMathOperator{\im}{im}

\renewcommand{\mod}{{\rm\;mod\;}}
\DeclareMathOperator{\ord}{ord}

\DeclareMathOperator{\rk}{rk}

\renewcommand{\phi}{\varphi}
\renewcommand{\theta}{\vartheta}
\renewcommand{\epsilon}{\varepsilon}
\usepackage{amsfonts}
\newcommand{\BOne} {{\mathchoice{\hbox{\rm1\kern-2.7pt l\kern.9pt}}
                              {\hbox{\rm1\kern-2.7pt l\kern.9pt}}
                              {\hbox{\scriptsize\rm1\kern-2.3pt l\kern.4pt}}
                              {\hbox{\scriptsize\rm1\kern-2.4pt l\kern.5pt}}}}

\newcommand{\BA}{{\mathbb{A}}}

\newcommand{\BF}{{\mathbb{F}}}
\newcommand{\BG}{{\mathbb{G}}}

\newcommand{\BN}{{\mathbb{N}}}

\newcommand{\BP}{{\mathbb{P}}}

\newcommand{\BZ}{{\mathbb{Z}}}

\newcommand{\CalD}{{\cal{D}}}

\newcommand{\CG}{{\cal{G}}}

\newcommand{\CJ}{{\cal{J}}}
\newcommand{\CK}{{\cal{K}}}
\newcommand{\CL}{{\cal{L}}}

\newcommand{\CO}{{\cal{O}}}

\newcommand{\FG}{{\mathfrak{G}}}

\newcommand{\Fa}{{\mathfrak{a}}}
\newcommand{\Fb}{{\mathfrak{b}}}

\newcommand{\Fm}{{\mathfrak{m}}}

\newcommand{\Fp}{{\mathfrak{p}}}

\newbox\dotCOdbox
\newbox\dotCOtbox
\newbox\dotCOsbox
\newbox\dotCOssbox
\setbox\dotCOdbox=\vbox{\hbox{\kern3.5pt\bf.}\vskip-4.5pt\hbox{$\CO$}}
\setbox\dotCOtbox=\vbox{\hbox{\kern3.5pt\bf.}\vskip-4.5pt\hbox{$\CO$}}
\setbox\dotCOsbox=\vbox{\hbox{\kern2.1pt.}
                       \vskip-6.6pt\hbox{$\scriptstyle\CO$}}
\setbox\dotCOssbox=\vbox{\hbox{\kern1.6pt.}
                        \vskip-8.1pt\hbox{$\scriptscriptstyle\CO$}}

\let\setminus\smallsetminus

\newcommand{\es}{\enspace}

\newcommand{\dual}{^{\SSC\lor}}

\newcommand{\mal}{^{\SSC\times}}
\newcommand{\fdot}{{\,{\SSC\bullet}\,}}
\newcommand{\ul}[1]{{\underline{#1}}}
\newcommand{\ol}[1]{{\overline{#1}}}
\newcommand{\wh}[1]{{\widehat{#1}}}
\newcommand{\wt}[1]{{\widetilde{#1}}}

\usepackage{ifthen}

\newcommand{\invlim}[1][]{\ifthenelse{\equal{#1}{}}% falls Argument leer
{\DS \lim_{\longleftarrow}}%                         verwende niedrige Version
{\DS \lim_{\underset{#1}{\longleftarrow}}}%  sonst:  verwende Argument
}

\newcommand{\dirlim}[1][]{\ifthenelse{\equal{#1}{}}% falls Argument leer
{\DS \lim_{\longrightarrow}}%                        verwende niedrige Version
{\DS \lim_{\underset{#1}{\longrightarrow}}}% sonst:  verwende Argument
}

\newcommand{\dbl}{{\mathchoice{\mbox{\rm [\hspace{-0.15em}[}}
                              {\mbox{\rm [\hspace{-0.15em}[}}
                              {\mbox{\scriptsize\rm [\hspace{-0.15em}[}}
                              {\mbox{\tiny\rm [\hspace{-0.15em}[}}}}
\newcommand{\dbr}{{\mathchoice{\mbox{\rm ]\hspace{-0.15em}]}}
                              {\mbox{\rm ]\hspace{-0.15em}]}}
                              {\mbox{\scriptsize\rm ]\hspace{-0.15em}]}}
                              {\mbox{\tiny\rm ]\hspace{-0.15em}]}}}}
\newcommand{\dpl}{{\mathchoice{\mbox{\rm (\hspace{-0.15em}(}}
                              {\mbox{\rm (\hspace{-0.15em}(}}
                              {\mbox{\scriptsize\rm (\hspace{-0.15em}(}}
                              {\mbox{\tiny\rm (\hspace{-0.15em}(}}}}
\newcommand{\dpr}{{\mathchoice{\mbox{\rm )\hspace{-0.15em})}}
                              {\mbox{\rm )\hspace{-0.15em})}}
                              {\mbox{\scriptsize\rm )\hspace{-0.15em})}}
                              {\mbox{\tiny\rm )\hspace{-0.15em})}}}}

\usepackage{color}
\newcounter{commentcounter}
\newcommand{\comment}[1]{\stepcounter{commentcounter}%
                          {\color{red}%
}
\immediate\write16{}
\immediate\write16{Warning: There was still a comment . . . }
\immediate\write16{}}

\def\?{\ 
{\bf\color{red}???}\ 
\immediate\write16{}
\immediate\write16{Warning: There was still a question mark . . . }
\immediate\write16{}}

\long\def\forget#1{}

\def\longto{\longrightarrow}
\def\into{\hookrightarrow}
\def\onto{\twoheadrightarrow}
\def\isoto{\stackrel{}{\mbox{\hspace{1mm}\raisebox{+1.4mm}{$\SC\sim$}\hspace{-3.5mm}$\longrightarrow$}}}
\def\longinto{\lhook\joinrel\longrightarrow}

\newbox\mybox
\def\arrover#1{\mathrel{
       \setbox\mybox=\hbox spread 1.4em{\hfil$\scriptstyle#1$\hfil}
       \vbox{\offinterlineskip\copy\mybox
             \hbox to\wd\mybox{\rightarrowfill}}}}

\DeclareMathOperator{\Dr}{Dr}
\DeclareMathOperator{\CDr}{\CalD{\it r}}
\newcommand{\CoL}[1]{\ell^{\,^\bullet}_{#1}}

\newcommand{\Zar}{{Zar}}

\def\olF{{\,\overline{\!F}}}

\def\olCL{{\,\overline{\!\CL}}}

\def\ulE{{\underline{E\!}\,}}
\def\ulM{{\underline{M\!}\,}}
\def\ulN{{\underline{N\!}\,}}
\def\ulHM{{\underline{\hat M\!}\,}}
\def\ulV{{\underline{V\!}\,}}
\def\ulX{{\underline{X\!}\,}}
\def\ulY{{\underline{Y\!}\,}}

\def\s{\sigma^\ast}

\begin{document}

%%%%%%%%%%%%%%%%%%%%%%%%%%%%%%%%%%%%%%%%%%%%%%%%%%%%%%%%%%%%%%%%%%%%%%

\author{Urs Hartl}

\title{Isogenies of abelian Anderson $A$-modules and $A$-motives}

\maketitle

\begin{abstract}
As a generalization of Drinfeld modules, Greg Anderson introduced abelian $t$-modules and $t$-motives over a perfect field. In this article we study relative versions of these over rings. We investigate isogenies among them. Our main results state that every isogeny possesses a dual isogeny in the opposite direction, and that a morphism between abelian $t$-modules is an isogeny if and only if the corresponding morphism between their associated $t$-motives is an isogeny. We also study torsion submodules of abelian $t$-modules which in general are non-reduced group schemes. They can be obtained from the associated $t$-motive via the finite shtuka correspondence of Drinfeld and Abrashkin. The inductive limits of torsion submodules are the function field analogs of $p$-divisible groups. These limits correspond to the local shtukas attached to the $t$-motives associated with the abelian $t$-modules. In this sense the theory of abelian $t$-modules is captured by the theory of $t$-motives.
 
\noindent
{\it Mathematics Subject Classification (2010)\/}: 
11G09,  % Drinfeld Modules, higher dimensional motives
%11F80,  % Galois representations (Discontinuous groups and automorphic forms)
%11S20,  % Galois theory of local and $p$-adic fields
%11S25,  % Galois cohomology of local and $p$-adic fields
(14K02,  % Isogeny (in 14K: Abelian varieties and schemes)
13A35,  % Characteristic $p$ methods (Frobenius endomorphism) ...
%14F30,  % $p$-adic cohomology, crystalline cohomology
%14G20,  % Local ground fields
%14G22   % Rigid analytic geometry
%14G35,  % Modular and Shimura varieties
14L05)  % Formal groups, $p$-divisible groups
%14M15,  % Grassmannians, Schubert varieties, flag manifolds
%20G25   % Linear algebraic groups over local fields and their integers
\end{abstract}

\bigskip

%%%%%%%%%%%%%%%%%%%%%%%%%%%%%%%%%%%%%%%%%%%%%%%%%%%%%%%%%%%%%%%%%%%%%%
%
%    Introduction
%
%%%%%%%%%%%%%%%%%%%%%%%%%%%%%%%%%%%%%%%%%%%%%%%%%%%%%%%%%%%%%%%%%%%%%%

\setcounter{tocdepth}{1}
\tableofcontents

\section{Introduction}
\setcounter{equation}{0}

\renewcommand{\epsilon}{\Fp}
\newcommand{\name}{\epsilon}

As a generalization of Drinfeld modules \cite{Drinfeld}, Greg Anderson \cite{Anderson86} introduced \emph{abelian $t$-modules} and \emph{$t$-motives} over a perfect field. In this article we study relative versions of these over rings and generalize them to \emph{abelian Anderson $A$-modules} and \emph{$A$-motives}. The upshot of our results is that the entire theory of abelian Anderson $A$-modules is contained in the theory of $A$-motives. More precisely, let $\BF_q$ be a finite field with $q$ elements, let $C$ be a smooth projective geometrically irreducible curve over $\BF_q$ and let $Q=\BF_q(C)$ be its function field. Let $\infty\in C$ be a closed point and let $A=\Gamma(C\setminus\{\infty\},\CO_C)$ be the ring of functions which are regular outside $\infty$. Let $(R,\gamma)$ be an \emph{$A$-ring}, that is a commutative unitary ring together with a ring homomorphism $\gamma\colon A\to R$. We consider the ideal $\CJ:=(a\otimes1-1\otimes\gamma(a)\colon a\in A)=\ker(\gamma\otimes\id_R\colon A_R\to R)\subset A_R:=A\otimes_{\BF_q}R$ and the endomorphism $\sigma:=\id_A\otimes\Frob_{q,R}\colon a\otimes b\mapsto a\otimes b^q$ of $A_R$. For an $A_R$-module $M$ we set $\sigma^*M:=M\otimes_{A_R,\,\sigma}A_R=M\otimes_{R,\,\Frob_{q,R}}R$, and for an element $m\in M$ we write $\sigma_M^*m:=m\otimes1\in\sigma^*M$.

\begin{definition}\label{DefEffectiveAMotive}
An \emph{effective $A$-motive of rank $r$ over an $A$-ring $R$} is a pair $\ulM=(M,\tau_M)$ consisting of a locally free $A_R$-module $M$ of rank $r$ and an $A_R$-homomorphism $\tau_M\colon \sigma^*M\to M$ whose cokernel is annihilated by $\CJ^n$ for some positive integer $n$. We say that $\ulM$ has \emph{dimension $d$} if $\coker\tau_M$ is a locally free $R$-module of rank $d$ and annihilated by $\CJ^d$. We write $\rk\ulM=r$ and $\dim\ulM=d$ for the rank and the dimension of $\ulM$.

A \emph{morphism} $f\colon(M,\tau_M)\to (N,\tau_N)$ between effective $A$-motives is an $A_R$-homomorphism $f\colon M\to N$ which satisfies $f\circ\tau_M=\tau_N\circ\sigma^*f$.
\end{definition}

Note that $\tau_M$ is always injective and $\coker\tau_M$ is always a finite locally free $R$-module by Proposition~\ref{PropAMotiveEffective} below. We give some explanations for this definition in Section~\ref{SectAMotives} and also define non-effective $A$-motives. If $R$ is a perfect field, $A=\BF_q[t]$ and in addition, $M$ is finitely generated over the non-commutative polynomial ring $R\{\tau\}:=\bigl\{\,\TS\sum\limits_{i=0}^nb_i\tau^i\colon n\in\BN_0, b_i\in R\,\bigr\}$ with $\tau b=b^q\tau$, which acts on $m\in M$ via $\tau\colon m\mapsto \tau_M(\sigma_M^*m)$, then $(M,\tau_M)$ is a $t$-motive in the sense of Anderson~\cite[\S\,1.2]{Anderson86}.

\bigskip

Next let us define abelian Anderson $A$-modules. In Section~\ref{SectAndModules} we give some explanations on the terminology in the following

\begin{definition}\label{DefAndModule}
Let $d$ and $r$ be positive integers. An \emph{abelian Anderson $A$-module of rank $r$ and dimension $d$ over $R$} is a pair $\ulE=(E,\phi)$ consisting of a smooth affine group scheme $E$ over $\Spec R$ of relative dimension $d$, and a ring homomorphism $\phi\colon A\to\End_{R\text{-groups}}(E),\,a\mapsto\phi_a$ such that
\begin{enumerate}
\item \label{DefAndModule_A}
there is a faithfully flat ring homomorphism $R\to R'$ for which $E\times_R\Spec R'\cong\BG_{a,R'}^d$ as $\BF_q$-module schemes, where $\BF_q$ acts on $E$ via $\phi$ and $\BF_q\subset A$, \comment{Try to only require that this is an isom of group schemes !}
\item \label{DefAndModule_B}
$\bigl(\Lie \phi_a-\gamma(a)\bigr)^d=0$ on $\Lie E$ for all $a\in A$,
\item \label{DefAndModule_C}
the set $M:=M(\ulE):=M_q(\ulE):=\Hom_{R\text{-groups},\BF_q\text{\rm-lin}}(E,\BG_{a,R})$ of $\BF_q$-equivariant homomorphisms of $R$-group schemes is a locally free $A_R$-module of rank $r$ under the action given on $m\in M$ by
\[
\begin{array}{rll}
A\ni a\colon & M\longto M, & m\mapsto m\circ \phi_a\\[2mm]
R\ni b\colon & M\longto M, & m\mapsto b\circ m
\end{array}
\]
\end{enumerate}
A \emph{morphism} $f\colon(E,\phi)\to (E',\phi')$ between abelian Anderson $A$-modules is a homomorphism of group schemes $f\colon E\to E'$ over $R$ which satisfies $\phi'_a\circ f=f\circ\phi_a$ for all $a\in A$.
\end{definition}

In particular, if $R$ is a perfect field and $A=\BF_q[t]$, then an abelian Anderson $A$-module is nothing else than an abelian $t$-module in the sense of Anderson~\cite[\S\,1.1]{Anderson86}. When $q$ is not a prime and $R$ is not a field, we do not know the answer to the following

\begin{question}\label{QGpIsomEnough}
If we weaken Definition~\ref{DefAndModule}\ref{DefAndModule_A} and only require that there is an isomorphism of \emph{group schemes} $E\times_{\Spec R}\Spec R'\cong\BG_{a,R'}^d$, do we get an equivalent definition?
\end{question}

For general $A$ and $R$, the abelian Anderson $A$-modules of dimension $1$ over $R$ are precisely the Drinfeld $A$-modules over $R$; see Definition~\ref{DefDrinfeldMod} and Theorem~\ref{ThmDrinfeldMod}. Anderson's anti-equivalence \cite[Theorem~1]{Anderson86} between abelian $t$-modules and $t$-motives directly generalizes to the following

\bigskip\noindent
{\bfseries Theorem~\ref{ThmMotiveOfAModule}.}
{\itshape
If $\ulE=(E,\phi)$ is an abelian Anderson $A$-module then $\ulM(\ulE)=(M,\tau_M)$ with $\tau_M\colon\sigma^*M\to M$, $\sigma_M^*m\mapsto\Frob_{q,\BG_{a,R}}\!\circ\, m$ is an effective $A$-motive of the same rank and dimension as $\ulE$. The contravariant functor $\ulE\mapsto\ulM(\ulE)$ is fully faithful. Its essential image consists of all effective $A$-motives $\ulM=(M,\tau_M)$ over $R$ for which there exists a faithfully flat ring homomorphism $R\to R'$ such that $M\otimes_RR'$ is a finite free left $R'\{\tau\}$-module under the map $\tau\colon M\to M,\,m\mapsto\tau_M(\sigma_M^*m)$.
}

\bigskip

The main purpose of this article is to study isogenies and their (co-)kernels over arbitrary $A$-rings $R$. Here a morphism $f\colon\ulE\to\ulE'$ between abelian Anderson $A$-modules over $R$ is an \emph{isogeny} if it is finite and surjective. On the other hand, a morphism $f\in\Hom_R(\ulM,\ulN)$ between $A$-motives over $R$ is an \emph{isogeny} if $f$ is injective and $\coker f$ is finite and locally free as $R$-module. We give equivalent characterizations in Propositions~\ref{PropIsogAModule}, \ref{PropIsogDrinfeldMod} and \ref{PropDim}. The following are our two main results.

\bigskip\noindent
{\bfseries Theorem~\ref{ThmIsogeny}.}
{\itshape
Let $f\in\Hom_R(\ulE,\ulE')$ be a morphism between abelian Anderson $A$-modules and let $\ulM(f)\in\Hom_R(\ulM',\ulM)$ be the associated morphism between the associated effective $A$-motives $\ulM=\ulM(\ulE)$ and $\ulM'=\ulM(\ulE')$. Then
\begin{enumerate}
\item
$f$ is an isogeny if and only if $\ulM(f)$ is an isogeny.
\stepcounter{enumi}
\item
If $f$ is an isogeny, then $\ker f$ and $\coker\ulM(f)$ correspond to each other under the finite shtuka equivalence which we review in Section~\ref{SectFinShtukas}.
\end{enumerate}
}

\bigskip\noindent
{\bfseries Corollary~\ref{CorDualIsog}.}
{\itshape If $f\in\Hom_R(\ulM,\ulN)$ is an isogeny between $A$-motives then there is an element $0\ne a\in A$ and an isogeny $g\in\Hom_R(\ulN,\ulM)$ with $f\circ g=a\cdot\id_\ulN$ and $g\circ f=a\cdot\id_\ulM$. The same is true for abelian Anderson $A$-modules.
}

\bigskip

This leads to the following result about torsion points in Section~\ref{SectTorsionPts}. Let $(0)\ne\Fa\subset A$ be an ideal and let $\ulE=(E,\phi)$ be an abelian Anderson $A$-module over $R$. The \emph{$\Fa$-torsion submodule $\ulE[\Fa]$ of $\ulE$} is the closed subscheme of $E$ defined by $\ulE[\Fa](S)\,=\,\{\,P\in E(S)\colon\phi_a(P)=0\text{ for all }a\in\Fa\,\}$ on any $R$-algebra $S$.

\bigskip\noindent
{\bfseries Theorem~\ref{ThmATorsion}.}
{\itshape
$\ulE[\Fa]$ is a finite locally free group scheme over $R$. It is \'etale over $R$ if and only if $\Fa+\CJ=A_R$. If $\ulM=\ulM(\ulE)$ is the associated $A$-motive then $\ulE[\Fa]$ and $\ulM/\Fa\ulM$ correspond to each other under the finite shtuka equivalence reviewed in Section~\ref{SectFinShtukas}.
}

\bigskip

If $\Fa+\CJ=A_R$ and $\bar s=\Spec\Omega$ is a geometric base point of $\Spec R$, then we also prove in Theorem~\ref{ThmEtFundGp} that $\ulE[\Fa](\Omega)$ is a free $A/\Fa$-module of rank $r$ which carries a continuous action of the \'etale fundamental group $\pi_1^\et(\Spec R,\bar s)$.

\medskip

In the final Section~\ref{SectDivLocAndMod} we turn towards the case where $\epsilon\subset A$ is a maximal ideal and where all elements of $\gamma(\epsilon)\subset R$ are nilpotent. In this case, we can associate with an $A$-motive $\ulM$ over $R$ a \emph{local shtuka} $\ulHM_\epsilon(\ulM)$; see Example~\ref{ExLocShtuka} and with an abelian Anderson $A$-module $\ulE$ a \emph{divisible local Anderson module} $\ulE[\epsilon^\infty]:=\dirlim\ulE[\epsilon^n]$ in the sense of \cite{HartlSingh}; see Definition~\ref{DefZDivGp} and Theorem~\ref{ThmDivLocAModOfAMod}. If $\ulM=\ulM(\ulE)$ then $\ulHM_\epsilon(\ulM)$ and $\ulE[\epsilon^\infty]$ correspond to each other under the local shtuka equivalence from \cite{HartlSingh}; see Theorems~\ref{ThmEqZDivGps} and \ref{ThmDivLocAModOfAMod}.

\bigskip\noindent
{\it Acknowledgments.} The author acknowledges support of the DFG (German Research Foundation) in form of SFB 878. 

\subsection*{Notation}
Throughout this article we denote by
\begin{tabbing}
$\sigma^*M:=M\otimes_{R,\,\Frob_{q,R}}R$ \=\kill
%$A:=\Gamma(C\setminus\{\infty\},\CO_C)$\quad \=\kill
$\BN_{>0}$ and $\BN_0$\>the positive, respectively the non-negative integers,\\[1mm]
$\BF_q$\> a finite field with $q$ elements and characteristic $p$,\\[1mm]
$C$\> a smooth projective geometrically irreducible curve over $\BF_q$,\\[1mm]
$Q:=\BF_q(C)$\> the function field of $C$,\\[1mm]
$\infty$\> a fixed closed point of $C$,\\[1mm]
$\BF_\infty$\> the residue field of the point $\infty\in C$,\\[1mm]
$A:=\Gamma(C\setminus\{\infty\},\CO_C)$\> the ring of regular functions on $C$ outside $\infty$,\\[1mm]
$(R,\gamma\colon A\to R)$\> an \emph{$A$-ring}, that is a ring $R$ with a ring homomorphism $\gamma\colon A\to R$,\\[1mm]
$A_R:=A\otimes_{\BF_q}R$,\\[1mm]
$\sigma:=\id_A\otimes\Frob_{q,R}$\> the endomorphism of $A_R$ with $a\otimes b\mapsto a\otimes b^q$ for $a\in A$ and $b\in R$,\\[1mm]
$\sigma^*M:=M\otimes_{R,\,\Frob_{q,R}}R$\>$=M\otimes_{A_R,\,\sigma}A_R$ the Frobenius pullback for an $A_R$-module $M$,\\[1mm]
$\sigma^*V:=V\otimes_{R,\,\Frob_{q,R}}R$ \> the Frobenius pullback more generally for an $R$-module $V$,\\[1mm]
$\sigma_V^*v:=v\otimes1\in\sigma^*V$\> for an element $v\in V$,\\[1mm]
$\sigma^*f:=f\otimes\id\colon\sigma^*M\to\sigma^*N$ for a morphism $f\colon M\to N$ of $A_R$-modules, \\[1mm]
$\CJ:=\ker(\gamma\otimes\id_R\colon A_R\to R)=(a\otimes1-1\otimes\gamma(a)\colon a\in A)\subset A_R$. 
\end{tabbing}
Note that $\gamma$ makes $R$ into an $\BF_q$-algebra. Further note that $\CJ$ is a locally free $A_R$-module of rank $1$. Indeed, $\CJ=I\otimes_{A_A}A_R$ for the ideal $I:=(a\otimes1-1\otimes a\colon a\in A)\subset A_A=A\otimes_{\BF_q}A$. The latter is a locally free $A_A$-module of rank $1$ by Nakayama's lemma, because $I\otimes_{A_A}A_A/I=I/I^2=\Omega^1_{A/\BF_q}$ is a locally free module of rank $1$ over $A_A/I=A$.

\medskip

We will sometimes reduce from the ring $A$ to the polynomial ring $\BF_q[t]$ by applying the following

\begin{lemma}\label{LemmaF_q[t]}
Let $a\in A\setminus\BF_q$ and let $\BF_q[t]$ be the polynomial ring in the variable $t$. Then the homomorphism $\BF_q[t]\to A,\,t\mapsto a$ makes $A$ into a finite free $\BF_q[t]$-module of rank equal to $-[\BF_\infty:\BF_q]\ord_\infty(a)$, where $\ord_\infty$ is the normalized valuation of the discrete valuation ring $\CO_{C,\infty}$.
\end{lemma}

\begin{proof}
If $\ord_\infty(a)=0$ then $a$ would have no pole on the curve $C$, hence would be constant. Since $C$ is geometrically irreducible this would imply $a\in\BF_q$ which was excluded. Therefore $a$ is non-constant and defines a finite surjective morphism of curves $f\colon C\to\BP^1_{\BF_q}$ with $\Spec A\to\Spec\BF_q[t]=\BP^1_{\BF_q}\setminus\{\infty'\}$, where $\infty'\in\BP^1_{\BF_q}$ is the pole of $t$. By \cite[Proposition~15.31]{GoertzWedhorn} its degree can be computed in the fiber $f^{-1}(\infty')=\{\infty\}$ as $\deg f=[\BF_\infty:\BF_{\infty'}]\!\cdot\!e_f(\infty)$ where $\BF_{\infty'}=\BF_q$ and $e_f(\infty)=\ord_\infty f^*(\tfrac{1}{t})=-\ord_\infty(a)$ is the ramification index of $f$ at $\infty$. Since $\Spec A=f^{-1}(\Spec\BF_q[t])$ we conclude that $A$ is a finite (locally) free $\BF_q[t]$-module of rank $-[\BF_\infty:\BF_q]\ord_\infty(a)$.
\end{proof}

%%%%%%%%%%%%%%%%%%%%%%%%%%%%%%%%%%%%%%%%%%%%%%%%%%%%%%%%%%%%%%%%%%%%%%
%
%    A-Motives
%
%%%%%%%%%%%%%%%%%%%%%%%%%%%%%%%%%%%%%%%%%%%%%%%%%%%%%%%%%%%%%%%%%%%%%%

\section{$A$-Motives}\label{SectAMotives}
\setcounter{equation}{0}

%Let $R$ be a ring equipped with a ring homomorphism $\gamma\colon A\to R$. In particular $R$ is an $\BF_q$-algebra. We set $A_R:=A\otimes_{\BF_q}R$ and $\CJ:=(a\otimes1-1\otimes\gamma(a)\colon a\in A)=\ker(\gamma\otimes\id_R\colon A_R\to R)\subset A_R$. Note that $\CJ$ is generated by the elements $a\otimes1-1\otimes\gamma(a)$ where $a$ lies in a generating system of the ring $A$. In particular $\CJ$ is finitely generated, because the ring $A$ is. We consider the endomorphism $\sigma:=\id_A\otimes\Frob_{q,R}\colon a\otimes b\mapsto a\otimes b^q$ of $A_R$. For an $A_R$-module $M$ we set $\sigma^*M:=M\otimes_{A_R,\,\sigma}A_R$ and for a morphism $f\colon M\to N$ of $A_R$-modules we set $\sigma^*f:=f\otimes\id\colon\sigma^*M\to\sigma^*N$. More generally, for an $R$-module $V$ we set $\sigma^*V:=V\otimes_{R,\,\Frob_{q,R}}R$.

We keep the notation introduced in the introduction and generalize Definition~\ref{DefEffectiveAMotive} to not necessarily effective $A$-motives.

\begin{definition}\label{DefAMotive}
An \emph{$A$-motive of rank $r$ over an $A$-ring $R$} is a pair $\ulM=(M,\tau_M)$ consisting of a locally free $A_R$-module $M$ of rank $r$ and an isomorphism outside the zero locus $\Var(\CJ)$ of $\CJ$ between the induced finite locally free sheaves $\tau_M\colon \sigma^*M|_{\Spec A_R\setminus\Var(\CJ)}\isoto M|_{\Spec A_R\setminus\Var(\CJ)}$.

A \emph{morphism} $f\colon(M,\tau_M)\to (N,\tau_N)$ between $A$-motives is an $A_R$-homomorphism $f\colon M\to N$ which satisfies $f\circ\tau_M=\tau_N\circ\sigma^*f$. We write $\Hom_R(\ulM,\ulN)$ for the $A$-module of morphisms between $\ulM$ and $\ulN$. The elements of $\QHom_R(\ulM,\ulN):=\Hom_R(\ulM,\ulN)\otimes_AQ$ are called \emph{quasi-morphisms}. We also set $\End_R(\ulM):=\Hom_R(\ulM,\ulM)$ and $\QEnd_R(\ulM):=\QHom_R(\ulM,\ulM)=\End_R(\ulM)\otimes_A Q$.
\end{definition}

To explain the relation between Definitions~\ref{DefEffectiveAMotive} and \ref{DefAMotive} we begin with a
 
\begin{lemma}\label{LemmaInjCokerLocFree}
Let $f\colon M\to N$ be a homomorphism between finite locally free $A_R$-modules $M$ and $N$ of the same rank, and assume that $\coker f$ is a finitely generated $R$-module, then $f$ is injective and $\coker f$ is a finite locally free $R$-module.
\end{lemma}

\begin{proof}
To make the proof more transparent, we choose an element $t\in A\setminus\BF_q$. Then $A$ is a finite free $\BF_q[t]$-module by Lemma~\ref{LemmaF_q[t]}, and $M$ and $N$ are finite locally free modules over $R[t]$. Also $t$ acts as an endomorphism of the finite $R$-module $\coker f$. By the Cayley-Hamilton Theorem \cite[Theorem~4.3]{Eisenbud} there is a monic polynomial $g\in R[t]$ which annihilates $\coker f$. This implies on the one hand that
\[
M/gM \longto N/gN \longto \coker f \longto 0
\]
is exact, and therefore $\coker f$ is an $R$-module of finite presentation, because $R[t]/(g)$ is a finite free $R$-module of rank $\deg_tg$. On the other hand it implies that $M[\tfrac{1}{g}] \onto N[\tfrac{1}{g}]$ is an epimorphism, whence an isomorphism by \cite[Corollary~8.12]{GoertzWedhorn}, because $M$ and $N$ are finite locally free over $R[t]$ of the same rank. Since $g$ is a non-zero divisor on $R[t]$ and thus also on $M$, the localization map $M\to M[\tfrac{1}{g}]$ is injective, and hence also $f$ is injective.

We obtain the exact sequence $0\to M\to N\to\coker f\to0$, which yields for every maximal ideal $\Fm\subset R$ with residue field $k=R/\Fm$ the exact sequence
\[
0\;\longto\;\Tor_1^R(k,\coker f)\;\longto\; M\otimes_Rk\;\longto\; N\otimes_Rk\;\longto\;(\coker f)\otimes_Rk\;\longto\;0\,.
\]
Again the $k[t]$-modules $M\otimes_Rk$ and $N\otimes_Rk$ are locally free of the same rank and $(\coker f)\otimes_Rk$ is a torsion $k[t]$-module, annihilated by $g$. Since $k[t]$ is a PID, this implies that $M\otimes_Rk\to N\otimes_Rk$ is injective and so $\Tor_1^R(k,\coker f)=(0)$. Since $\coker f$ is finitely presented, it is locally free of finite rank by Nakayama's Lemma; e.g.\ \cite[Exercise~6.2]{Eisenbud}.
\end{proof}

For the next proposition note that $\CJ$ is an invertible sheaf on $\Spec A_R$ as we remarked before Lemma~\ref{LemmaF_q[t]}. 

\begin{proposition}\label{PropAMotiveEffective}
\begin{enumerate}
\item \label{PropAMotiveEffective_A}
Let $(M,\tau_M)$ be an $A$-motive. Then there exist integers $e,e'\in\BZ$ such that $\CJ^e\cdot\tau_M(\sigma^*M)\subset M$ and $\CJ^{e'}\cdot\tau_M^{-1}(M)\subset\sigma^*M$. For any such $e,e'$ the induced $A_R$-homomorphism $\tau_M\colon\CJ^e\cdot\sigma^*M\to M$ is injective, and the quotient $M/\tau_M(\CJ^e\cdot\sigma^*M)$ is a locally free $R$-module of finite rank, which is annihilated by $\CJ^{e+e'}$. 
\item \label{PropAMotiveEffective_B}
An $A$-motive $(M,\tau_M)$ is an effective $A$-motive, if and only if $\tau_M(\sigma^*M)\subset M$.
\item \label{PropAMotiveEffective_C}
Let $(M,\tau_M)$ be an effective $A$-motive over $R$. Then $(M,\tau_M|_{\Spec A_R\setminus\Var(\CJ)})$ is an $A$-motive. Moreover, $\tau_M\colon\sigma^*M\to M$ is injective and $\coker\tau_M$ is a finite locally free $R$-module.
\item \label{PropAMotiveEffective_D}
Let $\ulM=(M,\tau_M)$ be an effective $A$-motive over a field $k$. Then $\ulM$ has dimension $\dim_k\coker\tau_M$.
\end{enumerate}
\end{proposition}

\begin{proof}
\ref{PropAMotiveEffective_A} Working locally on affine subsets of $\Spec A_R$ we may assume that $\CJ$ is generated by a non-zero divisor $h\in\CJ$. By \cite[I, Th\'eor\`eme~1.4.1(d1)]{EGA} we obtain for every generator $m$ of the $A_R$-module $\sigma^*M$ an integer $n$ such that locally $h^n\cdot\tau_M(m)\in M$. Taking $e$ as the maximum of the $n$ when $m$ runs through a finite generating system of $\sigma^*M$, yields $\CJ^e\cdot\tau_M(\sigma^*M)\subset M$. The inclusion $\CJ^{e'}\cdot\tau_M^{-1}(M)\subset\sigma^*M$ is proved analogously.

Let $e$ and $e'$ be any integers with $\tau_M(\CJ^e\cdot\sigma^*M)\subset M$ and $\tau_M^{-1}(\CJ^{e'}\cdot M)\subset\sigma^*M$, whence $\CJ^{e+e'}\cdot M\subset\tau_M(\CJ^e\cdot\sigma^*M)$. %An element $m\in\sigma^*M$ whose restriction to $\Spec A_R\setminus\Var(\CJ)$ is zero, locally satisfies $h^n\cdot m=0$ for some $n$ by \cite[I, Th\'eor\`eme~1.4.1(d2)]{EGA}. Since $h$ is a non-zero divisor, this implies $m=0$. It follows that $\sigma^*M\to\Gamma(\Spec A_R\setminus\Var(\CJ),\sigma^*M)$ is injective, and hence $\tau_M\colon\CJ^e\cdot\sigma^*M\to M$ is injective.
%Let further $e'$ be an integer with $\tau_M^{-1}(\CJ^{e'}\cdot M)\subset\sigma^*M$, whence $\CJ^{e+e'}\cdot M\subset\tau_M(\CJ^e\cdot\sigma^*M)$. 
Then $M/\tau_M(\CJ^e\cdot\sigma^*M)$ is annihilated by $\CJ^{e+e'}$, and hence a finite module over $A_R/\CJ^{e+e'}$ and over $R$. Therefore $\tau_M\colon\CJ^e\cdot\sigma^*M\to M$ is injective, and the quotient $M/\tau_M(\CJ^e\cdot\sigma^*M)$ is a finite locally free $R$-module by Lemma~\ref{LemmaInjCokerLocFree}.

\medskip\noindent
\ref{PropAMotiveEffective_C} Since $\CJ^n\cdot\coker\tau_M=(0)$, the map $\tau_M|_{\Spec A_R\setminus\Var(\CJ)}$ is an epimorphism between locally free sheaves of the same rank, and hence an isomorphism by \cite[Corollary~8.12]{GoertzWedhorn}. Thus $\ulM$ is an $A$-motive and the remaining assertions follow from \ref{PropAMotiveEffective_A}. Also \ref{PropAMotiveEffective_B} follows directly.

\medskip\noindent
\ref{PropAMotiveEffective_D} Set $d:=\dim_k\coker\tau_M$. Since every $h\in\CJ$ which generates $\CJ$ locally on $\Spec A_k$ is nilpotent on the $k$-vector space $\coker\tau_M$, it satisfies $h^d=0$ by the Cayley-Hamilton theorem from linear algebra. We conclude that $\CJ^d\cdot\coker\tau_M=(0)$ and $\ulM$ has dimension $d$.
\end{proof}

\begin{proposition}\label{PropNoetherian}
\begin{enumerate}
\item \label{PropNoetherian_A}
If $S$ is an $R$-algebra, then $\ulM=(M,\tau_M)\longmapsto\ulM\otimes_RS:=(M\otimes_RS,\tau_M\otimes\id_S)$ defines a functor from (effective) $A$-motives of rank $r$ (and dimension $d$) over $R$ to (effective) $A$-motives of rank $r$ (and dimension $d$) over $S$.
\item \label{PropNoetherian_B}
Every $A$-motive over $R$ and every morphism $f\in\Hom(\ulM,\ulN)$ between $A$-motives over $R$ can be defined over a subring $R'$ of $R$, which via $\gamma\colon A\to R'\subset R$ is a finitely generated $A$-algebra, hence noetherian.
\end{enumerate}
\end{proposition}

\begin{proof}
\ref{PropNoetherian_A} This is obvious. 

\medskip\noindent
\ref{PropNoetherian_B} Every $A$-motive $\ulM=(M,\tau_M)$ has a presentation of the form $A_R^{\oplus n_1}\xrightarrow{\;U\,} A_R^{\oplus n_0}\xrightarrow{\;\,\rho\;} M\longto 0$. Since $M$ is locally free over $A_R$, there is a section $s$ of the epimorphism $\rho$. It corresponds to an endomorphism $S$ of $A_R^{\oplus n_0}$ with $SU=0$ such that there is a map $W\colon A_R^{\oplus n_0}\to A_R^{\oplus n_1}$ with $S-\Id=UW$. The isomorphism $\tau_M$ lifts to diagram
\begin{equation}\label{EqNoetherian}
\xymatrix @C+1pc { (\sigma^*A_R^{\oplus n_1})|_{\Spec A_R\setminus\Var(\CJ)} \ar[r]^{\TS\sigma^*U} \ar[d]_{\TS T_1} &  (\sigma^*A_R^{\oplus n_0})|_{\Spec A_R\setminus\Var(\CJ)} \ar[r]^{\TS\sigma^*\rho} \ar[d]_{\TS T_0} & \sigma^*M|_{\Spec A_R\setminus\Var(\CJ)} \ar[r] \ar[d]_{\TS\tau_M} & 0\,\,\\
A_R^{\oplus n_1}|_{\Spec A_R\setminus\Var(\CJ)} \ar[r]^{\TS U} &  A_R^{\oplus n_0}|_{\Spec A_R\setminus\Var(\CJ)} \ar[r]^{\TS\rho} & M|_{\Spec A_R\setminus\Var(\CJ)} \ar[r] & 0 \,.
}
\end{equation}
Likewise $\tau_M^{-1}$ lifts to a similar diagram with vertical morphism $T'_0$ and $T'_1$. The equations $\tau_M\circ\tau_M^{-1}=\id$ and $\tau_M^{-1}\circ\tau_M=\id$ imply the existence of matrices $V$ and $V'$ in $A_R^{n_1\times n_0}|_{\Spec A_R\setminus\Var(\CJ)}$ with $T_0\circ T'_0-\Id= U\circ V$ and $T'_0\circ T_0-\Id= \sigma^*U\circ V'$. Let $R'\subset R$ be the $A$-algebra generated by the finitely many elements of $R$ which occur in the entries of the matrices $U$, $S$, $W$, $T_0$, $T_1$, $T'_0$, $T'_1$, $V$ and $V'$. Define $M'$ as the $A_{R'}$-module which is the cokernel of $U\in A_{R'}^{n_0\times n_1}$, and define $\tau_{M'}\colon\sigma^*M'|_{\Spec A_R\setminus\Var(\CJ)}\to M'|_{\Spec A_R\setminus\Var(\CJ)}$ and $\tau_{M'}^{-1}\colon M'|_{\Spec A_R\setminus\Var(\CJ)}\to \sigma^*M'|_{\Spec A_R\setminus\Var(\CJ)}$ as the $A_{R'}$-homomorphisms given by diagram~\eqref{EqNoetherian} and its analog for $\tau_M^{-1}$. Then $M'$ is via $S$ a direct summand of $A_{R'}^{\oplus n_0}$, hence a finite locally free $A_{R'}$-module, and $\tau_{M'}$ and $\tau_{M'}^{-1}$ are inverse to each other. It follows from diagram~\eqref{EqNoetherian} that $M'\otimes_{R'}R=M$ and $\tau_{M'}\otimes\id_R=\tau_M$. 

Finally, the assertion for the morphism $f\in\Hom_R(\ulM,\ulN)$ follows from a diagram similar to \eqref{EqNoetherian} for $f$ instead of $\tau_M$.
\end{proof}

We end this section with the following observation.

\begin{proposition}\label{PropBaseChangeInj}
Let $\ulM$ and $\ulN$ be $A$-motives over $R$ and let $f\in\Hom_R(\ulM,\ulN)$ be a morphism. Then the set $X$ of points $s\in\Spec R$ such that $f\otimes\id_{\kappa(s)}=0$ in $\Hom_{\kappa(s)}\bigl(\ulM\otimes_R\kappa(s),\ulN\otimes_R\kappa(s)\bigr)$ is open and closed, but possibly empty. Let $\Spec\wt R\subset\Spec R$ be this set, then $f\otimes\id_{\wt R}=0$ in $\Hom_{\wt R}\bigl(\ulM\otimes_R\wt R,\ulN\otimes_R\wt R\bigr)$. In particular if $\Spec R$ is connected and $S\ne(0)$ is an $R$-algebra, then the map $\Hom_R(\ulM,\ulN)\to\Hom_S(\ulM\otimes_RS,\ulN\otimes_RS),\,f\mapsto f\otimes\id_S$ is injective. %The same is true for abelian Anderson $A$-modules.
\end{proposition}

\begin{proof}
We fix an element $t\in A\setminus\BF_q$. Then $A$ is a finite free $\BF_q[t]$-module. By Proposition~\ref{PropAMotiveEffective} we can find integers $e,e'$ with $\CJ^e\cdot\tau_N(\sigma^*N)\subset N$ and $\CJ^{e'}\cdot\tau_M^{-1}(M)\subset\sigma^*M$, such that $d:=e+e'$ is a power of $q$. We obtain morphisms $(t-\gamma(t))^e\tau_N\colon \sigma^*N\to N$ and $(t-\gamma(t))^{e'}\tau_M^{-1}\colon M\to \sigma^*M$. So the equation $f\circ\tau_M=\tau_N\circ\sigma^*f$ implies $(t^d-\gamma(t)^d) f=(t-\gamma(t))^e\tau_N\circ\sigma^*f\circ (t-\gamma(t))^{e'}\tau_M^{-1}$. We view $M$ and $N$ as modules over $R[t]$ and replace $A_R$ by $R[t]$. Since $M$ and $N$ are finite projective $R[t]$-modules there are split epimorphisms $R[t]^{\oplus n'}\onto M$ and $R[t]^{\oplus n}\onto N$. Then $R[t]^{\oplus n'}\onto M\xrightarrow{\;f\,}N\into R[t]^{\oplus n}$ is given by a matrix $F\in R[t]^{n\times n'}$ whose entries are polynomials in $t$. Let $I\subset R$ be the ideal generated by the coefficients of all these polynomials. A prime ideal $\Fp\subset R$ belongs to the set $X$ if and only if $I\subset\Fp$. In particular $X=\Var(I)\subset\Spec R$ is closed.

On the other hand, we consider the map $R[t]^{\oplus n}\onto \sigma^*N\xrightarrow{\,(t-\gamma(t))^e\tau_N}N\into R[t]^{\oplus n}$ as a matrix $T\in R[t]^{n\times n}$ and the map $R[t]^{\oplus n'}\onto M\xrightarrow{\,(t-\gamma(t))^{e'}\tau_M^{-1}}\sigma^*M\into R[t]^{\oplus n'}$ as a matrix $V\in R[t]^{n'\times n'}$. The formula $(t^d-\gamma(t)^d) f=(t-\gamma(t))^e\tau_N\circ\sigma^*f\circ (t-\gamma(t))^{e'}\tau_M^{-1}$ implies $(t^d-\gamma(t)^d) F=T\,\sigma(F)\,V$, and it follows that the entries of the matrix $(t^d-\gamma(t)^d) F$ are polynomials in $t$ whose coefficients lie in $I^q$. If $\sum_{i=0}^\ell b_it^i$ is an entry of $F$ then $(t^d-\gamma(t)^d)\sum_{i=0}^\ell b_it^i=\sum_{i=0}^{\ell+d}(b_{i-d}-\gamma(t)^db_i)t^i$ is the corresponding entry of $(t^d-\gamma(t)^d) F$ and all $b_{i-d}-\gamma(t)^db_i\in I^q$. By descending induction on $i=\ell+d,\ldots,0$ we see that all $b_i\in I^q$. It follows that the finitely generated ideal $I\subset R$ satisfies $I=I^q$. By Nakayama's lemma \cite[Corollary~4.7]{Eisenbud} there is an element $b\in 1+I$ such that $b\!\cdot\!I=(0)$. Now let $\Fp\subset R$ be a prime ideal which lies in $X$, that is $I\subset\Fp$. Then $\Fp$ lies in the open subset $\Spec R[\tfrac{1}{b}]\subset\Spec R$ on which $F=0$ and hence $f=0$. In particular $X\subset\Spec R[\tfrac{1}{b}]\subset X$. Therefore $X$ is open and closed and $f=0$ on $X$.

Now let $\Spec R$ be connected and $S\ne(0)$ be an $R$-algebra. Let $f\in\Hom_R(\ulM,\ulN)$ be such that $f\otimes\id_S=0$ in $\Hom_S(\ulM\otimes_RS,\ulN\otimes_RS)$. Let $s\in\Spec S$ be a point and let $s'\in\Spec R$ be its image. Then $f\otimes\id_{\kappa(s')}=0$ and the set $X$ from above is non-empty. Since it is open and closed and $\Spec R$ is connected, it follows that $X=\Spec R$ and $f=0$. This proves the injectivity. %Finally the statement about abelian Anderson $A$-modules follows from Theorem~\ref{ThmMotiveOfAModule}.
\end{proof}

\begin{corollary}\label{CorHomFinGen}
Let $\ulM$ and $\ulN$ be $A$-motives over $R$ with $\Spec R$ connected. Then $\Hom_R(\ulM,\ulN)$ is a finite projective $A$-module of rank less or equal to $(\rk\ulM)\!\cdot\!(\rk\ulN)$. %The same is true for abelian Anderson $A$-modules.
\end{corollary}

\begin{proof}
If $R=k$ is a field and $\ulM$ and $\ulN$ are effective, the result is due to Anderson \cite[Corollary~1.7.2]{Anderson86}. For general $R$ we apply Proposition~\ref{PropBaseChangeInj} with $S=R/\Fm$ for $\Fm\subset R$ a maximal ideal, and use that over the Dedekind ring $A$ every submodule of a finite projective module is itself finite projective. 
\end{proof}

%%%%%%%%%%%%%%%%%%%%%%%%%%%%%%%%%%%%%%%%%%%%%%%%%%%%%%%%%%%%%%%%%%%%%%
%
%    Anderson modules
%
%%%%%%%%%%%%%%%%%%%%%%%%%%%%%%%%%%%%%%%%%%%%%%%%%%%%%%%%%%%%%%%%%%%%%%

\section{Abelian Anderson $A$-modules}\label{SectAndModules}
\setcounter{equation}{0}

We recall Definition~\ref{DefAndModule} of \emph{abelian Anderson $A$-modules} from the introduction. Let us give some explanations. All group schemes in this article are assumed to be commutative. 

\begin{definition}
Let $\CO$ be a commutative unitary ring. An \emph{$\CO$-module scheme} over $R$ is a commutative group scheme $E$ over $R$ together with a ring homomorphism $\CO\to\End_R(E)$.
\end{definition}

For a group scheme $E$ over $\Spec R$ we let $E^n:=E\times_R\ldots\times_R E$ be the $n$-fold fiber product over $R$. We denote by $e\colon\Spec R\to E$ its zero section and by $\Lie E:=\Hom_R(e^*\Omega^1_{E/R},R)$ the tangent space of $E$ along $e$. If $E$ is smooth over $\Spec R$, then $\Lie E$ is a locally free $R$-module of rank equal to the relative dimension of $E$ over $R$. In particular $\Lie E^n=(\Lie E)^{\oplus n}$. For a homomorphism $f\colon E\to E'$ of group schemes over $\Spec R$ we denote by $\Lie f\colon \Lie E\to \Lie E'$ the induced morphism of $R$-modules. Also we define the \emph{kernel of $f$} as the $R$-group scheme $\ker f:=E\underset{f,\,E',e'}{\times}\Spec R$ where $e'\colon\Spec R\to E'$ is the zero section. There is a canonical isomorphism
\begin{equation}\label{EqIsomKernel}
E\underset{f,\,E',f}{\times}E\;\isoto\; E\underset{R}{\times}\ker f
\end{equation}
given on $T$-valued points $P,Q\in E(T)$ for any $R$-scheme $T$ by $(P,Q)\mapsto(P,Q-P)$. If $P\in E(k)$ for a field $k$ and $P'=f(P)\in E'(k)$, pulling back \eqref{EqIsomKernel} under $P\colon\Spec k\to E$ yields an isomorphism of the fiber $\Spec k\underset{P',\,E',f}{\times}E$ of $f$ over $P'$ with $\Spec k\times_R\ker f$.

On $\BG_{a,R}=\Spec R[x]$ the elements $b\in R$, and in particular $\gamma(a)\in R$ for $a\in\BF_q$, act via $b^*\colon R[x]\to R[x],\,x\mapsto bx$. This makes $\BG_{a,R}$ into an $\BF_q$-module scheme. In addition let $\tau:=\Frob_{q,\BG_{a,R}}$ be the relative $q$-Frobenius endomorphism of $\BG_{a,R}=\Spec R[x]$ given by $x\mapsto x^q$. It satisfies $\Lie \tau=0$ and $\tau\circ b=b^q\circ\tau$. We let 
\begin{equation}\label{EqRTau}
R\{\tau\}\;:=\;\bigl\{\,\TS\sum\limits_{i=0}^nb_i\tau^i\colon n\in\BN_0, b_i\in R\,\bigr\}\quad\text{with}\quad\tau b\;=\;b^q\tau
\end{equation}
be the non-commutative polynomial ring in $\tau$ over $R$. For an element $f=\sum_ib_i\tau^i\in R\{\tau\}$ we set $f(x):=\sum_ib_i x^{q^i}$.

\begin{lemma}\label{LemmaEndGa}
There is an isomorphism of $R$-modules $R\{\tau\}^{d'\times d}\isoto\Hom_{R\text{\rm-groups},\BF_q\text{\rm-lin}}(\BG_{a,R}^d,\BG_{a,R}^{d'})$, which sends the matrix $F=(f_{ij})_{i,j}$ to the $\BF_q$-equivariant morphism $f\colon\BG_{a,R}^d\to\BG_{a,R}^{d'}$ of group schemes over $R$ with $f^*(y_i)=\sum_j f_{ij}(x_j)$ where $\BG_{a,R}^d=\Spec R[x_1,\ldots,x_d]$ and $\BG_{a,R}^{d'}=\Spec R[y_1,\ldots,y_{d'}]$. Under this isomorphism the map $f\mapsto \Lie f$ is given by the map $R\{\tau\}^{d'\times d}\to R^{d'\times d},\,F=\sum_n F_n\tau^n\mapsto F_0$.
\end{lemma}

\begin{proof}
This is straight forward to prove using Lucas's theorem \cite{Lucas1878} on congruences of binomial coefficients which states that $\left(\begin{smallmatrix}pn+t\\pm+s\end{smallmatrix}\right)\equiv\left(\begin{smallmatrix}n\\m\end{smallmatrix}\right)\left(\begin{smallmatrix}t\\s\end{smallmatrix}\right)\mod p$ for all $n,m,t,s\in\BN_0$, and implies that $\left(\begin{smallmatrix}n\\i\end{smallmatrix}\right)\equiv 0\mod p$ for all $0<i<n$ if and only if $n=p^e$ for an $e\in\BN_0$.
\end{proof}

\begin{remark}\label{RemMotiveOfAModule}
The affine group scheme $E$ and its multiplication map $\Delta\colon E\times_R E\to E$ are described by its coordinate ring $B_E:=\Gamma(E,\CO_E)$ together with the comultiplication $\Delta^*\colon B_E\to B_E\otimes_R B_E$. If we write $\BG_{a,R}=\Spec R[\xi]$ the map
\begin{eqnarray*}
M(\ulE) & \isoto & \bigl\{\,x\in B_E\colon\Delta^*x=x\otimes 1+1\otimes x\text{ and }\phi_a^*x=\gamma(a)x\text{ for all }a\in\BF_q\,\bigr\}\\[1mm]
m\quad & \longmapsto & \quad m^*(\xi)
\end{eqnarray*}
is an isomorphism of $A_R$-modules. Choosing an element $\lambda\in\BF_q$ with $\BF_q=\BF_p(\lambda)$ we obtain an exact sequence of $R$-modules
\begin{equation}\label{EqM(E)}
\xymatrix @R=0pc {
0 \ar[r] & M(\ulE) \ar[r] & B_E \ar[r] & B_E\otimes_R B_E \es\oplus\es B_E\\
& m\ar@{|->}[r] & m^*(\xi)\,,\quad x \ar@{|->}[r] & \bigl(\Delta^*x-x\otimes 1-1\otimes x\,,\es\phi_\lambda^*x-\gamma(\lambda)x\bigr)
}
\end{equation}
This shows that for every flat $R$-algebra $R'$ we have $M(\ulE)\otimes_RR'=M(\ulE\times_R\Spec R')$, because $\Gamma(E\times_RR',\CO_{E\times R'})=B_E\otimes_R R'$. In particular, if $R'$ satisfies condition \ref{DefAndModule_A} of Definition~\ref{DefAndModule} then $M(\ulE)\otimes_RR'\cong R'\{\tau\}^{1\times d}$ by Lemma~\ref{LemmaEndGa}.

From this we see that for any $R$-algebra $S$ the tensor product of the sequence \eqref{EqM(E)} with $S$ stays exact and $M(\ulE)\otimes_RS=M(\ulE\times_{\Spec R}\Spec S)$. Namely, we choose a faithfully flat morphism $R\to R'$ as in Definition~\ref{DefAndModule}\ref{DefAndModule_A} and tensor \eqref{EqM(E)} with $S\otimes_RR'$. This tensor product stays exact by Lemma~\ref{LemmaEndGa} because $M(\ulE)\otimes_RR'\cong R'\{\tau\}^{1\times d}$. Since $S\to S\otimes_RR'$ is faithfully flat, already the tensor product of \eqref{EqM(E)} with $S$ was exact.
\end{remark}

\begin{definition}\label{DefMotiveOfAModule}
If $\ulE$ is an abelian Anderson $A$-module we consider in addition on $M(\ulE)$ the map $\tau\colon m\mapsto\Frob_{q,\BG_{a,R}}\!\circ\, m$. Since $\tau(bm)=b^q\tau(m)$ the map $\tau$ is $\sigma$-semilinear and induces an $A_R$-linear map $\tau_M\colon\sigma^*M\to M$. We set $\ulM(\ulE):=\bigl(M(\ulE),\tau_M)$ and call it the \emph{(effective) $A$-motive associated with $\ulE$}.
\end{definition}

This definition is justified by the following relative version of Anderson's theorem~\cite[Theorem~1]{Anderson86}.

\begin{theorem}\label{ThmMotiveOfAModule}
If $\ulE=(E,\phi)$ is an abelian Anderson $A$-module of rank $r$ and dimension $d$ then $\ulM(\ulE)=(M,\tau_M)$ is an effective $A$-motive of rank $r$ and dimension $d$. There is a canonical isomorphism of $R$-modules
\begin{equation}\label{EqT0}
\coker\tau_M\;\isoto\;\Hom_R(\Lie E,R),\quad m\mod\tau_M(\sigma^*M)\;\longmapsto\; \Lie m\,.
\end{equation}
The contravariant functor $\ulE\mapsto\ulM(\ulE)$ is fully faithful. Its essential image consists of all effective $A$-motives $\ulM=(M,\tau_M)$ over $R$ of some dimension $d$, for which there exists a faithfully flat ring homomorphism $R\to R'$ such that $M\otimes_RR'$ is a finite free left $R'\{\tau\}$-module under the map $\tau\colon M\to M,\,m\mapsto\tau_M(\sigma_M^*m)$.
\end{theorem}

\begin{proof}
We first establish the isomorphism \eqref{EqT0}. If $m=\tau_M(\sum_i m_i\otimes b_i)=\sum_i b_i\circ\Frob_{q,\BG_{a,R}}\!\circ\, m_i$ with $m_i\in M$ and $b_i\in R$, then $\Lie m=0$ because $\Lie \Frob_{q,\BG_{a,R}}=0$. So the map \eqref{EqT0} is well defined. To prove that it is an isomorphism one can apply a faithfully flat base change $R\to R'$, see \cite[\S\,0$_{\text{I}}$.6.6]{EGA}, such that $E\otimes_RR'\cong\BG_{a,R'}^d$ and $\Lie E\otimes_RR'\cong(R')^{\oplus d}$. Then $M\otimes_RR'\cong R'\{\tau\}^{1\times d}$ by Remark~\ref{RemMotiveOfAModule}, and the inverse map is given by the natural inclusion $(R')^{1\times d}\subset R'\{\tau\}^{1\times d},F_0\mapsto F_0\tau^0$.

As a consequence, $\coker\tau_M$ is a locally free $R$-module of rank equal to $d=\dim\ulE$ and annihilated by $\CJ^d$ because of condition \ref{DefAndModule_B} in Definition~\ref{DefAndModule}. This implies $\coker\tau_M|_{\Spec A_R\setminus\Var(\CJ)}=(0)$, and therefore the morphism $\tau_M\colon \sigma^*M|_{\Spec A_R\setminus\Var(\CJ)}\to M|_{\Spec A_R\setminus\Var(\CJ)}$ is surjective. By \cite[Corollary~8.12]{GoertzWedhorn} it is an isomorphism, because $M$ and $\sigma^*M$ are finite locally free over $A_R$ of the same rank. Thus $\ulM(\ulE)$ is an $A$-motive and even an effective $A$-motive of dimension $d$ by Proposition~\ref{PropAMotiveEffective}.
%To prove that $\tau_M$ is injective choose an element $t\in A\setminus\BF_q$ and consider the finite flat inclusions $\BF_q[t]\subset A$ and $R[t]\subset A_R$. Set $\theta:=\gamma(t)\in R$. Since $t-\theta\in\CJ$ and $\coker\tau_M$ is annihilated by $\CJ^d$ the map $\tau_M$ induces an epimorphism $\tau_M\colon\sigma^*M[\tfrac{1}{t-\theta}]\onto M[\tfrac{1}{t-\theta}]$ of finite locally free modules of the same rank over $A_R[\tfrac{1}{t-\theta}]$, which therefore is an isomorphism. Since $t-\theta$ is a non-zero divisor on $R[t]$ the localization map $R[t]\to R[t][\tfrac{1}{t-\theta}]$ is injective. Likewise $\sigma^*M\to \sigma^*M[\tfrac{1}{t-\theta}]$ is injective because $\sigma^*M$ is flat over $A_R$. This implies that $\tau_M\colon\sigma^*M\to M$ is injective and that $\ulM(\ulE)$ is an effective $A$-motive of dimension $d$.

\medskip
Let $\ulE=(E,\phi)$ and $\ulE'=(E',\phi')$ be two abelian Anderson $A$-modules over $R$ and let $\ulM=\ulM(\ulE)$ and $\ulM'=\ulM(\ulE')$ be the associated effective $A$-motives. To prove that the map 
\begin{equation}\label{EqHoms}
\Hom_R(\ulE,\ulE')\;\longto\;\Hom_R\bigl(\ulM,\ulM'\bigr),\quad f\;\longmapsto(m'\mapsto m'\circ f)
\end{equation}
is bijective, we again apply a faithfully flat base change $R\to R'$, such that $E\otimes_RR'\cong\BG_{a,R'}^d$ and $E'\otimes_RR'\cong\BG_{a,R'}^{d'}$. Then $\Hom_{R'}(\ulE\otimes_RR',\ulE'\otimes_RR')\cong\bigl\{F\in R'\{\tau\}^{d'\times d}\colon\phi'_a\circ F=F\circ\phi_a\;\forall\,a\in A\bigr\}$ by Lemma~\ref{LemmaEndGa}. Also $\ulM(\ulE)\otimes_RR'\cong R'\{\tau\}^{1\times d}$ and $\ulM(\ulE')\otimes_RR'\cong R'\{\tau\}^{1\times d'}$. The condition $h\circ\tau_{M'}=\tau_M\circ\sigma^*h$ on an element $h\in\Hom_{R'}\bigl(\ulM(\ulE')\otimes_RR',\ulM(\ulE)\otimes_RR'\bigr)$ implies that $h\colon R'\{\tau\}^{1\times d'}\to R'\{\tau\}^{1\times d}$ is a homomorphism of left $R'\{\tau\}$-modules, hence given by multiplication on the right by a matrix $H\in R'\{\tau\}^{d'\times d}$. Then $m'\circ\phi'_a\circ H=h\bigl((a\otimes1)\cdot m')=(a\otimes1)\cdot h(m')=m'\circ H\circ\phi_a$ implies $\phi'_a\circ H=H\circ\phi_a$ for all $a\in A$. It follows that the map \eqref{EqHoms} is bijective over $R'$. So every $h\in\Hom_R\bigl(\ulM(\ulE'),\ulM(\ulE)\bigr)$ gives rise to a morphism $f'\in\Hom_{R'}(\ulE\otimes_RR',\ulE'\otimes_RR')$ which carries a descent datum because $h$ was defined over $R$. Since by \cite[\S\,6.1, Theorem~6(a)]{BLR} the descent of morphisms relative to the faithfully flat morphism $R\to R'$ is effective, $f'$ descends to the desired $f\in\Hom_R(\ulE,\ulE')$. This shows that the functor $\ulE\mapsto\ulM(\ulE)$ is fully faithful.

\medskip
Let $\ulM=(M,\tau_M)$ be an effective $A$-motive of dimension $d$ over $R$ for which there exists a faithfully flat ring homomorphism $R\to R'$ such that $M\otimes_RR'\cong R'\{\tau\}^{1\times d}$. Observe that $\coker(\tau_M\otimes\id_{R'})\cong (R'\{\tau\}/R'\{\tau\}\tau)^{1\times d}=(R')^{1\times d}$. For all $a\in A$ we have $\tau\cdot(a\otimes1)m=\sigma(a\otimes1)\cdot\tau(m)=(a\otimes1)\tau m$. Therefore the map $m\mapsto(a\otimes1)m$ is a homomorphism of left $R'\{\tau\}$-modules, and hence given by $(a\otimes1)m=m\cdot\phi'_a$ for a matrix $\phi'_a\in R'\{\tau\}^{d\times d}$. Then $\ulE':=(E'=\BG_{a,R'}^d,\phi'\colon A\to R'\{\tau\}^{d\times d},a\mapsto\phi'_a)$ satisfies $\ulM(\ulE')=\ulM\otimes_RR'$. Again $\bigl(a\otimes1-1\otimes\gamma(a)\bigr)^d=0$ on $\coker\tau_M$ implies $\bigl(\Lie \phi'_a-\gamma(a)\bigr)^d=0$ on $\Lie E'$. So $\ulE'$ is an abelian Anderson $A$-module over $R'$ with $\ulM(E')\cong\ulM\otimes_RR'$. Consider the ring $R'':=R'\otimes_RR'$ and the two maps $p_1,p_2\colon R'\to R''$ given by $p_1(b')=b'\otimes1$ and $p_2(b')=1\otimes b'$. The canonical isomorphism $p_1^*(\ulM\otimes_RR')=p_2^*(\ulM\otimes_RR')$ induces an isomorphism $p_1^*\ulE'\cong p_2^*\ulE'$ which is a descend datum on $\ulE'$ relative to $R\to R'$. Since faithfully flat descend on affine schemes is effective by \cite[\S\,6.1, Theorem~6(b)]{BLR} there exists a group scheme $E$ over $R$ with a ring homomorphism $\phi\colon A\to\End_{R\text{-groups}}(E)$ such that $(E,\phi)\otimes_RR'\cong\ulE'$. By \cite[IV$_2$, Proposition~2.7.1 and IV$_4$, Corollaire 17.7.3]{EGA} the group scheme $E$ is affine and smooth over $R$ and hence $(E,\phi)$ is an abelian Anderson $A$-module with $\ulM(E,\phi)\cong\ulM$.
\end{proof}

The theorem implies the following

\begin{corollary}\label{CorAModuleHomFinGen}
The assertions of Proposition~\ref{PropBaseChangeInj} and Corollary~\ref{CorHomFinGen} also hold for abelian Anderson $A$-modules. \qed
\end{corollary}

\medskip

An important class of examples are Drinfeld modules. We recall their definition from \cite[\S\,5]{Drinfeld} and \cite[\S\,1]{Saidi}.

\begin{definition}\label{DefDrinfeldMod}
A \emph{Drinfeld $A$-module of rank $r\in\BN_{>0}$ over $R$} is a pair $\ulE=(E,\phi)$ consisting of a smooth affine group scheme $E$ over $\Spec R$ of relative dimension $1$ and a ring homomorphism $\phi\colon A\to\End_{R\text{-groups}}(E),\,a\mapsto\phi_a$ satisfying the following conditions:
\begin{enumerate}
\item \label{DefDrinfeldMod_A}
Zariski-locally on $\Spec R$ there is an isomorphism $\alpha\colon E\isoto\BG_{a,R}$ of $\BF_q$-module schemes such that
\item \label{DefDrinfeldMod_B}
the coefficients of $\Phi_a:=\alpha\circ\phi_a\circ\alpha^{-1}=\sum\limits_{i\ge0}b_i(a)\tau^i\in\End_{R\text{-groups},\BF_q\text{\rm-lin}}(\BG_{a,R})=R\{\tau\}$ satisfy $b_0(a)=\gamma(a)$, $b_{r(a)}(a)\in R\mal$ and $b_i(a)$ is nilpotent for all $i>r(a):=-r\,[\BF_\infty:\BF_q]\ord_\infty(a)$.
\end{enumerate}
If $b_i(a)=0$ for all $i>r(a)$ we say that $\ulE$ is \emph{in standard form}.
\end{definition}

It is well known that every Drinfeld $A$-module over $R$ can be put in standard form; see \cite[\S\,5]{Drinfeld} or \cite[\S\,4.2]{Matzat}. This is a consequence of the following lemma of Drinfeld~\cite[Propositions~5.1 and 5.2]{Drinfeld} which we will need again below. For the convenience of the reader we recall the proof.

\begin{lemma}\label{LemmaStandardForm}
\begin{enumerate}
\item \label{LemmaStandardForm_A}
Let $b=\sum_{i=0}^n b_i\tau^i\in R\{\tau\}$ and let $r$ be a positive integer such that $b_r\in R\mal$ and $b_i$ is nilpotent for all $i>r$. Then there is a unique unit $c=\sum_{i\ge0}c_i\tau^i\in R\{\tau\}\mal$ with $c_0=1$ and $c_i$ nilpotent for $i>0$, such that $c^{-1}bc=\sum_{i=0}^r b'_i\tau^i$ with $b'_r\in R\mal$.
\item \label{LemmaStandardForm_B}
Let $\Spec R$ be connected and let $b=\sum_{i=0}^m b_i\tau^i$ and $c=\sum_{i=0}^n c_i\tau^i\in R\{\tau\}$ with $m,n>0$ and $b_m,c_n\in R\mal$. Let $d\in R\{\tau\}\setminus\{0\}$ satisfy $db=cd$. Then $m=n$ and $d=\sum_{i=0}^r d_i\tau^i$ with $d_r\in R\mal$.
\end{enumerate}
\end{lemma}

\begin{proof}
\ref{LemmaStandardForm_A} was also reproved in \cite[Lemma~1.1.2]{Laumon} and \cite[Proposition~1.4]{Matzat}.

\smallskip\noindent
\ref{LemmaStandardForm_B} We write $d=\sum_{i=0}^r d_i\tau^i$ with $d_r\ne0$.
% We note that the set of prime ideals $\Fp\subset R$ with $d_r\in\Fp$ is closed in $\Spec R$. We also show that this set is open. We consider the finitely generated subring $\wt R=\BF_q(b_i,c_i,d_i)\subset R$ which is noetherian. 
The equation $db=cd$ implies $\sum_j(d_{i-j}b_j^{\,q^{i-j}}-c_j d_{i-j}^{\;q^j})=0$ for all $i$, where the sum runs over $j=\max\{0,i-r\},\ldots,\min\{i,\max\{m,n\}\}$. We now distinguish three cases.

If $m>n$ then $i=m+r$ yields $d_rb_m^{\,q^r}=0$, whence $d_r=0$ which is a contradiction. 

If $m<n$ then $i=n+r$ yields $c_nd_r^{\,q^n}=0$, whence $d_r\in\Fp$ for every prime ideal $\Fp\subset R$. For $n+r>i\ge n$ we obtain $c_nd_{i-n}^{\,q^n}=\sum\limits_{0\le j<n}(d_{i-j}b_j^{\,q^{i-j}}-c_j d_{i-j}^{\;q^j})$ and by descending induction on $i$ it follows that $d_{i-n}\in\Fp$ for every prime ideal $\Fp\subset R$ for all $i-n=r,\ldots,0$. So the ideal $I:=(d_i\colon 0\le i\le r)\subset R$ is contained in every prime ideal $\Fp\subset R$. Now $i=m+r$ yields $d_rb_m^{\,q^r}=\sum\limits_{j=m}^{m+r}c_j d_{m+r-j}^{\;q^j}$, whence $d_r\in I^q$. For $m+r>i\ge m$ we obtain $d_{i-m}b_m^{\,q^{i-m}}=\sum\limits_{0\le j<m} d_{i-j}b_j^{\,q^{i-j}}-\sum\limits_{0\le j\le n} c_j d_{i-j}^{\;q^j}$ and by descending induction on $i$ it follows that $d_{i-m}\in I^q$ for all $i-m=r,\ldots,0$. Therefore the finitely generated ideal $I$ satisfies $I=I^q$ and by Nakayama's lemma \cite[Corollary~4.7]{Eisenbud} there is an element $f\in 1+I$ such that $f\!\cdot\!I=(0)$. Since $I\subset\Fp$ for all prime ideals $\Fp\subset R$, the element $1-f$ is a unit in $R$ and $I=0$. Therefore $d_i=0$ for all $i$ which is a contradiction.

If $m=n$ then $c_m d_r^{q^m}=d_r b_m^{q^r}$ and we consider the ideal $I=(d_r)\subset R$. Again $I=I^{q^m}$ and by \cite[Corollary~4.7]{Eisenbud} there is an element $f\in 1+I$ such that $f\!\cdot\! d_r=0$. Now assume that $d_r\in\Fp$ for some prime ideal $\Fp\subset R$. Then $f\notin\Fp$, whence $\Fp\in\Spec R[\tfrac{1}{f}]\subset\Spec R$ and $d_r=0$ on the open neighborhood $\Spec R[\tfrac{1}{f}]$ of $\Fp$. Since the set of prime ideals $\Fp\subset R$ with $d_r\in\Fp$ is closed in $\Spec R$ and the latter is connected, it follows that $d_r=0$ on all of $\Spec R$. This is a contradiction and so our assumption was false. In particular $d_r$ is not contained in any prime ideal and so $d_r\in R\mal$ as desired.
\end{proof}

\begin{theorem}\label{ThmDrinfeldMod}
The abelian Anderson $A$-modules of dimension $1$ and rank $r$ over $R$ are precisely the Drinfeld $A$-modules of rank $r$ over $R$.
\end{theorem}

\begin{proof}
Let $\ulE$ be a Drinfeld $A$-module of rank $r$ over $R$. Choose a Zariski covering as in Definition~\ref{DefDrinfeldMod}\ref{DefDrinfeldMod_A} such that $\ulE$ is in standard form. Since $\Spec R$ is quasi-compact this Zariski covering can be refined to a covering by finitely many affines. Their disjoint union is of the form $\Spec R'$ and the ring homomorphism $R\to R'$ is faithfully flat. So $\ulE$ satisfies conditions \ref{DefAndModule_A} and \ref{DefAndModule_B} of Definition~\ref{DefAndModule}. Choose an element $t\in A\setminus\BF_q$. Then $A$ is a finite free $\BF_q[t]$-module of rank equal to $-[\BF_\infty:\BF_q]\ord_\infty(t)$ by Lemma~\ref{LemmaF_q[t]}. Writing $\Phi_t=\sum_{i=0}^{r(t)}b_i(t)\tau^i$ with $r(t)=-r\,[\BF_\infty:\BF_q]\ord_\infty(t)$ and $b_{r(t)}(t)\in(R')\mal$, we make the following 
\begin{equation}\label{EqClaim}
\text{{\it Claim.}\quad As an $R'[t]$-module $\DS M(\ulE)\otimes_R R'=\bigoplus\limits_{\ell=0}^{r(t)-1}R'[t]\!\cdot\!\tau^\ell$.\qquad\qquad\qquad\qquad\qquad\qquad\qquad\es}
\end{equation}
By Remark~\ref{RemMotiveOfAModule} and Lemma~\ref{LemmaEndGa} we have $M(\ulE)\otimes_R R'=M(\ulE\times_{\Spec R}\Spec R')=R'\{\tau\}$. We prove by induction on $n$ that for every $c=\sum_{i=0}^nc_i\tau^i\in R'\{\tau\}=M(\ulE)$ there are uniquely determined elements $f_\ell(t)\in R'[t]$ such that $c=\sum_{\ell=0}^{r(t)-1}f_\ell(t)\!\cdot\!\tau^\ell$. If $n<r(t)$ then we take $f_\ell(t)=c_\ell$. If $n\ge r(t)$, dividing $c$ by $\Phi_t$ on the right produces uniquely determined $g=\sum_{i=0}^{n-r(t)}g_i\tau^i$ and $h=\sum_{\ell=0}^{r(t)-1}h_\ell\tau^\ell\in R'\{\tau\}$ with $c=g\Phi_t+h$. Namely, starting with $g_i=0$ for $i>n-r(t)$ we can and must take $g_i=b_{r(t)}^{-q^i}\bigl(c_{i+r(t)}-\sum\limits_{j=i+1}^{i+r(t)}g_j\,b_{i+r(t)-j}^{q^j}\bigr)$ for $i=n-r(t),\ldots,0$ and $h_\ell=c_\ell-\sum\limits_{j=0}^\ell g_j\,b_{\ell-j}^{q^j}$ for $\ell=r(t)-1,\ldots,1$. The induction hypothesis implies $g=\sum\limits_{\ell=0}^{r(t)-1}\tilde f_\ell(t)\!\cdot\!\tau^\ell$. Now $f_\ell(t):=\tilde f_\ell(t)\!\cdot\!t+h_\ell$ satisfies $c=\sum_{\ell=0}^{r(t)-1}f_\ell(t)\!\cdot\!\tau^\ell$. This proves the claim.

By faithfully flat descent \cite[IV$_2$, Proposition~2.5.2]{EGA} with respect to $R[t]\to R'[t]$ and by the claim, $M(\ulE)$ is finite, locally free over $R[t]$ and in particular flat over $R$. We next show that it is finitely presented over $A_R$. Namely, let $(m_i)_{i\in I}$ be a finite generating system of $M(\ulE)$ over $R[t]$. Using it as a generating system over $A_R$ we obtain an epimorphism $\rho\colon A_R^I\onto M(\ulE)$. Since $A_R$ is a finite free $R[t]$-module, also $A_R^I$ is a finite free $R[t]$-module and so the kernel of $\rho$ is a finitely generated $R[t]$-module, whence a finitely generated $A_R$-module. This shows that $M(\ulE)$ is a finitely presented $A_R$-module. From \cite[IV$_3$, Th\'eor\`eme~11.3.10]{EGA} it follows that $M(\ulE)$ is finite locally free over $A_R$, because for every point $s\in\Spec R$ the finite $A_{\kappa(s)}$-module $M(\ulE)\otimes_R\kappa(s)$ is a free $\kappa(s)[t]$-module and hence a torsion free and flat $A_{\kappa(s)}$-module. Its rank is $r$ as can be computed by comparing the ranks of $A_{R'}$ and $M(\ulE)\otimes_R R'$ over $R'[t]$. This proves that $\ulE$ is an abelian Anderson $A$-module of dimension $1$ and rank $r$ over $R$.

\medskip
Conversely let $\ulE=(E,\phi)$ be an abelian Anderson $A$-module of dimension $1$ and rank $r$ over $R$. Let $R\to R'$ be a faithfully flat ring homomorphism and let $\alpha\colon E\times_R\Spec R'\isoto\BG_{a,R'}$ be an isomorphism of $\BF_q$-module schemes as in Definition~\ref{DefAndModule}\ref{DefAndModule_A}. For $a\in A$ write
\[
\Phi_a\;:=\;{\TS\sum\limits_{i=0}^{n(a)}}b_i(a)\tau^i\;:=\;\alpha\circ\phi_a\circ\alpha^{-1}\;\in\;\End_{R'\text{-groups},\BF_q\text{\rm-lin}}(\BG_{a,R'})\;=\;R'\{\tau\}\,,
\]
where $n(a)\in\BN_0$ and $b_i(a)\in R'$. For $a\in\BF_q$ we obtain $\Phi_a=\gamma(a)\!\cdot\!\tau^0$. For $t:=a\in A\setminus\BF_q$ we consider $A$ as a finite free $\BF_q[t]$-module of rank $-[\BF_\infty:\BF_q]\ord_\infty(a)$ by Lemma~\ref{LemmaF_q[t]}. Then $M(\ulE)$ is a finite locally free $R[t]$-module of rank $r(a):=-r\,[\BF_\infty:\BF_q]\ord_\infty(a)$ by condition \ref{DefAndModule_C} of Definition~\ref{DefAndModule}. Let $\Fp\subset R'$ be a prime ideal, set $k=\Quot(R'/\Fp)$, and consider the abelian Anderson $A$-module $\ulE\times_R\Spec k$ over $k$ and the free $k[t]$-module $M(\ulE)\otimes_Rk=M(\ulE\times_R\Spec k)$ of rank $r(a)$. By an argument similarly to our claim \eqref{EqClaim} we see that $\deg_\tau\bigl(\Phi_a\otimes_{R'}1_k\bigr)=r(a)$, that is $b_{r(a)}(a)\otimes 1_{k}\in k\mal$ and $b_i(a)\otimes 1_{k}=0$ for all $i>r(a)$. This implies that $b_{r(a)}(a)\in(R')\mal$ and $b_i(a)$ is nilpotent for all $i>r(a)$ by \cite[Corollary~2.12]{Eisenbud}. By Lemma~\ref{LemmaStandardForm}\ref{LemmaStandardForm_A} we may change the isomorphism $\alpha$ such that $\Phi_a=\sum_{i=0}^{r(a)}b_i(a)\tau^i$ with $b_{r(a)}(a)\in(R')\mal$ for one $a\in A$, and by Lemma~\ref{LemmaStandardForm}\ref{LemmaStandardForm_B} this then holds for all $a\in A$, because $\Phi_a\Phi_b=\Phi_{ab}=\Phi_b\Phi_a$. By condition \ref{DefAndModule_B} of Definition~\ref{DefAndModule} we have $b_0(a)=\gamma(a)$. Thus $\ulE\times_R\Spec R'$ is a Drinfeld $A$-module of rank $r$ over $R'$ in standard form.

It remains to show that we can replace the faithfully flat covering $\Spec R'\to\Spec R$ by a Zariski covering. For this purpose consider $R'':=R'\otimes_RR'$ and the two projections $pr_i\colon\Spec R''\to\Spec R'$ onto the $i$-th factor for $i=1,2$. Then $h:=\sum_{i\ge0}h_i\tau^i:= pr_2^*\alpha\circ pr_1^*\alpha^{-1}\in R''\{\tau\}\mal$ satisfies $h_0\in(R'')\mal$ and $h_i$ is nilpotent for all $i>0$; see \cite[Proposition~1.4]{Matzat}. By Lemma~\ref{LemmaStandardForm}\ref{LemmaStandardForm_B} the equation $pr_2^*\Phi_a\circ h=h\circ pr_1^*\Phi_a$ implies that $h_i=0$ for all $i>0$ and $h=h_0\in(R'')\mal\subset R''\{\tau\}\mal$. The cocycle $\ul h:=(\Spec R'\to \Spec R,\,h)$ defines an element in the \v{C}ech cohomology group $\CKoh^1_\fpqc(\Spec R,\BG_m)$. By Hilbert~90, see \cite[Proposition~III.4.9]{Milne} we have $\CKoh^1_\fpqc(\Spec R,\BG_m)=\CKoh^1_\Zar(\Spec R,\BG_m)$. This means that there is a Zariski covering $\Spec\wt R\to\Spec R$, where $\Spec\wt R=\coprod_i \Spec \wt R_i$ is a disjoint union of open affine subschemes $\Spec \wt R_i\subset\Spec R$, and a unit $\tilde h=(\tilde h_{ij})_{i,j}\in(\wt R\otimes_R\wt R)\mal=\prod_{i,j}(\wt R_i\otimes_R \wt R_j)\mal$, such that $(\Spec\wt R\to\Spec R,\,\tilde h)=\ul h$. Let $\wt E$ be the smooth affine group and $\BF_q$-module scheme over $\Spec R$ with $\beta_i\colon\wt E|_{\Spec\wt R_i}\isoto\BG_{a,\wt R_i}$ and $\beta_j=\tilde h_{ij}\circ\beta_i$ on $\Spec\wt R_i\otimes_R\wt R_j$. Then over $\Spec R'\otimes_R\wt R=\coprod_i\Spec R'\otimes_R\wt R_i$ we have an isomorphism $\tilde\alpha:=(\beta_i^{-1}\circ\alpha)_i\colon E\isoto\wt E$. Let $p_i\colon\Spec(R'\otimes_R\wt R)\otimes_R(R'\otimes_R\wt R)\to\Spec R'\otimes_R\wt R$ be the projection onto the $i$-th factor for $i=1,2$. Then $p_2^*\tilde\alpha\circ p_1^*\tilde\alpha^{-1}=(\tilde h_{ij}^{-1}h)_{i,j}=1$. This shows that $\tilde\alpha$ descends to an isomorphism $\tilde\alpha\colon E\isoto\wt E$ over $\Spec R$ by \cite[\S\,6.1, Theorem~6(a)]{BLR}. On $\Spec\wt R_i$, now $\beta_i\circ\tilde\alpha\colon E\isoto\BG_{a,\wt R_i}$ is an isomorphism of $\BF_q$-module schemes. Moreover $\wt\Phi_a:=\beta_i\tilde\alpha\circ\phi_a\circ\tilde\alpha^{-1}\beta_i^{-1}\in\wt R_i\{\tau\}$ satisfies $\wt\Phi_a\otimes 1_{R'}=\Phi_a\otimes 1_{\wt R_i}$ in $(R'\otimes_R\wt R_i)\{\tau\}\supset\wt R_i\{\tau\}$ and by what we proved for $\Phi_a$ above, this implies that $\ulE$ is a Drinfeld $A$-module of rank $r$ over $R$ which by $\wt R$ and $(\beta_i\circ\tilde\alpha)_i$ is put in standard form.
\end{proof}

%%%%%%%%%%%%%%%%%%%%%%%%%%%%%%%%%%%%%%%%%%%%%%%%%%%%%%%%%%%%%%%%%%%%%%
%
%    Finite Shtukas
%
%%%%%%%%%%%%%%%%%%%%%%%%%%%%%%%%%%%%%%%%%%%%%%%%%%%%%%%%%%%%%%%%%%%%%%

\section{Review of the finite shtuka equivalence}\label{SectFinShtukas}
\setcounter{equation}{0}

In preparation for our main results in Sections~\ref{SectIsogenies} and \ref{SectTorsionPts} we need to recall Drinfeld's functor \cite[\S\,2]{Drinfeld87} and the equivalence it defines between finite $\BF_q$-shtukas and finite locally free strict $\BF_q$-module schemes; see also \cite{Abrashkin}, \cite[\S\,1]{Taguchi95}, \cite[\S\,B.3]{Laumon} and \cite[\S\S\,3-5]{HartlSingh}. 

\begin{definition}\label{DefFiniteShtuka}
A \emph{finite $\BF_q$-shtuka} over $R$ is a pair $\ulV=(V,F_V)$ consisting of a finite locally free $R$-module $V$ on $R$ and an $R$-module homomorphism $F_V\colon\sigma^\ast V\to V$. A \emph{morphism} $f\colon(V,F_V)\to(V',F_{V'})$ of finite $\BF_q$-shtukas is an $R$-module homomorphism $f\colon V\to V'$ satisfying $f\circ F_V=F_{V'}\circ\sigma^*f$. 

We say that $F_V$ is \emph{nilpotent} if there is an integer $n$ such that $F_V^n:=F_V\circ\sigma^*F_V\circ\ldots\circ\sigma^{(n-1)*}F_V=0$. A finite $\BF_q$-shtuka over $R$ is called \emph{\'etale} if $F_V$ is an isomorphism. If $\ulV=(V,F_V)$ is \'etale, we define for any $R$-algebra $R'$ the \emph{$\tau$-invariants of $\ulV$ over $R'$} as the $\BF_q$-vector space
\begin{equation}\label{EqTauInv}
\ulV^\tau(R')\;:=\;\{\,v\otimes V\otimes_RR'\colon v=F_V(\sigma_V^*v)\,\}\,.
\end{equation}
\end{definition}

Recall that an $R$-group scheme $G=\Spec B$ is \emph{finite locally free} if $B$ is a finite locally free $R$-module. By \cite[I$_{\rm new}$, Proposition~6.2.10]{EGA} this is equivalent to $G$ being finite, flat and of finite presentation over $\Spec R$. Every finite locally free $R$-group scheme $G=\Spec B$ is a relative complete intersection by \cite[III.4.15]{SGA3}. This means that locally on $\Spec R$ we can choose a presentation $B = R[X_1,..., X_n]/I$ where the ideal $I$ is generated by a regular sequence; compare \cite[IV$_4$, Proposition~19.3.7]{EGA}. The zero section $e\colon\Spec R\to G$ defines an augmentation $e_B:=e^*\colon B\onto R$ of the $R$-algebra $B$. Set $I_B:=\ker e_B$. For the polynomial ring $R[\ulX]=R[X_1,\ldots,X_n]$ set $I_{R[\ulX]}=(X_1,\ldots,X_n)$ and $e_{R[\ulX]}\colon R[\ulX]\onto R,\,X_\nu\mapsto0$. Faltings~\cite{Faltings02} and Abrashkin~\cite{Abrashkin} consider the deformation $B^\flat:=R[\ulX]/(I\!\cdot\!I_{R[\ulX]})$ and the canonical epimorphism $B^\flat\onto B$. They remark that there is a unique morphism
\[
\Delta^\flat\colon B^\flat\;\longto\;(B\otimes_R B)^\flat\;:=\;R[\ulX\otimes 1,1 \otimes\ulX]/(I\otimes 1 + 1 \otimes I)(I_{R[\ulX]}\otimes 1 + 1 \otimes I_{R[\ulX]})
\]
lifting the comultiplication $\Delta\colon B\to B\otimes_R B$ and satisfying $(\id_{B^\flat}\otimes e_B^\flat) \circ \Delta^\flat = \id_{B^\flat} = (e_B^\flat\otimes\id_{B^\flat}) \circ \Delta^\flat $, where $e_B^\flat\colon B^\flat\onto R$ is the augmentation map; see \cite[\S\,1.2]{Abrashkin} or \cite[Remark after Definition~3.5]{HartlSingh}. It satisfies $\Delta^\flat(x)-x\otimes 1-1\otimes x\,\in\, I_{B^\flat}\otimes I_{B^\flat}$ for all $x\in I_{B^\flat}$. Set $\CG=(G,G^\flat):=(\Spec B,\Spec B^\flat)$. The \emph{co-Lie complex of $\CG$ over $\Spec $R} (that is, the fiber at the zero section of $G$ of the cotangent complex; see \cite[\S\,VII.3.1]{Illusie72}) is the complex of finite locally free $R$-modules of rank $n$
\begin{eqnarray}
\CoL{\CG/\Spec R}\colon\qquad 0 \; \longto \; (I/{I^2})\otimes_{B,\,e_B}R \;\xrightarrow{\es d\;}\; \Omega^1_ {R[\ulX]/R} \otimes_{R[\ulX],\,e_{R[\ulX]}} R\; \longto \; 0
\end{eqnarray}
concentrated in degrees $-1$ and $0$ with $d$ being the differential map. Note that $(I/{I^2})\otimes_{B,\,e_B}R=\ker(B^\flat\onto B)$ and $\Omega^1_ {R[\ulX]/R} \otimes_{R[\ulX],\,e_{R[\ulX]}} R=\ker(e_B^\flat)/\ker(e_B^\flat)^2$ can be computed from $(B,B^\flat)$. Up to homotopy equivalence it only depends on $G$ and not on the presentation $B=R[\ulX]/I$. The \emph{co-Lie module of $G$ over $R$} is defined as $\omega_G:=\Koh^0(\CoL{\CG/\Spec R}):=\coker d$. We can now recall the definition of strict $\BF_q$-module schemes from Faltings~\cite{Faltings02} and Abrashkin~\cite{Abrashkin}; see also \cite[\S\,4]{HartlSingh}.

\begin{definition}\label{DefStrict}
Let $(G,[\,.\,])$ be a pair, where $G=\Spec B$ is an affine flat commutative group scheme over $R$ which is a relative complete intersection and where $[\,.\,]\colon\BF_q\to\End_{R\text{-groups}}(G),\,a\mapsto[a]$ is a ring homomorphism. Then $(G,[\,.\,])$ is called a \emph{strict $\BF_q$-module scheme} if there exists a presentation $B=R[\ulX]/I$ and a lift $[\,.\,]^\flat\colon\BF_q\to\End_{R\text{-algebras}}(B^\flat),\,a\mapsto[a]^\flat$ of the $\BF_q$-action on $G$, such that the induced action on $\CoL{\CG/\Spec R}$ is equal to the scalar multiplication via $\gamma\colon\BF_q\to R$, and such that $[1]^\flat=\id_{B^\flat}$ and $[0]^\flat=e_B^\flat$, as well as $[a\tilde a]^\flat=[a]^\flat\circ[\tilde a]^\flat$ and $[a+\tilde a]^\flat=m\circ([a]^\flat\otimes[\tilde a]^\flat)\circ\Delta^\flat$, where $m\colon(B\otimes_RB)^\flat\to B^\flat$ is induced by the multiplication map $B^\flat\otimes_RB^\flat\to B^\flat$ in the ring $B^\flat$ and the homomorphism $[a]^\flat\otimes[\tilde a]^\flat\colon B^\flat\otimes_RB^\flat\to B^\flat\otimes_RB^\flat$ induces a homomorphism $(B\otimes_R B)^\flat\to(B\otimes_R B)^\flat$ denoted again by $[a]^\flat\otimes[\tilde a]^\flat$. If $G$ is finite locally free, such a lift $a\mapsto[a]^\flat$ then exists for every presentation and is uniquely determined by \cite[Lemmas~4.4 and 4.7]{HartlSingh}.
\end{definition}

\begin{example}\label{ExampleStrictMorphism}
The group scheme $\BG_{a,R}^d$ is a strict $\BF_q$-module scheme for any $d$, because we can choose $B=R[X_1,\ldots,X_d]$ and so $I=(0)$ and $B^\flat=B$, and $a\in\BF_q$ acts as $[a]^*X_i=a\cdot X_i$. Moreover, every $\BF_q$-linear group homomorphism $\BG_{a,R}^d\to\BG_{a,R}^{d'}$ is strict in the sense of \cite[Definition~1]{Faltings02}, meaning that the homomorphism lifts to a homomorphism between the $B^\flat$ which is equivariant for the $\BF_q$-action via $[\,.\,]^\flat$.
\end{example}

\begin{lemma}\label{LemmaStrict}
Let $G$ be a finite locally free group scheme over $R$, let $\BF_q\to\End_{R\text{\rm-groups}}(G)$ be a ring homomorphism, and let $R\to R'$ be a faithfully flat ring homomorphism. Then $G$ is a strict $\BF_q$-module scheme if and only if $G\times_RR'$ is.
\end{lemma}

\begin{proof}
Let $pr\colon\Spec R'\to\Spec R$ be the induced morphism and let $pr_i\colon\Spec R'\otimes_RR'\to\Spec R'$ be the projection onto the $i$-th factor. Let $G=\Spec B$, let $R'[\ulX]\onto B\otimes_RR'$ be a presentation, and let $\BF_q\to\End_{R\text{-algebras}}\bigl((B\otimes_RR')^\flat\bigr)$, $a\mapsto[a]^\flat$ be a lift of the $\BF_q$-action on $G$ as in Definition~\ref{DefStrict}, which makes $G\times_RR'$ into a strict $\BF_q$-module scheme over $R'$. Moreover, let $f\colon R[\ulY]\onto B$ be an arbitrary presentation and let $\wt\CG=\bigl(\Spec B,\,\Spec R[\ulY]/(\ulY)\!\cdot\!\ker(f)\bigr)$ be the corresponding deformation. By \cite[Lemmas~4.4 and 4.7]{HartlSingh} there exists a unique lift $a\mapsto\wt{[a]}{}^\flat$ on the deformation $\wt\CG\times_RR'=pr^*\wt\CG$. By the uniqueness the two lifts $pr_1^*\wt{[a]}{}^\flat$ and $pr_2^*\wt{[a]}{}^\flat$ on the deformation $pr_1^*\,pr^*\wt\CG=pr_2^*\,pr^*\wt\CG$ coincide. By faithfully flat descent \cite[\S\,6.1, Theorem~6]{BLR} this lift descends to a lift on the deformation $\wt\CG$, which makes $G$ into a strict $\BF_q$-module scheme over $R$.
\end{proof}

To explain the equivalence between finite $\BF_q$-shtukas and finite locally free strict $\BF_q$-module schemes over $R$ we recall Drinfeld's functor.

\begin{definition}\label{DefDr_q}
Let $\ulV=(V,F_V)$ be a pair consisting of a (not necessarily finite locally free) $R$-module $V$ and a morphism $F_V\colon\sigma^*V\to V$ of $R$-modules.  Following Drinfeld \cite[\S\,2]{Drinfeld87} we define
\[
\Dr_q(\ulV)\es :=\es \Spec\;\bigl({\TS\bigoplus\limits_{n\ge0}}\Sym^n_{R} V\bigr)/I
\]
where the ideal $I$ is generated by the elements $v^{\otimes q}-F_V(\sigma_V^*v)$ for all $v\in V$. (Here $v^{\otimes q}$ lives in $\Sym^q V$ and $F_V(\sigma_V^*v)$ in $\Sym^1 V$.) Then $\Dr_q(\ulV)$ is a group scheme over $R$ via the comultiplication $\Delta\colon v\mapsto v\otimes1+1\otimes v$ and an $\BF_q$-module scheme via $[a]\colon v\mapsto av$ for $a\in\BF_q$. It has a canonical deformation
\[
\Dr_q(\ulV)^\flat\es :=\es \Spec\;\bigl({\TS\bigoplus\limits_{n\ge0}}\Sym^n_{R} V\bigr)/(I\cdot I_0),
\]
where $I_0=\bigoplus_{n\ge1}\,\Sym^n_{R} V$ is the ideal generated by the $v\in V$. This deformation is equipped with the comultiplication $\Delta^\flat\colon v\mapsto v\otimes1+1\otimes v$ and the $\BF_q$-action $[a]^\flat\colon v\mapsto av$. We set $\CDr_q(\ulV):=(\Dr_q(\ulV),\Dr_q(\ulV)^\flat)$. On its co-Lie complex $[a]$ acts by scalar multiplication with $a$ because $(av)^{\otimes q}-F_V(\sigma_V^*(av))=a^q(v^{\otimes q}-F_V(\sigma_V^*v))$. Therefore $\Dr_q(\ulV)$ is a finite locally free strict $\BF_q$-module scheme if $V$ is a finite locally free $R$-module. Every morphism $(V,F_V)\to(W,F_W)$, that is, every $R$-homomorphism $f\colon V\to W$ with $f\circ F_V=F_W\circ\sigma^*f$, induces a morphism $\Dr_q(f)\colon\Dr_q(W,F_W)\to\Dr_q(V,F_V)$. So $\Dr_q$ is a contravariant functor. If $f$ is surjective then $\Dr_q(f)$ is a closed immersion.
\end{definition}

Conversely, with a (not necessarily finite locally free) $\BF_q$-module scheme $G$ over $R$ we associate the pair $\ulM_q(G):=\bigl(M_q(G),F_{M_q(G)}\bigr)$ consisting of the $R$-module
\[
M_q(G) \es :=\es \Hom_{R\text{-groups},\BF_q\text{\rm-lin}}(G,\BG_{a,R})
\]
and the $R$-homomorphism $F_{M_q(G)}\colon \sigma^\ast M_q(G)\to M_q(G)$ which is induced from $M_q(G)\to M_q(G)$, $m\mapsto\Frob_{q,\BG_{a,R}}\!\circ\, m$. Every morphism of $\BF_q$-module schemes $f\colon G\to G'$ induces an $R$-homomorphism $\ulM_q(G')\to\ulM_q(G),\,m'\mapsto m'\circ f$. Note that by an argument as in Remark~\ref{RemMotiveOfAModule} we have $\ulM_q(G)\otimes_RS=\ulM_q(G\times_{\Spec R}\Spec S)$ for every $R$-algebra $S$.

There is a natural morphism $\ulV\to\ulM_q(\Dr_q(\ulV)),\,v\mapsto f_v$, where $f_v\colon\Dr_q(\ulV)\to\BG_{a,R}=\Spec R[\xi]$ is given by $f_v^*(\xi)=v$. There is also a natural morphism of group schemes $G\to\Dr_q(\ulM_q(G))$ given by $\bigoplus\limits_{n\ge0}\Sym^n_R M_q(G)/I\to\Gamma(G,\CO_G),\,m\mapsto m^*(\xi)$, which is well defined because $F_{M_q(G)}(\sigma^\ast m)^*(\xi)=(\Frob_{q,\BG_{a,R}}\!\circ\, m)^*(\xi)=m^*(\xi^q)=(m^*(\xi))^q$. 

\begin{example}\label{ExVOfAModule}
For example if $\ulE=(E,\phi)$ is an abelian Anderson $A$-module of dimension $d$, then $\ulM_q(\ulE)=(M_q(\ulE),F_{M_q(\ulE)})$ was denoted $\ulM(\ulE)=(M(\ulE),\tau_{M(\ulE)})$ in Definition~\ref{DefAndModule}. There is a canonical isomorphism $\ulE\isoto\Dr_q(\ulM_q(\ulE))$ which is constructed as follows. We set $\BG_{a,R}=\Spec R[\xi]$ and consider for each $m\in M_q(\ulE)=\Hom_{R\text{-groups},\BF_q\text{\rm-lin}}(E,\BG_{a,R})$ the element $m^*(\xi)\in\Gamma(E,\CO_E)$. We claim that
\begin{equation}\label{EqVOfAMotive}
\bigl({\TS\bigoplus\limits_{n\ge0}}\Sym_R^n M_q(\ulE)\bigr)\,\big/\,\bigl(m^{\otimes q}-F_{M_q(\ulE)}(\sigma_{M_q(\ulE)}^*m)\colon m\in M_q(\ulE)\bigr) \;\isoto\;\Gamma(E,\CO_E)\,,\quad m\mapsto m^*(\xi)
\end{equation}
is an isomorphism of $R$-algebras. To prove that it is an isomorphism we may apply a faithfully flat base change $R\to R'$ over which we have an $\BF_q$-linear isomorphism $\alpha\colon E\otimes_RR'\isoto\BG_{a,R'}^d=\Spec R'[x_1,\ldots,x_d]$. Let $m_i:=pr_i\circ\alpha\in M_q(\ulE)\otimes_RR'$ where $pr_i\colon\BG_{a,R'}^d\to\BG_{a,R'}$ is the projection onto the $i$-th factor. Then $M_q(\ulE)\otimes_RR'=\bigoplus_{i=0}^d R'\{\tau\}\!\cdot\!m_i$ by Remark~\ref{RemMotiveOfAModule} and the inverse of \eqref{EqVOfAMotive} sends $\alpha^*(x_i)$ to $m_i$. This is indeed the inverse, because \eqref{EqVOfAMotive} sends each of the generators $\tau^j m_i=\Frob_{q^j,\BG_{a,R}}\!\circ m_i$ of the $R'$-module $M_q(\ulE)\otimes_R R'$ to $(\Frob_{q^j,\BG_{a,R}}\!\circ m_i)^*(\xi)=m_i^*(\xi^{q^j})=\alpha^*(x_i)^{q^j}$, and this inverse sends it back to $m_i^{\otimes q^j}=\Frob_{q^j,\BG_{a,R}}\!\circ m_i=\tau^j m_i$.
\end{example}

The following theorem goes back to Abrashkin~\cite[Theorem~2]{Abrashkin}. Statements \ref{ThmEqAModSch_B}--\ref{ThmEqAModSch_F} were proved in \cite[Theorem~5.2]{HartlSingh}.

\begin{theorem} \label{ThmEqAModSch}
\begin{enumerate}
\item \label{ThmEqAModSch_A}
The contravariant functors $\Dr_q$ and $\ulM_q$ are mutually quasi-inverse anti-equiva\-len\-ces between the category of finite $\BF_q$-shtukas over $R$ and the category of finite locally free strict $\BF_q$-module schemes over $R$. Both functors are $\BF_q$-linear and exact.
\end{enumerate}
\smallskip\noindent 
Let $\ulV=(V,F_V)$ be a finite $\BF_q$-shtuka over $R$ and let $G=\Dr_q(\ulV)$. Then
\begin{enumerate}
\setcounter{enumi}{1}
\item \label{ThmEqAModSch_B}
the $\BF_q$-module scheme $\Dr_q(\ulV)$ is \'etale over $R$ if and only if $\ulV$ is \'etale.
%\item \label{ThmEqAModSch_C}
%the $\BF_q$-module scheme $\Dr_q(\ulV)$ is radicial over $R$ if and only if $F_V$ is nilpotent.
%\item \label{ThmEqAModSch_D}
%the order of the $R$-group scheme $\Dr_q(\ulV)$ is $q^{\rk V}$.
\item \label{ThmEqAModSch_E}
the natural morphisms $\ulV\to\ulM_q(\Dr_q(\ulV)),\,v\mapsto f_v$ and $G\to\Dr_q(\ulM_q(G))$ are isomorphisms.
\item \label{ThmEqAModSch_F}
the co-Lie complex $\CoL{\CDr_q(\ulV)/S}$ is canonically isomorphic to the complex \ $0\to\sigma^*V\xrightarrow{\;F_V}V\to0$.
\end{enumerate}
\end{theorem}

%%%%%%%%%%%%%%%%%%%%%%%%%%%%%%%%%%%%%%%%%%%%%%%%%%%%%%%%%%%%%%%%%%%%%%
%
%    Isogenies
%
%%%%%%%%%%%%%%%%%%%%%%%%%%%%%%%%%%%%%%%%%%%%%%%%%%%%%%%%%%%%%%%%%%%%%%

\section{Isogenies}\label{SectIsogenies}
\setcounter{equation}{0}

\begin{definition}\label{DefIsogAModule}
A morphism $f\in\Hom_R(\ulE,\ulE')$ between two abelian Anderson $A$-modules $\ulE$ and $\ulE'$ over $R$ is an \emph{isogeny} if $f\colon E\to E'$ is finite and surjective. If there exists an isogeny between $\ulE$ and $\ulE'$ then they are called \emph{isogenous}. (Being isogenous is an equivalence relation; see Corollary~\ref{CorIsogEquiv} below.)

An isogeny $f\colon\ulE\to\ulE'$ is \emph{separable} if $f$ is \'etale, or equivalently if the group scheme $\ker f$ is \'etale over $R$. Indeed, since $f$ is flat by Proposition~\ref{PropIsogAModule}\ref{PropIsogAModule_B} it suffices to see that all fibers of $f$ over $E'$ are \'etale  by \cite[\S\,2.4, Proposition~8]{BLR}. Now all fibers are isomorphic to $\ker f$ by the remarks after \eqref{EqIsomKernel}.
\end{definition}

We recall the following well known criterion for being an isogeny. For the convenience of the reader we include a proof.

\begin{proposition}\label{PropIsogAModule}
Let $f\colon E\to E'$ be a morphism between two affine, smooth $R$-group schemes $E$ of relative dimension $d$ and $E'$ of relative dimension $d'$, such that the fibers of $E'$ over all points of $\Spec R$ are connected. Then the following are equivalent:
\begin{enumerate}
\item \label{PropIsogAModule_A}
$f$ is finite and faithfully flat, that is flat and surjective; see \cite[0$_{\text{I}}$.6.7.8]{EGA},
\item \label{PropIsogAModule_B}
$\ker f$ is finite and $f$ is flat,
\item \label{PropIsogAModule_C}
$\ker f$ is finite and $f$ is surjective,
\item \label{PropIsogAModule_D}
$\ker f$ is finite and $d=d'$,
\item \label{PropIsogAModule_F}
$\ker f$ is finite and $f$ is an epimorphism of sheaves for the \fpqc-topology.
\end{enumerate}
If $R=k$ is a field, then these conditions are equivalent to
\begin{enumerate}
\setcounter{enumi}{5}
\item \label{PropIsogAModule_E}
$f$ is surjective and $d=d'$.
\end{enumerate}
\end{proposition}

\begin{proof}
We show that \ref{PropIsogAModule_A} implies all other conditions. This is obvious for \ref{PropIsogAModule_B}, \ref{PropIsogAModule_C} and \ref{PropIsogAModule_F}. To prove that $d=d'$ let $\Fm\subset R$ be a maximal ideal and consider the base change to $k=R/\Fm$. Then $f\times\id_k\colon E\times_Rk\to E'\times_Rk$ is a finite surjective morphism, and hence $d=\dim E\times_Rk=\dim E'\times_Rk=d'$; see \cite[Corollary~9.3]{Eisenbud}.

Conversely, clearly \ref{PropIsogAModule_F}$\Longrightarrow$\ref{PropIsogAModule_C}. We now show \ref{PropIsogAModule_E}$\Longrightarrow$\ref{PropIsogAModule_C} and \ref{PropIsogAModule_B}$\Longrightarrow$\ref{PropIsogAModule_C}$\Longrightarrow$\ref{PropIsogAModule_D}$\Longrightarrow$\ref{PropIsogAModule_B}$\Longrightarrow$\ref{PropIsogAModule_A}. Generally note that by the remarks after \eqref{EqIsomKernel} all non-empty fibers of $f$ are isomorphic to $\ker f$.

First assume \ref{PropIsogAModule_E} and note that when $R=k$ is a field, the ring $\Gamma(E',\CO_{E'})$ is an integral domain by our assumptions on $E'$. The surjectivity of $f$ implies that $f^*\colon\Gamma(E',\CO_{E'})\into\Gamma(E,\CO_E)$ is injective of relative transcendence degree $d-d'=0$. Since all fibers of $f$ are isomorphic to $\ker f$, \cite[Corollary~14.6]{Eisenbud} implies that $\ker f$ is finite over $\Spec k$ and \ref{PropIsogAModule_C} holds. 

We next show for general $R$ that \ref{PropIsogAModule_B} implies \ref{PropIsogAModule_C}. Namely, $f$ is of finite presentation by \cite[IV$_1$, Proposition~1.6.2(v)]{EGA}, because $E$ and $E'$ are of finite presentation over $R$. Therefore \ref{PropIsogAModule_B} implies that $f$ is universally open by \cite[IV$_2$, Th\'eor\`eme~2.4.6]{EGA}. In particular $(f\times\id_k)(E\times_Rk)\subset E'\times_Rk$ is open for every point $\Spec k\to\Spec R$ of $\Spec R$. Since $E'\times_Rk$ was assumed to be connected, it possesses no proper open subgroup, and hence $f\times\id_k$ is surjective. This establishes \ref{PropIsogAModule_C}.

To prove that \ref{PropIsogAModule_C} implies \ref{PropIsogAModule_D} again consider the morphism $f\times\id_k\colon E\times_Rk\to E'\times_Rk$ over a point $\Spec k\to\Spec R$ of $\Spec R$. Since $f\times\id_k$ is surjective, $f^*\otimes\id_k\colon\Gamma(E',\CO_{E'})\otimes_Rk\,\into\,\Gamma(E,\CO_E)\otimes_Rk$ is injective, because otherwise its kernel would define a proper closed subscheme of $E'\times_Rk$ through which $f\times\id_k$ factors. Since all fibers of $f$ are isomorphic to $\ker f$, and hence finite, \cite[Corollary~13.5]{Eisenbud} shows that $d'\,=\,\dim\Gamma(E',\CO_{E'})\otimes_Rk\,=\,\dim\Gamma(E,\CO_E)\otimes_Rk\,=\,d$.

We prove the implication \ref{PropIsogAModule_D}$\Longrightarrow$\ref{PropIsogAModule_B}. Consider the fiber $f\times\id_k\colon E\times_Rk\to E'\times_Rk$ over a point $\Spec k\to\Spec R$ of $\Spec R$ and the inclusion $\bigl(\Gamma(E',\CO_{E'})\otimes_Rk\bigr)/\ker(f^*\otimes\id_k)\,\longinto\,\Gamma(E,\CO_E)\otimes_Rk$. Since all fibers of $f$ are finite, \cite[Corollary~13.5]{Eisenbud} implies $\dim\Gamma(E',\CO_{E'})\otimes_Rk\,=\,d'\,=\,d\,=\,\dim\Gamma(E,\CO_E)\otimes_Rk\,=\,\dim\bigl(\Gamma(E',\CO_{E'})\otimes_Rk\bigr)/\ker(f^*\otimes\id_k)$. It follows that $\ker(f^*\otimes\id_k)=(0)$ and $f^*\otimes\id_k\colon\Gamma(E',\CO_{E'})\otimes_Rk\,\into\,\Gamma(E,\CO_E)\otimes_Rk$ is injective. Let $\Fm\subset\Gamma(E,\CO_E)\otimes_Rk$ be a maximal ideal. Then $(f^*\otimes\id_k)^{-1}(\Fm)\subset\Gamma(E',\CO_{E'})\otimes_Rk$ is a maximal ideal by \cite[Theorem~4.19]{Eisenbud}. Since the fiber of $f$ over $\Fm$ is finite, \cite[Theorem~18.16(b)]{Eisenbud} implies that $f\otimes\id_k$ is flat at $\Fm$. Since $E$ and $E'$ are smooth over $R$ it follows from \cite[IV$_3$, Th\'eor\`eme~11.3.10]{EGA} that $f$ is flat.

Finally we show that \ref{PropIsogAModule_B} and \ref{PropIsogAModule_C} together imply \ref{PropIsogAModule_A}. By \ref{PropIsogAModule_B} and \ref{PropIsogAModule_C} the morphism $f\colon E\to E'$ is faithfully flat. Whether $f$ is finite can by \cite[IV$_2$, Proposition~2.7.1]{EGA} be tested after the faithfully flat base change $E\to E'$. By \eqref{EqIsomKernel} the finiteness of the projection $E\times_{E'}E\to E$ onto the first factor follows from the finiteness of $\ker f$ over $\Spec R$. This proves \ref{PropIsogAModule_A}.
\end{proof}

\begin{corollary}\label{CorKernel}
Let $f\in\Hom_R(\ulE,\ulE')$ be an isogeny. Then
\begin{enumerate}
\item \label{CorKernel_A}
the kernel $\ker f$ of $f$ is a finite locally free group scheme and a strict $\BF_q$-module scheme over $R$.
\item \label{CorKernel_B}
$E'$ is the quotient $E/\ker f$.
\end{enumerate}
\end{corollary}

\begin{proof}
\ref{CorKernel_A} Since $f$ is flat of finite presentation by \cite[IV$_1$, Proposition~1.6.2(v)]{EGA}, $\ker f$ is flat of finite presentation over $R$. Since it is also finite, it is finite locally free. Over a faithfully flat $R$-algebra $R'$ both $E$ and $E'$ become isomorphic to powers of $\BG_{a,R'}$ and hence are strict $\BF_q$-module schemes by Example~\ref{ExampleStrictMorphism}. Therefore $(\ker f)\otimes_R R'$ is a strict $\BF_q$-module scheme over $R'$ by \cite[Proposition~2]{Faltings02} and $\ker f$ is a strict $\BF_q$-module scheme over $R$ by Lemma~\ref{LemmaStrict}.

\smallskip\noindent
\ref{CorKernel_B} This follows from \cite[Th\'eor\`eme~V.4.1]{SGA3}.
\end{proof}

\begin{proposition}\label{PropIsogDrinfeldMod}
\begin{enumerate}
\item \label{PropIsogDrinfeldMod_A}
If $\ulE$ and $\ulE'$ are Drinfeld $A$-modules over $R$ with $\Spec R$ connected and $f\in\Hom_k(\ulE,\ulE')$, then $f$ is an isogeny if and only if $f\ne0$.
\item \label{PropIsogDrinfeldMod_B}
If this is the case then $f$ is separable if and only if $\Lie f\in R\mal$.
\end{enumerate}
\end{proposition}

\begin{proof}
\ref{PropIsogDrinfeldMod_A} Let $f\colon\ulE\to\ulE'$ be an isogeny, then $f\ne0$ because the zero morphism is not surjective. Conversely let $f\ne0$. By Proposition~\ref{PropIsogAModule}\ref{PropIsogAModule_D} we must show that $\ker f$ is finite. This question is local on $\Spec R$, so we may assume that $E=E'=\BG_{a,R}$ and that $\ulE=(E,\phi)$ and $\ulE'=(E',\psi)$ are in standard form. Let $t\in A\setminus\BF_q$, and hence $\deg_\tau\phi_t>0$ and $\deg_\tau\psi_t>0$. By Lemma~\ref{LemmaStandardForm}\ref{LemmaStandardForm_B} applied to $f\circ\phi_t=\psi_t\circ f$ we have $f=\sum_{i=0}^n f_i\tau^i\in R\{\tau\}$ with $f_n\in R\mal$. It follows that $\ker f=\Spec R[x]/(\sum_{i=0}^n f_i x^{q^i})$ which is finite over $R$.

\medskip\noindent
\ref{PropIsogDrinfeldMod_B} By the Jacobi criterion \cite[\S 2.2, Proposition~7]{BLR}, $\ker f=\Spec R[x]/(\sum_{i=0}^n f_i x^{q^i})$ is \'etale if and only if $\Lie f=f_0=\tfrac{\partial f(x)}{\partial x}\in R\mal$.
\end{proof}

\bigskip

Next we turn to $A$-motives.

\begin{definition}\label{DefIsogAMotive}
A morphism $f\in\Hom_R(\ulM,\ulN)$ between $A$-motives over $R$ is an \emph{isogeny} if $f$ is injective and $\coker f$ is finite and locally free as $R$-module. If there exists an isogeny between $\ulM$ and $\ulN$ then they are called \emph{isogenous}. (Being isogenous is an equivalence relation; see Corollary~\ref{CorIsogEquiv} below.) A quasi-morphism $f\in\QHom_R(\ulM,\ulN)$ which is of the form $g\otimes c$ for an isogeny $g\in\Hom_R(\ulM,\ulN)$ and a $c\in Q$ is called a \emph{quasi-isogeny}.

If $f$ is an isogeny and $\ulM$ and $\ulN$ are effective, then the snake lemma yields the following commutative diagram with exact rows and columns
\begin{equation}\label{EqDiagIsogeny}
\xymatrix { & & 0 \ar[d] & 0 \ar[d] & \ker(\tau_{\coker f}) \ar@{^{ (}->}[d] \\
& 0 \ar[r] & \sigma^*M \ar[r]^{\TS\sigma^*f} \ar[d]^{\TS\tau_M} & \sigma^*N \ar[r] \ar[d]^{\TS\tau_{N}} & \sigma^*(\coker f) \ar[r] \ar[d]^{\tau_{\coker f}} & 0\;\, \\
& 0 \ar[r] & M \ar[r]^{\TS f} \ar@{->>}[d] & N \ar[r] \ar@{->>}[d] & \coker f \ar[r] \ar@{->>}[d] & 0\;\, \\
0 \ar[r] & \ker(\tau_{\coker f}) \ar[r] & \coker\tau_M \ar[r] & \coker\tau_{N} \ar[r] & \coker(\tau_{\coker f}) \ar[r] & 0\,.
}
\end{equation}
Namely, by local freeness of $\coker f$ the upper row is again exact and identifies $\sigma^*(\coker f)$ with $\coker(\sigma^*f)$.

An isogeny $f\colon\ulM\to\ulN$ between effective $A$-motives is \emph{separable} if $\tau_{\coker f}\colon\sigma^*(\coker f)\to\coker f$ is an isomorphism. \comment{Need a definition without the assumption of effectivity !}
\end{definition}

\begin{remark}\label{RemIsogBaseCh}
If $f\in\Hom_R(\ulM,\ulN)$ is an isogeny and $S$ is an $R$-algebra, then the base change $f\otimes\id_S\in\Hom_S(\ulM\otimes_RS,\ulN\otimes_RS)$ of $f$ to $S$ is again an isogeny. This follows from the exact sequence $0\longto\ulM\xrightarrow{\es f\;}\ulN\longto\coker f\longto0$ because $\coker f$ is a flat $R$-module.
\end{remark}

\begin{example}\label{ExPhi_a}
For $0\ne a\in A$ the morphism $a\colon\ulM\to\ulM$ is an isogeny with $\coker a=M/aM$. Let $\ulM$ be effective. Then $a$ is separable if and only if $\ker(\tau_{\coker a})=\coker(\tau_{\coker a})=(0)$. That is, if and only if multiplication with $a$ is an automorphism of $\coker\tau_M$. Since $a-\gamma(a)$ is nilpotent on $\coker\tau_M$ this is the case if and only if $\gamma(a)\in R\mal$. For the corresponding result about abelian Anderson $A$-modules see Corollary~\ref{CorPhi_a}. \comment{Also if $\ulM$ is not effective?}
\end{example}

\begin{proposition}\label{PropDim}
Let $\ulM$ and $\ulN$ be $A$-motives over $R$. If $\ulM$ and $\ulN$ are isogenous then $\rk\ulM=\rk\ulN$, and if, moreover, $\ulM$ and $\ulN$ are effective, then $\rk_R\coker\tau_M=\rk_R\coker\tau_{N}$. Conversely assume $\rk\ulM=\rk\ulN$ and let $f\in\Hom_R(\ulM,\ulN)$ be a morphism such that $\coker f$ is a finitely generated $R$-module. Then $f$ is an isogeny.
\end{proposition}

\begin{proof}
Let $f\colon\ulM\to\ulN$ be an isogeny. Since $M$, respectively $\coker\tau_M$, are finite locally free over $A_R$, respectively over $R$, we can compute their ranks by choosing a maximal ideal $\Fm\subset R$ and applying the base change from $R$ to $k=R/\Fm$. Then $f\otimes\id_k$ is still an isogeny by Remark~\ref{RemIsogBaseCh}. Since $\coker(f\otimes\id_k)$ is a torsion $A_k$-module it follows that
\[
\rk\ulM\,=\,\rk_{A_R}M\,=\,\rk_{A_k}(M\otimes_R k)\,=\,\rk_{A_k}(N\otimes_R k)\,=\,\rk_{A_R}N\,=\,\rk\ulN\,.
\]
If $\ulM$ and $\ulN$ are effective, we consider diagram~\eqref{EqDiagIsogeny} for the isogeny $f\otimes\id_k$. Since $\coker(f\otimes\id_k)$ and $\sigma^*\coker(f\otimes\id_k)$ are finite dimensional $k$-vector spaces of the same dimension, the right vertical column and the bottom row of diagram~\eqref{EqDiagIsogeny} imply that 
\[
\rk_R\coker\tau_M\,=\,\dim_k\coker(\tau_M\otimes\id_k)\,=\,\dim_k\coker(\tau_{N}\otimes\id_k)\,=\,\rk_R\coker\tau_{N}\,.
\]
The converse follows from Lemma~\ref{LemmaInjCokerLocFree}.
\end{proof}

After these preparations we are now able to formulate and prove our main theorem.

\begin{theorem}\label{ThmIsogeny}
Let $f\in\Hom_R(\ulE,\ulE')$ be a morphism between abelian Anderson $A$-modules and let $\ulM(f)\in\Hom_R(\ulM',\ulM)$ be the associated morphism between the associated effective $A$-motives $\ulM=\ulM(\ulE)$ and $\ulM'=\ulM(\ulE')$. Then
\begin{enumerate}
\item \label{ThmIsogeny_A}
$f$ is an isogeny if and only if $\ulM(f)$ is an isogeny.
\item \label{ThmIsogeny_B}
$f$ is a separable isogeny if and only if $\ulM(f)$ is a separable isogeny.
\item \label{ThmIsogeny_C}
If $f$ is an isogeny there are canonical $A$-equivariant isomorphisms of finite $\BF_q$-shtukas
\[
\bigl(\coker\ulM(f),\tau_{\coker\ulM(f)}\bigr) \; \isoto \; \ulM_q(\ker f)
\]
and of finite locally free $R$-group schemes
\[
\Dr_q\bigl(\coker\ulM(f)\bigr) \; \isoto \; \ker f\,.
\]
\end{enumerate}
\end{theorem}

\begin{proof}
In the beginning we do neither assume that $f$ nor that $\ulM(f)$ is an isogeny. We denote by $\iota$ the inclusion $\ker f\into E$. Consider the $A_R$-homomorphism $\ulM(\ulE)\to\ulM_q(\ker f),\,m\mapsto m\circ\iota$, which is compatible with the Frobenii $\tau_{M(\ulE)}$ and $F_{M_q(\ker f)}$. Since $m=\ulM(f)(m')=m'\circ f$ implies $m'\circ f\circ\iota=0$, it factors over
\begin{equation}\label{EqFactors}
\coker\ulM(f)\;\longto\;\ulM_q(\ker f),\quad m\mod\im\ulM(f)\mapsto m\circ\iota\,.
\end{equation}
On the other hand we claim that there are $A$-equivariant morphisms
\begin{equation}\label{EqClosedIm}
\Dr_q\bigl(\ulM_q(\ker f)\bigr)\longto\Dr_q(\coker\ulM(f)) \longinto \ker f\longinto E\,.
\end{equation}
where the last two are closed immersions. The first morphism is obtained from \eqref{EqFactors}. Moreover, the epimorphism $\ulM(\ulE)\onto\coker\ulM(f)$ induces by Example~\ref{ExVOfAModule} an $A$-equivariant closed immersion $\alpha\colon\Dr_q(\coker\ulM(f)) \into \Dr_q\bigl(\ulM(\ulE)\bigr)=\ulE$. We compose it with $f\colon E\to E'$ and show that the composition factors through the zero section $e'\colon\Spec R\to E'$. This will imply that $\alpha$ factors through $\ker f$. We can study this composition after a faithfully flat base change $R\to R'$ over which we have an $\BF_q$-linear isomorphism $\beta\colon E'\otimes_RR'\cong\BG_{a,R'}^{d'}=\Spec R'[y_1,\ldots,y_d]$. Let $m'_i:=pr_i\circ\beta\in M(\ulE')\otimes_RR'$ where $pr_i\colon\BG_{a,R'}^d\to\BG_{a,R'}=\Spec R[\xi]$ is the projection onto the $i$-th factor. Then $pr_i^*(\xi)=y_i$ and $\alpha^*f^*\beta^*(y_i)=\alpha^*f^*m'_i{}^*(\xi)=\alpha^*\circ\ulM(f)(m'_i)^*(\xi)=0$ because $\ulM(f)(m'_i)=0$ in $\coker\ulM(f)$.

\medskip\noindent
\ref{ThmIsogeny_A} Now assume that $f$ is an isogeny. Then $\ker f$ is a finite locally free group scheme over $R$, and a strict $\BF_q$-module scheme by Corollary~\ref{CorKernel}\ref{CorKernel_A}. So $\ulM_q(\ker f)$ is a finite locally free $R$-module by Theorem~\ref{ThmEqAModSch} and the morphism $\Dr_q\bigl(\ulM_q(\ker f)\bigr)\to \ker f$ in \eqref{EqClosedIm} is an isomorphism. This shows that $\Dr_q(\coker\ulM(f))\isoto\ker f$. We next show that the map \eqref{EqFactors} is an isomorphism. Its cokernel is a finite $R$-module because $\ulM_q(\ker f)$ is. We apply again a faithfully flat base change $R\to R'$ such that $E\otimes_RR'\cong\BG_{a,R'}^d$ and $E'\otimes_RR'\cong\BG_{a,R'}^{d'}$. Then $f$ is given by a matrix $F\in R'\{\tau\}^{d'\times d}$ by Lemma~\ref{LemmaEndGa}. By faithfully flat descent and by Nakayama's lemma~\cite[Corollaries~2.9 and 4.8]{Eisenbud} the map \eqref{EqFactors} will be surjective if for all maximal ideals $\Fm'\subset R'$ its tensor product with $k:=R'/\Fm'$ is surjective. By Remark~\ref{RemMotiveOfAModule} and its analog for $\ulM_q(\ker f)$ the tensor product of \eqref{EqFactors} with $k$ equals $\coker\ulM(f\times\id_k)\to\ulM_q\bigl(\ker (f\times\id_k)\bigr)$, where $f\times\id_k\colon\ulE\times_Rk\to\ulE'\times_Rk$ is given by the matrix $\olF:=F\otimes1_k$. In particular $\ker(f\times\id_k)=\Spec k[x_1,\ldots,x_d]/(f^*(y_\ell)\colon 1\le \ell\le d)$. Since $\ker f$ is finite, $k[x_1,\ldots,x_d]/(f^*(y_\ell)\colon 1\le \ell\le d)$ is a finite dimensional $k$-vector space. For fixed $i$ this implies that $\{x_i,x_i^q,x_i^{q^2},\ldots\}$ is linearly dependent and there is a positive integer $N$ and $b_{i,n}\in k$ such that $x_i^{q^{N+1}}=\sum\limits_{n=0}^N b_{i,n}\!\cdot\!x_i^{q^n}$ in $k[x_1,\ldots,x_d]/(f^*(y_\ell)\colon 1\le \ell\le d)$. We introduce the new variables $z_{i,n}:=x_i^{q^n}$ for $1\le i\le d$ and $0\le n\le N$. Then $f^*(y_\ell)$ is a $k$-linear relation between the $z_{i,n}$. Furthermore
\begin{eqnarray*}
& k[x_1,\ldots,x_d]/(f^*(y_\ell)\colon 1\le \ell\le d)\;\cong\; k[z_{i,n}\colon 1\le i\le d,\,\le n\le N]/I \quad\text{with}\\[1mm]
& I\;=\;\bigl(f^*(y_i),\,\TS z_{i,N}^q-\sum\limits_{n=0}^N b_{i,n}\!\cdot\!z_{i,n},\,z_{i,n}^q-z_{i,n+1}\colon 1\le i\le d,0\le n<N\bigr)\,.
\end{eqnarray*}
Let $\tilde z_1,\ldots,\tilde z_r$ be a $k$-basis of $(\bigoplus\limits_{i=1}^d\bigoplus\limits_{n=0}^Nk\!\cdot\!z_{i,n})/(f^*(y_\ell)\colon 1\le \ell\le d)$. Then there are elements $c_{ij}\in k$ for $1\le i,j\le r$ such that 
\[
k[x_1,\ldots,x_d]/(f^*(y_\ell)\colon 1\le \ell\le d)\;\cong\;k[\tilde z_1,\ldots,\tilde z_r]\big/\bigl(\tilde z_i^q-\TS\sum\limits_{j=1}^r c_{ij}\tilde z_j\colon 1\le i\le r\bigr)\;=:\;B\,.
\]
Moreover, the group law on $\ker f$ is given by the comultiplication $\Delta^*\colon B\to B\otimes_kB,\,\Delta^*(\wt z_i)=\wt z_i\otimes 1+1\otimes\wt z_i$ and the $\BF_q$-action is given by $\phi_\lambda\colon B\to B,\,\phi_\lambda^*(\wt z_i)=\gamma(\lambda)\!\cdot\!\wt z_i$.

We now are ready to compute $\ulM_q\bigl(\ker(f\times\id_k)\bigr)$ from \eqref{EqM(E)}. If $\BG_{a,k}=\Spec k[\xi]$ then every element $\wt m\in \ulM_q\bigl(\ker(f\times\id_k)\bigr)$ satisfies $\wt m^*(\xi)=\sum\limits_{\ell_i\in\{0\ldots q-1\}}d_{\ell_1,\ldots,\ell_r}\cdot\tilde z_1^{\ell_1}\cdot\ldots\cdot\tilde z_r^{\ell_r}$ with $d_{\ell_1,\ldots,\ell_r}\in k$. Since the $\tilde z_1^{\ell_1}\cdot\ldots\cdot\tilde z_r^{\ell_r}$ form a $k$-basis of $B$, the conditions $\Delta^*\wt m^*(\xi)=\wt m^*(\xi)\otimes 1+ 1\otimes\wt m^*(\xi)$ in $B\otimes_kB$ and $\phi_\lambda^*\wt m^*(\xi)=m^*(\gamma(\lambda)\!\cdot\!\xi)=\gamma(\lambda)\!\cdot\!\wt m^*(\xi)$ in $B$ for $\lambda\in\BF_q$ imply as in Lemma~\ref{LemmaEndGa} that $\wt m^*(\xi)=d_{1,0\ldots0}\!\cdot\!\tilde z_1+\ldots+d_{0\ldots0,1}\!\cdot\!\tilde z_r$. Since $\tilde z_i$ is a $k$-linear combination of the $z_{j,n}=x_j^{q^n}$ the morphism $m\colon E\times_Rk\to\BG_{a,k}$ with $m^*(\xi)=d_{1,0\ldots0}\!\cdot\!\tilde z_1+\ldots+d_{0\ldots0,1}\!\cdot\!\tilde z_r$ belongs to $\ulM(E\times_Rk)$ and maps to $\wt m$ under the map $\coker\ulM(f\times\id_k)\to\ulM_q\bigl(\ker (f\times\id_k)\bigr)$. This proves that \eqref{EqFactors} is surjective.

In order to show that \eqref{EqFactors} is injective let $m\in M(\ulE)$ be an element with $m\circ\iota=0$. By \cite[Th\'eor\`eme~V.4.1]{SGA3} the morphism $m\colon E\to\BG_{a,R}$ factors through $E/\ker f\isoto E'$ (use Corollary~\ref{CorKernel}\ref{CorKernel_B}) in the form $m=m'\circ f$ for an $m'\in M(\ulE')$. This shows that $m\mod\im\ulM(f)=0$ in $\coker\ulM(f)$. All together we have proved that $\coker\ulM(f)\isoto\ulM_q(\ker f)$ is a finite locally free $R$-module. Moreover, $\ulM(f)$ is injective, because if $m'\in M(\ulE')$ satisfies $m'\circ f=\ulM(f)(m')=0$ the surjectivity of $f$ implies $m'=0$. More precisely, $f$ is an epimorphism of sheaves for the \fpqc-topology by Proposition~\ref{PropIsogAModule}\ref{PropIsogAModule_F}. Now the injectivity of $\ulM(f)$ follows from the left exactness of the functor $\Hom_{R\text{-groups},\BF_q\text{\rm-lin}}(\fdot,\BG_{a,R})$. This proves that $\ulM(f)$ is an isogeny, and it also proves \ref{ThmIsogeny_C}.

\medskip\noindent
Conversely assume that $\ulM(f)$ is an isogeny. Then $d:=\dim\ulE=\dim\ulE'$ by Theorem~\ref{ThmMotiveOfAModule} and Proposition~\ref{PropDim}. We prove that $\ker f$ is finite. For this purpose we apply a faithfully flat base change $R\to R'$ such that $E\otimes_RR'\cong\BG_{a,R'}^d=\Spec R'[x_1,\ldots,x_d]$ and $E'\otimes_RR'\cong\BG_{a,R'}^d=\Spec R[y_1,\ldots,y_d]$. Also when we write $\BG_{a,R'}=\Spec R'[\xi]$ then $\ulM(\ulE\times_RR')\cong\bigoplus\limits_{i=1}^d R'\{\tau\}\!\cdot\!m_i$ and $\ulM(\ulE'\times_RR')\cong\bigoplus\limits_{i=1}^d R'\{\tau\}\!\cdot\!m'_i$ where $m_i^*(\xi)=x_i$ and $m'_i{}^*(\xi)=y_i$. Consider the epimorphism of $R$-modules
\[
\xymatrix {
\bigoplus\limits_{i=1}^d\bigoplus\limits_{0\le n}R'\!\cdot\!\tau^n m_i\;\cong\;\ulM(\ulE\times_R R') \ar@{->>}[r]^{\TS\qquad\quad\es\delta} & \coker\ulM(f\otimes\id_{R'})\,.
}
\]
Since $\coker\ulM(f\otimes\id_{R'})$ is finite locally free over $R'$, and hence projective, this epimorphism has a section $s$ whose image lies in $\bigoplus\limits_{i=1}^d\bigoplus\limits_{n=0}^N R'\!\cdot\!\tau^n m_i$ for some $N$. It follows that $\tau^{N+1}m_i-s(\delta(\tau^{N+1}m_i))$ maps to zero in $\coker\ulM(f\otimes\id_{R'})$. That is, there are elements $b_{i,j,n}\in R'$ and $\wt m'_i\in\ulM(\ulE'\times_RR')$ with $\tau^{N+1}m_i\,-\,\sum\limits_{j=1}^d\sum\limits_{n=0}^Nb_{i,j,n}\!\cdot\!\tau^nm_j\,=\,\ulM(f)(\wt m'_i)$. Applying this equation to $\xi$ yields
\[
x_i^{q^{N+1}}-\;\sum_{j=1}^d\sum_{n=0}^Nb_{i,j,n}\!\cdot\!x_j^{q^n}\;=\;f^*\wt m'_i{}^*(\xi)\;\in\;f^*R'[y_1,\ldots,y_d]\;\cong\;f^*\Gamma(E',\CO_{E'})\otimes_R R'\,.
\]
Thus $f\times\id_{R'}\colon E\times_RR'\to E'\times_RR'$ is finite. By faithfully flat descent \cite[IV$_{\rm 2}$, Proposition~2.7.1]{EGA} also $f$ is finite. By Proposition~\ref{PropIsogAModule}\ref{PropIsogAModule_D} this proves that $f$ is an isogeny and establishes \ref{ThmIsogeny_A}.

\medskip\noindent
Finally \ref{ThmIsogeny_B} follows from \ref{ThmIsogeny_C} and Theorem~\ref{ThmEqAModSch}\ref{ThmEqAModSch_B}.
\end{proof}

\begin{corollary}\label{CorRkIsog}
If $\ulE$ and $\ulE'$ are isogenous abelian Anderson $A$-modules then $\rk\ulE=\rk\ulE'$.
\end{corollary}

\begin{proof}
This follows directly from Theorems~\ref{ThmMotiveOfAModule}, \ref{ThmIsogeny} and Proposition~\ref{PropDim}.
\end{proof}

\begin{corollary}\label{CorPhi_a}
Let $\ulE$ be an abelian Anderson $A$-module over $R$ and let $a\in A$. Then $\phi_a\colon\ulE\to\ulE$ is an isogeny. It is separable if and only if $\gamma(a)\in R\mal$.
\end{corollary}

\begin{proof}
The assertion follows from Theorem~\ref{ThmIsogeny} and Example~\ref{ExPhi_a}. The criterion for separability can also be proved without reference to $A$-motives; see our proof of Theorem~\ref{ThmATorsion}\ref{ThmATorsion_B} below.
\end{proof}

We next come to our second main result.

\begin{theorem}\label{ThmDualIsog}
Let $\ulM$ and $\ulN$ be two $A$-motives over $R$ and let $f\in\Hom_R(\ulM,\ulN)$ be a morphism. Then the following are equivalent:
\begin{enumerate}
\item \label{ThmDualIsog_A}
$f$ is an isogeny,
\item \label{ThmDualIsog_B}
there is an element $0\ne a\in A$ such that $f$ induces an isomorphism of $A_R[\tfrac{1}{a}]$-modules $M[\tfrac{1}{a}]\isoto N[\tfrac{1}{a}]$.
\end{enumerate}
In particular, a quasi-morphism $f\in\QHom_R(\ulM,\ulN)$ is a quasi-isogeny if and only if it induces an isomorphism $f\colon M[\tfrac{1}{a}]\isoto N[\tfrac{1}{a}]$ for an element $a\in A\setminus\{0\}$.
\end{theorem}

\begin{proof}
\ref{ThmDualIsog_B}$\Longrightarrow$\ref{ThmDualIsog_A}
Clearly $\rk\ulM=\rk\ulN$. Since $\coker f$ is a finitely generated $A_R$-module, $(\coker f)\otimes_A A[\tfrac{1}{a}]=(0)$ implies that $a^n\cdot\coker f=(0)$ for some positive integer $n$. Therefore, $\coker f$ is a finitely generated module over $A_R/(a^n)=A/(a^n)\otimes_{\BF_q}R$, whence over $R$. So \ref{ThmDualIsog_A} follows from Proposition~\ref{PropDim}.

\medskip\noindent
\ref{ThmDualIsog_A}$\Longrightarrow$\ref{ThmDualIsog_B} 
If $R$ is a field this was proved in \cite[Corollary~5.4]{BH1} and also follows from \cite[Proposition~3.4.5]{Papanikolas} and \cite[Proposition~3.1.2]{TaelmanArtin}. We generalize the proof to the relative situation.

\medskip\noindent
1. If $f$ is an isogeny, then $\coker f$ is a finite locally free $R$-module, which we may assume to be free after passing to an open affine covering of $\Spec R$. Let $t\in A\setminus\BF_q$ and consider the finite flat homomorphism $\wt A:=\BF_q[t]\into A$ from Lemma~\ref{LemmaF_q[t]}, under which we view $\ulM$ and $\ulN$ as $\wt A$-motives by restriction of scalars. That is, we view $M$ and $N$ as locally free $R[t]$-modules of rank $\tilde r=\rk\ulM\cdot\rk_{\wt A}A$ and $\tau_M$ and $\tau_N$ as $R[t][\tfrac{1}{t-\gamma(t)}]$-isomorphisms. By multiplying both $\tau_M$ and $\tau_N$ with $(t-\gamma(t))^e$ for $e\gg0$ we may assume that $\ulM$ and $\ulN$ are effective $\wt A$-motives without changing the isogeny $f\colon\ulM\to\ulN$. Let $\Fa=\ann_{R[t]}(\coker f)=\ker\bigl(R[t]\to \End_R(\coker f)\bigr)$ be the annihilator of $\coker f$. By the Cayley-Hamilton theorem \cite[Theorem~4.3]{Eisenbud} (applied with $I=R$), the monic characteristic polynomial $\chi_t$ of the endomorphism $t$ of $\coker f$ lies in $\Fa$. This shows that $R[t]/\Fa$ is a quotient of the finite $R$-module $R[t]/(\chi_t)$. In particular the closed subscheme $V:=\Spec R[t]/\Fa$ of $\BA^1_R=\Spec R[t]$ is finite over $\Spec R$. On its open complement $f\colon M\to N$ is an isomorphism.

We now consider the exterior powers $\wedge^{\tilde r}M$ and $\wedge^{\tilde r}N$ of the $R[t]$-modules $M$ and $N$ and set $\CL:=(\wedge^{\tilde r}M)\dual\otimes\wedge^{\tilde r}N$. These are invertible $R[t]$-modules. The isogeny $f$ induces a global section $\wedge^{\tilde r}f$ of the invertible sheaf $\CL$ on $\BA^1_R$ which provides an isomorphism $\CO_{\BA^1_R}\isoto\CL$, $1\mapsto\wedge^{\tilde r}f$ on $\BA^1_R\setminus V$. Likewise we obtain global sections $\wedge^{\tilde r}\sigma^*f$, resp.\ $\wedge^{\tilde r}\tau_M$, resp.\ $\wedge^{\tilde r}\tau_N$ of the invertible sheaves $\sigma^*\CL$, resp.\ $(\wedge^{\tilde r}\sigma^*M)\dual\otimes\wedge^{\tilde r}M$, resp.\ $(\wedge^{\tilde r}\sigma^*N)\dual\otimes\wedge^{\tilde r}N$ by the effectivity assumption on $\ulM$ and $\ulN$. Diagram~\eqref{EqDiagIsogeny} implies that there is an equality of global sections 
\begin{equation}\label{EqOfGlobalSect}
\wedge^{\tilde r}f\otimes\wedge^{\tilde r}\tau_M\;=\;\wedge^{\tilde r}\tau_N\otimes\wedge^{\tilde r}\sigma^*f
\end{equation}
of $(\wedge^{\tilde r}\sigma^*M)\dual\otimes\wedge^{\tilde r}N\;=\;\CL\otimes(\wedge^{\tilde r}\sigma^*M)\dual\otimes\wedge^{\tilde r}M\bigr)\;=\;\bigl((\wedge^{\tilde r}\sigma^*N)\dual\otimes\wedge^{\tilde r}N\bigr)\otimes\sigma^*\CL$. 

Since $V$ is proper over $\Spec R$ and the projective line $\BP^1_{R}$ is separated, the map $V\into \BA^1_{R}\into\BP^1_{R}$ is a closed immersion which does not meet $\{\infty\}\times_{\BF_q}\Spec R$, where $\{\infty\}=\BP^1_{\BF_q}\setminus\BA^1_{\BF_q}$. Thus we may glue $\CL$ with the trivial sheaf $\CO_{\BP^1_{R}\setminus V}$ on $\BP^1_{R}\setminus V$ along the isomorphism $\CO_{\BP^1_{R}}\isoto\CL$, $1\mapsto \wedge^{\tilde r}f$ over $\BA^1_{R}\setminus V$. In this way we obtain an invertible sheaf $\olCL$ on the projective line $\BP^1_{R}$. By replacing $\olCL$ with $\olCL\otimes\CO_{\BP^1_{R}}(m\cdot\infty)$ for a suitable integer $m$ we may achieve that $\olCL$ has degree zero (see \cite[\S\,9.1, Proposition~2]{BLR}) and induces an $R$-valued point of the relative Picard functor $\Pic^0_{\BP^1/\BF_q}$; cf.~\cite[\S\,8.1]{BLR}. Since $\Pic^0_{\BP^1/\BF_q}$ is trivial, \cite[\S\,8.1, Proposition~4]{BLR} shows that $\olCL\cong\CK\otimes_{R}\CO_{\BP^1_{R}}$ for an invertible sheaf $\CK$ on $\Spec R$. Replacing $\Spec R$ by an open affine covering which trivializes $\CK$ we may assume that there is an isomorphism $\alpha\colon\CL\isoto R[t]$ of $R[t]$-modules. Let $h:=\alpha(\wedge^{\tilde r}f)\in R[t]$.

\medskip\noindent
2. Let $d:=\rk_R\coker\tau_M$. We claim that locally on $\Spec R$ there is a positive integer $n_0$ and for every integer $n\ge n_0$ an isomorphism of $R[t]$-modules
\begin{equation}\label{EqThmDualIsog1}
\bigl((\wedge^{\tilde r}\sigma^*M)\dual\otimes_{R[t]}\wedge^{\tilde r}M\bigr)^{\otimes q^n}\;\isoto\;R[t]\quad\text{with}\quad(\wedge^{\tilde r}\tau_M)^{\otimes q^n}\;\longmapsto\;\bigl(t-\gamma(t)\bigr)^{q^nd}
\end{equation}
and similarly for $\ulN$. To prove the claim %note that $t-\gamma(t)$ is a non-zero divisor in $R[t]$. Therefore the localization maps $R[t]\to R[t][\tfrac{1}{t-\gamma(t)}]$ and $\wedge^{\tilde r}M\to \wedge^{\tilde r}M\otimes_{R[t]}R[t][\tfrac{1}{t-\gamma(t)}]$ are injective. Since $\tau_M$ and hence also $\wedge^{\tilde r}\tau_M$ are isomorphisms after inverting $t-\gamma(t)$ we conclude that $\wedge^{\tilde r}\tau_M\colon\wedge^{\tilde r}\sigma^*M\to\wedge^{\tilde r}M$ is injective and its cokernel is annihilated by a power of $t-\gamma(t)$. As in Remark~\ref{RemCokerLocFree} it follows that the $R$-module $\coker\wedge^{\tilde r}\tau_M$ is finitely presented, and then Lemma~\ref{LemmaCokerLocFree} shows that it is locally free. 
we apply Proposition~\ref{PropAMotiveEffective}\ref{PropAMotiveEffective_C} to the $A$-motive $\wedge^{\tilde r}\ulM$ and derive that $\wedge^{\tilde r}\tau_M\colon\wedge^{\tilde r}\sigma^*M\to\wedge^{\tilde r}M$ is injective $\coker\wedge^{\tilde r}\tau_M$ is a finite locally free $R$-module, annihilated by a power of $t-\gamma(t)$. 
Consider the exact sequence
\begin{equation}\label{EqSeqCokerWedge}
0 \longto \wedge^{\tilde r}\sigma^*M\otimes_{R[t]} (\wedge^{\tilde r}M)\dual \xrightarrow{\es\wedge^{\tilde r}\tau_M\otimes\id_{(\wedge^{\tilde r}M)\dual}\;} R[t] \longto \coker\wedge^{\tilde r}\tau_M\otimes_{R[t]} (\wedge^{\tilde r}M)\dual \longto 0\,.
\end{equation}
Choose an open affine covering of $\Spec R[t]$ which trivializes the locally free $R[t]$-module $\wedge^{\tilde r}M$. Pulling back this covering under the section $\Spec R\isoto\Spec R[t]/(t-\gamma(t))\into\Spec R[t]$ gives an open affine covering of $\Spec R$ on which we may find an isomorphism $\coker\wedge^{\tilde r}\tau_M\otimes_{R[t]} (\wedge^{\tilde r}M)\dual\isoto\coker\wedge^{\tilde r}\tau_M$. We replace $\Spec R$ by this open affine covering and even shrink it further in such a way that $\coker\wedge^{\tilde r}\tau_M$ becomes a free $R$-module. By \cite[Proposition~4.1(b)]{Eisenbud} the sequence~\eqref{EqSeqCokerWedge} is then isomorphic to the sequence
\begin{equation}\label{EqSeqCokerWedge2}
\xymatrix @C+1pc{
0 \ar[r] & R[t] \ar[r]^{g} & R[t] \ar[r] & \coker\wedge^{\tilde r}\tau_M \ar[r] & 0\,,
}
\end{equation}
where $g\in R[t]$ is a monic polynomial of degree equal to $\rk_R(\coker\wedge^{\tilde r}\tau_M)$. We now tensor sequence~\eqref{EqSeqCokerWedge2} over $R$ with $k:=\Quot(R/\Fp)$ where $\Fp\subset R$ is a prime ideal. It remains exact because $\coker\wedge^{\tilde r}\tau_M$ is free. Since $k[t]$ is a principal ideal domain the elementary divisor theorem applied to 
\[
\xymatrix @C+2pc{
0 \ar[r] & \sigma^*M\otimes_Rk \ar[r]^{\tau_M\otimes\id_k} & M\otimes_Rk \ar[r] & \coker\tau_M\otimes_Rk \ar[r] & 0
}
\]
allows to write $\tau_M\otimes\id_k$ as a diagonal matrix. This shows that $\coker\wedge^{\tilde r}\tau_M\otimes_Rk$ is a $k$-vector space of dimension equal to $\rk_R(\coker\tau_M)=:d$. Since $t-\gamma(t)$ is nilpotent on this vector space, the Cayley-Hamilton theorem from linear algebra implies $g\mod\Fp=(t-\gamma(t))^{d}$. In particular the coefficients of the difference $g':=g-(t-\gamma(t))^{d}$ lie in every prime ideal of $R$, and hence are nilpotent by \cite[Corollary~2.12]{Eisenbud}. Therefore there is a positive integer $n_0$ with $(g')^{q^{n_0}}=0$, whence $g^{q^n}=(t-\gamma(t))^{q^nd}$ for every $n\ge n_0$. The $q^n$-th tensor power of the isomorphism between (the left entries in) the sequences~\eqref{EqSeqCokerWedge} and \eqref{EqSeqCokerWedge2} provides the isomorphism in \eqref{EqThmDualIsog1}. This proves the claim.

\medskip\noindent
3. Since $d=\rk_R\coker\tau_M=\rk_R\coker\tau_N$ by Proposition~\ref{PropDim}, equations~\eqref{EqOfGlobalSect} and \eqref{EqThmDualIsog1} imply that for $n\gg0$ there is an isomorphism $\beta\colon\sigma^*\CL^{\otimes q^n}\isoto\CL^{\otimes q^n}$ of $R[t]$-modules sending $(t-\gamma(t))^{q^n}(\sigma^*\wedge^{\tilde r}f)^{\otimes q^n}$ to $(t-\gamma(t))^{q^n}(\wedge^{\tilde r}f)^{\otimes q^n}$ and hence $(\sigma^*\wedge^{\tilde r}f)^{\otimes q^n}$ to $(\wedge^{\tilde r}f)^{\otimes q^n}$ because $t-\gamma(t)$ is a non-zero divisor. In particular the isomorphism 
\[
\alpha^{\otimes q^n}\circ\beta\circ(\sigma^*\alpha^{\otimes q^n})^{-1}\colon R[t]\isoto\sigma^*\CL^{\otimes q^n}\isoto\CL^{\otimes q^n}\isoto R[t]\,,
\]
which is given by multiplication with a unit $u\in R[t]\mal$, sends $\sigma(h^{q^n})=\sigma^*\alpha^{\otimes q^n}(\wedge^{\tilde r}\sigma^*f)^{\otimes q^n}$ to $h^{q^n}=\alpha^{\otimes q^n}(\wedge^{\tilde r}f)^{\otimes q^n}$. We thus obtain the equation $h^{q^n}=u\cdot\sigma(h^{q^n})$ in $R[t]$.

By Lemma~\ref{LemmaUnits} below, $u=\sum_{i\ge0}u_it^i$ with $u_0\in R\mal$ and $u_i\in R$ nilpotent for all $i\ge1$. Let $R'=R[v_0]/(v_0^{q-1}u_0-1)$ be the finite \'etale $R$-algebra obtained by adjoining a $(q-1)$-th root $v_0$ of $u_0^{-1}$. Then there is a unit $v=\sum_{i\ge1}v_it^i\in R'[t]\mal$ with $v=u\cdot\sigma(v)$. Indeed the latter amounts to the equations 
\[
\TS v_i\;=\;\sum\limits_{j=0}^i u_j v_{i-j}^q \qquad \text{and}\qquad \tfrac{v_i}{v_0}\;=\;(\tfrac{v_i}{v_0})^q+\sum\limits_{j\ge1}\tfrac{u_j}{u_0}\,(\tfrac{v_{i-j}}{v_0})^q
\]
which have the solutions $\tfrac{v_i}{v_0}=\sum_{n\ge0}\bigl(\sum_{j\ge1}\tfrac{u_j}{u_0}\,(\tfrac{v_{i-j}}{v_0})^q\bigr)^{q^n}$ because the $u_j$ are nilpotent. Therefore the element $v^{-1}h^{q^n}\in R'[t]$ satisfies $\sigma(v^{-1}h^{q^n})=v^{-1}h^{q^n}$. Working on each connected component of $\Spec R'$ separately, Lemma~\ref{LemmaTauInvInR} below shows that $a:=v^{-1}h^{q^n}\in\BF_q[t]\subset A$.

In the ring $R'[t][\tfrac{1}{a}]$ the element $h$ becomes a unit. Therefore the map $\alpha^{-1}\circ h\colon R'[t][\tfrac{1}{a}]\to\CL[\tfrac{1}{a}]$, $1\mapsto \wedge^{\tilde r}f$ is an isomorphism. This implies that $\wedge^{\tilde r}f\colon\wedge^{\tilde r}M[\tfrac{1}{a}]\to\wedge^{\tilde r}N[\tfrac{1}{a}]$ is an isomorphism, and hence also $f\colon M[\tfrac{1}{a}]\to N[\tfrac{1}{a}]$ by Cramer's rule (e.g.~\cite[III.8.6, Formulas (21) and (22)]{BourbakiAlgebra}). Thus we have established \ref{ThmDualIsog_B} \'etale locally on $\Spec R$. Replacing $a$ by the product of all the finitely many elements $a$ obtained locally, establishes \ref{ThmDualIsog_B} globally on $\Spec R$.

\medskip\noindent
4. To prove the statement about quasi-morphisms $f\in\QHom_R(\ulM,\ulN)$ assume first, that $f$ induces an isomorphism $f\colon M[\tfrac{1}{a}]\isoto N[\tfrac{1}{a}]$ for some $a\in A\setminus\{0\}$. Then $g:=a^n\cdot f\in\Hom_R(\ulM,\ulN)$ for $n\gg0$, because $M$ is finitely generated. In particular $g$ is an isogeny and $f=g\otimes a^{-n}$ is a quasi-isogeny.

Conversely, if $f$ is a quasi-isogeny, that is $f=g\otimes c$ for an isogeny $g\in\Hom_R(\ulM,\ulN)$ and a $c\in Q$, there is an element $a\in A\setminus\{0\}$ such that $g\colon M[\tfrac{1}{a}]\isoto N[\tfrac{1}{a}]$. If $d$ is the denominator of $c$ it follows that $f\colon M[\tfrac{1}{ad}]\isoto N[\tfrac{1}{ad}]$. 
\end{proof}

To finish the proof of Theorem~\ref{ThmDualIsog} we must demonstrate the following two lemmas.

\begin{lemma}\label{LemmaUnits}
An element $u=\sum_{i\ge0}u_it^i\in R[t]$ is a unit in $R[t]$ if and only if $u_0\in R\mal$ and $u_i$ is nilpotent for all $i\ge1$.
\end{lemma}

\begin{proof}
If the $u_i$ satisfy the assertion then there is a positive integer $n$ such that $u_i^{q^n}=0$ for all $i\ge1$. Therefore $u^{q^n}=u_0^{q^n}$ is a unit in $R[t]$ and so the same holds for $u$.

Conversely if $u$ is a unit then $u_0$ must be a unit in $R$. By \cite[Corollary~2.12]{Eisenbud} the kernel of the map $R\to\prod_{\Fp\subset R}R/\Fp$ where $\Fp$ runs over all prime ideals of $R$, equals the nil-radical of $R$. Under this map $u$ is sent to a unit in each factor $R/\Fp[t]$. Since $R/\Fp$ is an integral domain, the $u_i$ for $i\ge1$ must be sent to zero in each factor $R/\Fp$. This shows that $u_i$ is nilpotent for $i\ge1$. 
\end{proof}

\begin{lemma}\label{LemmaTauInvInR}
Assume that $R$ contains no idempotents besides $0$ and $1$, that is $\Spec R$ is connected. Then $R^\sigma:=\{x\in R\colon x^q=x\}=\BF_q$.
\end{lemma}

\begin{proof}
Let $\Fm\subset R$ be a maximal ideal and let $\bar x\in R/\Fm$ be the image of $x$. Then $\bar x^q=\bar x$ implies that $\bar x$ is equal to an element $\alpha\in\BF_q\subset R/\Fm$. Now $e:=(x-\alpha)^{q-1}$ satisfies $e^2=(x-\alpha)^{q-2}(x^q-\alpha^q)=(x-\alpha)^{q-1}=e$, that is $e$ is an idempotent. Since $e\in\Fm$ we cannot have $e=1$ and must have $e=0$. Therefore $x-\alpha=(x-\alpha)^q=(x-\alpha)\cdot e=0$ in $R$, that is $x=\alpha\in\BF_q$.
\end{proof}

\begin{corollary}\label{CorDualIsog}
If $f\in\Hom_R(\ulM,\ulN)$ is an isogeny between $A$-motives then there is an element $0\ne a\in A$ and an isogeny $g\in\Hom_R(\ulN,\ulM)$ with $f\circ g=a\cdot\id_\ulN$ and $g\circ f=a\cdot\id_\ulM$. The same is true for abelian Anderson $A$-modules.
\end{corollary}

\begin{proof}
Let $a\in A$ be the element from Theorem~\ref{ThmDualIsog}\ref{ThmDualIsog_B}. As in the proof of \ref{ThmDualIsog_B}$\Longrightarrow$\ref{ThmDualIsog_A} of this theorem there is a positive integer $n$ such $a^n\!\cdot\!\coker f=(0)$. Therefore there is a map $g\colon N\to M$ with $g\circ f=a^n\!\cdot\!\id_M$ and $f\circ g=a^n\!\cdot\!\id_N$. This implies that $g$ is injective, because $a^n$ is a non-zero divisor on $N$. From
\[
f\circ g\circ\tau_N\;=\;a^n\!\cdot\!\tau_N\;=\;\tau_N\circ\sigma^*a^n\!\cdot\!\id_N\;=\;\tau_N\circ\sigma^*f\circ\sigma^*g\;=\;f\circ\tau_M\circ g
\]
and the injectivity of $f$ we conclude that $g\circ\tau_N=\tau_M\circ\sigma^*g$ and that $g\in\Hom_R(\ulN,\ulM)$. By construction $g$ induces an isomorphism $N[\tfrac{1}{a}]\isoto M[\tfrac{1}{a}]$ after inverting $a$. So $g$ is an isogeny by Theorem~\ref{ThmDualIsog}. The statement about abelian Anderson $A$-modules follows from Theorems~\ref{ThmMotiveOfAModule} and \ref{ThmIsogeny}.
\end{proof}

\begin{corollary}\label{CorIsogEquiv}
The relation of being isogenous is an equivalence relation for $A$-motives and for abelian Anderson $A$-modules.
\end{corollary}

\begin{proof}
This follows from Theorem~\ref{ThmDualIsog} and Corollary~\ref{CorDualIsog}.
\end{proof}

\begin{corollary}\label{CorGenericChar}
Let $\gamma(A\setminus\{0\})\subset R\mal$ and let $f\in\Hom_R(\ulM,\ulN)$ be an isogeny between effective $A$-motives $\ulM$ and $\ulN$. Then $f$ is separable. The same is true for isogenies between abelian Anderson $A$-modules. \comment{Also for non-effective $A$-motives?}
\end{corollary}

\begin{proof}
Consider diagram~\eqref{EqDiagIsogeny} and set $K:=\coker(\tau_{\coker f})$. As in the proof of Theorem~\ref{ThmDualIsog} there is an element $0\ne a\in A$ and a positive integer $n$ with $a^n\cdot \coker f=(0)$, and hence $a^n\cdot K=(0)$. Let $e$ be an integer with $q^e\ge\rk_R\coker\tau_N$ and $q^e\ge n$. Then $(a\otimes1-1\otimes\gamma(a))^{q^e}\cdot \coker\tau_N=(0)$. Therefore
\[
0\;=\;(a\otimes1-1\otimes\gamma(a))^{q^e}\!\cdot\!K\;=\;(a^{q^e}\otimes1-1\otimes\gamma(a)^{q^e})\!\cdot\!K\;=\;-\gamma(a)^{q^e}\!\cdot\!K\,.
\]
Since $\gamma(a)\in R\mal$ we have $K=(0)$, and since $\coker f$ and $\sigma^*(\coker f)$ are finite locally free $R$ modules of the same rank, \cite[Corollary~8.12]{GoertzWedhorn} shows that $\tau_{\coker f}$ is an isomorphism, that is $f$ is separable. The statement about abelian Anderson $A$-modules follows from Theorem\ref{ThmIsogeny}\ref{ThmIsogeny_B}.
\end{proof}

\begin{corollary}
If $f\in\Hom_R(\ulM,\ulN)$ and $g\in\Hom_R(\ulN,\ulM)$ are isogenies between $A$-motives with  $f\circ g=a\cdot\id_\ulN$ and $g\circ f=a\cdot\id_\ulM$ for an $a\in A$, then there is an isomorphism of $Q$-algebras $\QEnd_R(\ulM)\isoto\QEnd_R(\ulN)$ given by $h\otimes b\mapsto f\circ h\circ g\otimes\tfrac{b}{a}$ for $h\in\End_R(\ulM)$. \qed
\end{corollary}

\begin{example}\label{ExFrobenius}
Let $R$ be an $A$-ring \emph{of finite characteristic $\epsilon$}, that is $\gamma\colon A\to R$ factors through $\BF_\epsilon:=A/\epsilon$ for a maximal ideal $\epsilon\subset A$. Let $\ell\in\BN_{>0}$ be divisible by $[\BF_\epsilon:\BF_q]$. Then $\sigma^{\ell*}(\CJ)=(a\otimes1-1\otimes\gamma(a)^{q^\ell}\colon a\in A)=\CJ\subset A_R$, because the elements $\gamma(a)\in\BF_\epsilon$ satisfy $\gamma(a)^{q^\ell}=\gamma(a)$. Let $\ulM=(M,\tau_M)$ be an $A$-motive over $R$. Then $\sigma^{\ell*}\ulM=(\sigma^{\ell*}M,\sigma^{\ell*}\tau_M)$ is also an $A$-motive over $R$, because $\sigma^{\ell*}\tau_M$ is an isomorphism outside $\Var(\sigma^{\ell*}\CJ)=\Var(\CJ)$. If $\ulM$ is effective, then the $A_R$-homomorphism
\begin{equation}\label{EqFrobenius}
{\rm Fr}_{q^\ell\!,\,\ulM}\;:=\;\tau_M^\ell\;:=\;\tau_M\circ\sigma^*\tau_M\circ\ldots\circ\sigma^{(\ell-1)*}\tau_M\colon\sigma^{\ell*}\ulM \longto\ulM
\end{equation}
satisfies $\tau_M\circ\sigma^*{\rm Fr}_{q^\ell\!,\,\ulM}={\rm Fr}_{q^\ell\!,\,\ulM}\circ\sigma^{\ell*}\tau_M$. Moreover, it is injective and its cokernel is a successive extension of the $\sigma^{i*}\coker\tau_M$ for $i=0,\ldots,\ell-1$, whence a finitely presented $R$-module. Therefore ${\rm Fr}_{q^\ell\!,\,\ulM}\in\Hom_R\bigl(\sigma^{\ell*}\ulM,\ulM)$ is an isogeny, called the \emph{$q^\ell$-Frobenius isogeny} of $\ulM$. It is always inseparable, because the $\ell$-th power of $\tau_M$, which equals ${\rm Fr}_{q^\ell\!,\,\ulM}$ annihilates the cokernel of ${\rm Fr}_{q^\ell\!,\,\ulM}$.

If $\ulM$ is not effective, let $n\in\BN_{>0}$ be such that $\epsilon^n=(a)$ is principal. Then $(a\otimes1)\subset\CJ$ and $(a\otimes1)\subset\sigma^{i*}\CJ$ for all $i$. This shows that 
\begin{equation}\label{EqFrobenius2}
{\rm Fr}_{q^\ell\!,\,\ulM}\;:=\;\tau_M^\ell\;:=\;\tau_M\circ\sigma^*\tau_M\circ\ldots\circ\sigma^{(\ell-1)*}\tau_M\colon\sigma^{\ell*}\ulM[\tfrac{1}{a}] \isoto\ulM[\tfrac{1}{a}]
\end{equation}
is a quasi-isogeny in $\QHom_R\bigl(\sigma^{\ell*}\ulM,\ulM)$ by Theorem~\ref{ThmDualIsog}, called the \emph{$q^\ell$-Frobenius quasi-isogeny} of $\ulM$.

Finally if $R=k$ is a field contained in $\BF_{q^\ell}$ then $\sigma^{\ell*}\ulM=\ulM$ and ${\rm Fr}_{q^\ell\!,\,\ulM}\in\QEnd_k\bigl(\ulM)$, respectively ${\rm Fr}_{q^\ell\!,\,\ulM}\in\End_k\bigl(\ulM)$ if $\ulM$ is effective. In this case, $A[\pi]$ lies in the center of $\End_k(\ulM)$ and $Q[\pi]$ lies in the center of $\QEnd_k(\ulM)$, because every $f\in\End_k(\ulM)$ satisfies $f\circ\tau_M=\tau_M\circ\sigma^*f$ and $\sigma^{\ell*}f=f$. If $k=\BF_{q^\ell}$, the center equals $A[\pi]$, respectively $Q[\pi]$, and the isogeny classes of $A$-motives are largely controlled by their Frobenius endomorphism; see \cite[Theorems~8.1 and 9.1]{BH2}.
\end{example}

%%%%%%%%%%%%%%%%%%%%%%%%%%%%%%%%%%%%%%%%%%%%%%%%%%%%%%%%%%%%%%%%%%%%%%
%
%    Torsion points
%
%%%%%%%%%%%%%%%%%%%%%%%%%%%%%%%%%%%%%%%%%%%%%%%%%%%%%%%%%%%%%%%%%%%%%%

\section{Torsion points}\label{SectTorsionPts}
\setcounter{equation}{0}

\begin{definition}\label{DefFaTorsion}
Let $(0)\ne\Fa=(a_1,\ldots,a_n)\subset A$ be an ideal and let $\ulE=(E,\phi)$ be an abelian Anderson $A$-module over $R$. Then 
\[
\ulE[\Fa]\;:=\;\ker\bigl(\phi_{a_1,\ldots,a_n}:=(\phi_{a_1},\ldots,\phi_{a_n})\colon E \longto E^n\bigr)
\]
is called the \emph{$\Fa$-torsion submodule of $\ulE$}.
\end{definition}

This definition is independent of the generators $(a_1,\ldots,a_n)$ of $\Fa$ by the following

\begin{lemma}\label{LemmaWellDef}
\begin{enumerate}
\item \label{LemmaWellDef_A}
If $(a_1,\ldots,a_n)\subset(b_1,\ldots,b_m)\subset A$ are ideals then $\ker(\phi_{b_1,\ldots,b_m})\into\ker(\phi_{a_1,\ldots,a_n})$ is a closed immersion.
\item \label{LemmaWellDef_B}
If $(a_1,\ldots,a_n)=(b_1,\ldots,b_m)$ then $\ker(\phi_{b_1,\ldots,b_m})=\ker(\phi_{a_1,\ldots,a_n})$.
\item \label{LemmaWellDef_C}
For any $R$-algebra $S$ we have $\ulE[\Fa](S)\,=\,\{\,P\in E(S)\colon\phi_a(P)=0\text{ for all }a\in\Fa\,\}$.
\item \label{LemmaWellDef_D}
$\ulE[\Fa]$ is an $A/\Fa$-module via $A/\Fa\to\End_R(\ulE[\Fa]),\,\bar b\mapsto\phi_b$.
\item \label{LemmaWellDef_E}
$\ulE[\Fa]$ is a finite $R$-group scheme of finite presentation.
\end{enumerate}
\end{lemma}

\begin{proof}
\ref{LemmaWellDef_A} By assumption there are elements $c_{ij}\in A$ with $a_i=\sum_jc_{ij}b_j$. Therefore $\phi_{a_i}=\sum_j\phi_{c_{ij}}\phi_{b_j}$ and the composition of $\phi_{b_1,\ldots,b_m}\colon E\to E^m$ followed by $(\phi_{c_{ij}})_{i,j}\colon E^m\to E^n$ equals $\phi_{a_1,\ldots,a_n}\colon E\to E^n$. This proves \ref{LemmaWellDef_A} and clearly \ref{LemmaWellDef_A} implies \ref{LemmaWellDef_B}.

\smallskip\noindent
To prove \ref{LemmaWellDef_C} let $P\colon\Spec S\to\ulE$ be an $S$-valued point in $\ulE(S)$ with $0=\phi_a(P):=\phi_a\circ P$ for all $a\in\Fa$. If $\Fa=(a_1,\ldots,a_n)$ then in particular $\phi_{a_i}\circ P=0$ for $i=1,\ldots,n$. Therefore $P$ factors through $\ker\phi_{a_1,\ldots,a_n}=\ulE[\Fa]$. 

Conversely let $P\colon\Spec S\to\ulE[\Fa]$ be an $S$-valued point in $\ulE[\Fa](S)$ and let $a\in\Fa$. By \ref{LemmaWellDef_B} we may write $\Fa=(a_1,\ldots,a_n)$ with $a_1=a$ to have $\ulE[\Fa]=\ker\phi_{a_1,\ldots,a_n}$. Therefore $\phi_a(P):=\phi_a\circ P=0$. This proves \ref{LemmaWellDef_C}.

\smallskip\noindent
\ref{LemmaWellDef_D} The relation $ab=ba$ in $A$ implies $\phi_a\circ\phi_b=\phi_b\circ\phi_a$. Using that the closed subscheme $\ulE[\Fa]$ is uniquely determined by \ref{LemmaWellDef_C} it follows that the ring homomorphism $A\to\End_R(\ulE[\Fa]),\,b\mapsto\phi_b|_{\ulE[\Fa]}$ is well defined. If $b\in\Fa$ then clearly $\phi_b|_{\ulE[\Fa]}=0$ and so this ring homomorphism factors through $A/\Fa$.

\smallskip\noindent
\ref{LemmaWellDef_E} If $\Fa=(a_1,\ldots,a_n)$ then $\ulE[\Fa]=\ker\phi_{a_1,\ldots,a_n}$ is of finite presentation, because $\phi_{a_1,\ldots,a_n}$ is a morphism of finite presentation between the schemes $E$ and $E^n$ of finite presentation over $R$ by \cite[IV$_1$, Proposition~1.6.2]{EGA}. The finiteness of $\ulE[\Fa]$ follows for $\Fa=(a)$ from Corollaries~\ref{CorPhi_a} and \ref{CorKernel}, and for general $\Fa$ from \ref{LemmaWellDef_A} by considering some $(a)\subset\Fa$.
\end{proof}

The following lemma is a version of the Chinese remainder theorem in our context.

\begin{lemma}\label{LemmaProduct}
Let $(0)\ne\Fa,\Fb\subset A$ be two ideals with $\Fa+\Fb=A$. 
\begin{enumerate}
\item \label{LemmaProduct_A}
For an abelian Anderson $A$-module $\ulE$ there is a canonical isomorphism $\ulE[\Fa]\times_R\ulE[\Fb]\isoto\ulE[\Fa\Fb]$.
\item \label{LemmaProduct_B}
For an effective $A$-motive $\ulM$ there is a canonical isomorphism $\ulM/\Fa\Fb\ulM\isoto\ulM/\Fa\ulM\oplus\ulM/\Fb\ulM$ of finite $\BF_q$-shtukas. 
\end{enumerate}
\end{lemma}

\begin{proof}
By the Chinese remainder theorem there is an isomorphism $A/\Fa\Fb\isoto A/\Fa\times A/\Fb$ whose inverse is given by $(x_\Fa,x_\Fb)\mapsto bx_\Fa+ax_\Fb$ for certain elements $a\in\Fa$ and $b\in\Fb$ which satisfy $a\equiv1\mod\Fb$ and $b\equiv1\mod\Fa$, and hence $a+b\equiv1\mod\Fa\Fb$.

\smallskip\noindent
\ref{LemmaProduct_B} follows directly from this, because $\ulM/\Fa\ulM=\ulM\otimes_AA/\Fa$.

\smallskip\noindent
\ref{LemmaProduct_A} By Lemma~\ref{LemmaWellDef}\ref{LemmaWellDef_A} the addition $\Delta$ on $\ulE[\Fa\Fb]$ defines a canonical morphism $\ulE[\Fa]\times_R\ulE[\Fb]\into\ulE[\Fa\Fb]\times_R\ulE[\Fa\Fb]\xrightarrow{\;\Delta\,}\ulE[\Fa\Fb]$. Its inverse is described as follows. The elements $a,b\in A$ from above satisfy $a\Fb\subset\Fa\Fb$ and $b\Fa\subset\Fa\Fb$. By Lemma~\ref{LemmaWellDef}\ref{LemmaWellDef_C} the endomorphism $\phi_a$ of $\ulE[\Fa\Fb]$ factors through $\ulE[\Fb]$ and $\phi_b$ factors through $\ulE[\Fa]$. So the inverse is the morphism $(\phi_b,\phi_a)\colon\ulE[\Fa\Fb]\to\ulE[\Fa]\times_R\ulE[\Fb]$. Indeed, for $x\in\ulE[\Fa\Fb]$, we compute $\phi_b(x)+\phi_a(x)=\phi_{a+b}(x)=\phi_1(x)=x$, because $a+b\equiv1\mod\Fa\Fb$. On the other hand, for $x\in\ulE[\Fa]$ and $y\in\ulE[\Fb]$, we compute $\phi_b(x+y)=\phi_b(x)=x$ and $\phi_a(x+y)=\phi_a(y)=y$, because $b\equiv1\mod\Fa$ and $a\equiv1\mod\Fb$.
\end{proof}

\begin{theorem}\label{ThmATorsion}
Let $\ulE$ be an abelian Anderson $A$-module and let  $(0)\ne\Fa\subset A$ be an ideal.
\begin{enumerate}
\item \label{ThmATorsion_A}
Then $\ulE[\Fa]$ is a finite locally free group scheme over $\Spec R$ and a strict $\BF_q$-module scheme.
\item \label{ThmATorsion_B}
$\ulE[\Fa]$ is \'etale over $R$ if and only if $R\cdot\gamma(\Fa)=R$, that is if and only if $\Fa+\CJ=A_R$.
\item \label{ThmATorsion_C}
If $\ulM=\ulM(\ulE)$ is the associated effective $A$-motive then there are canonical $A$-equivariant isomorphisms
\begin{eqnarray*}
\ulM/\Fa\ulM & \isoto & \ulM_q(\ulE[\Fa]) \qquad \text{of finite $\BF_q$-shtukas and}\\[2mm]
\Dr_q(\ulM/\Fa\ulM) & \isoto & \ulE[\Fa] \qquad \text{of finite locally free $R$-group schemes.}
\end{eqnarray*}
\end{enumerate}
\end{theorem}

\begin{proof}
Since $A$ is a Dedekind domain, $\Fa=\Fp_1^{e_1}\cdot\ldots\cdot\Fp_r^{e_r}$ for prime ideals $\Fp_i\in A$ and positive integers $e_i$. By Lemma~\ref{LemmaProduct} and the exactness of the functors $\Dr_q$ and $\ulM_q$, see Theorem~\ref{ThmEqAModSch}\ref{ThmEqAModSch_A}, it suffices to treat the case $\Fa=\Fp^e$. Let $A_\Fp$ be the localization of $A$ at $\Fp$. Since $A/\Fp^e=A_\Fp/\Fp^e A_\Fp$ there is an element $z\in A$ which is congruent modulo $\Fa$ to a uniformizer of $A_\Fp$. Moreover, since $\ulE[\Fp^e]$ is an $A_\Fp/\Fp^e A_\Fp$-module, every $\phi_s$ with $s\in A\setminus\Fp$ is an automorphism of $\ulE[\Fp^e]$. Let $0\le n\le e$. We denote the inclusion $\ulE[\Fp^n]\into\ulE[\Fp^e]$ of Lemma~\ref{LemmaWellDef}\ref{LemmaWellDef_A} by $i_{n,e}$. By Lemma~\ref{LemmaWellDef}\ref{LemmaWellDef_C} the endomorphism $\phi_z^{e-n}$ of $\ulE[\Fp^e]$ has kernel $\ulE[\Fp^{e-n}]$ and factors through the closed subscheme $\ulE[\Fp^n]$ via a morphism $j_{e,n}\colon\ulE[\Fp^e]\to\ulE[\Fp^n]$ with $\phi_z^{e-n}=i_{n,e}\circ j_{e,n}$. We claim that $j_{e,n}$ is an epimorphism in the category of sheaves on the big \fpqc-site over $\Spec R$, and we therefore have an exact sequence
\begin{equation}\label{EqTorsionExSeq}
\xymatrix @C+2pc { 0 \ar[r] & \ulE[\Fp^{e-n}] \ar[r]^{\TS i_{e-n,e}} & \ulE[\Fp^e] \ar[r]^{\TS j_{e,n}\quad} & \ulE[\Fp^n] \ar[r] & 0\,.
}
\end{equation}
To prove the claim let $S$ be an $R$-algebra and let $P\colon\Spec S\to\ulE[\Fp^n]$ be an $S$-valued point in $\ulE[\Fp^n](S)$. Since $\phi_{z^{e-n}}\colon\ulE\to\ulE$ is an isogeny by Corollary~\ref{CorPhi_a}, hence an epimorphism of \fpqc-sheaves by Proposition~\ref{PropIsogAModule}\ref{PropIsogAModule_F}, there exists a faithfully flat $S$-algebra $S'$ and a point $P'\in E(S')$ with $\phi_{z^{e-n}}(P')=P$. We have to show that $P'\in\ulE[\Fp^e](S')$. For this purpose let $a\in\Fp^e$. Then $\tfrac{a}{1}=\tfrac{c}{s}(\tfrac{z}{1})^e$ in $A_\Fp$ for $c\in A,\,s\in A\setminus\Fp$. We compute
\[
\phi_a(P')\,=\,\phi_s^{-1}\circ\phi_c\circ\phi_{z^n}\circ\phi_{z^{e-n}}(P')\,=\,\phi_s^{-1}\circ\phi_c\circ\phi_{z^n}(P)\,=\,0\,,
\]
because $z^n\in\Fp^n$. This proves our claim and establishes the exactness of \eqref{EqTorsionExSeq}.

We now use that $A$ is a Dedekind domain with finite ideal class group. This means that for the prime ideal $\Fp\subset A$ there are (arbitrarily large) integers $e$ such that $\Fp^e=(a)$ is principal. Then $\ulE[\Fp^e]=\ker\phi_a$ is a finite locally free $R$-group scheme by Corollaries~\ref{CorPhi_a} and \ref{CorKernel}. If $0\le n\le e$ then we show that $\ulE[\Fp^n]$ is flat over $R$. Namely, using the epimorphism $j_{e,n}\colon\ulE[\Fp^e]\to\ulE[\Fp^n]$ from \eqref{EqTorsionExSeq} and the flatness of $\ulE[\Fp^e]$ over $R$, the flatness of $\ulE[\Fp^n]$ will follow from \cite[IV$_3$, Th\'eor\`eme~11.3.10]{EGA} once we show that $j_{e,n}$ is flat in each fiber over a point of $\Spec R$. This follows from \cite[\S\,III.3, Corollaire~7.4]{DemazureGabriel} and so $\ulE[\Fp^n]$ is flat over $R$ for all $n$. By Lemma~\ref{LemmaWellDef}\ref{LemmaWellDef_E} this proves that $\ulE[\Fp^n]$ is a finite locally free group scheme over $\Spec R$. Moreover, it is a strict $\BF_q$-module scheme by \cite[Proposition~2]{Faltings02}, because for $\Fp^n=(a_1,\ldots,a_n)$ the morphism $\phi_{a_1,\ldots,a_n}$ is strict $\BF_q$-linear by Example~\ref{ExampleStrictMorphism}. So \ref{ThmATorsion_A} is established.

If $\Fa=\Fp^e=(a)$ we know from Theorem~\ref{ThmIsogeny}\ref{ThmIsogeny_C} applied to the isogeny $\phi_a$ and $\coker\ulM(\phi_a)=\ulM/a\ulM$ that \ref{ThmATorsion_C} holds. If $0\le n\le e$ we use the exact sequence \eqref{EqTorsionExSeq} and the fact that the functors $\Dr_q$ and $\ulM_q$ are exact by Theorem~\ref{ThmEqAModSch}. Namely, multiplication with $z^{e-n}$ on $\ulM/a\ulM$ has cokernel $\ulM/\Fp^{e-n}\ulM$ and image isomorphic to $\ulM/\Fp^n\ulM$. We obtain an exact sequence of finite $\BF_q$-shtukas
\begin{equation}\label{EqTorsionExSeq2}
\xymatrix @C+2pc { 0 \ar[r] & \ulM/\Fp^n\ulM \ar[r]^{\TS \beta_{n,e}} & \ulM/a\ulM \ar[r]^{\TS \alpha_{e,e-n}\quad} & \ulM/\Fp^{e-n}\ulM \ar[r] & 0
}
\end{equation}
with $\beta_{n,e}\circ\alpha_{e,n}=z^{e-n}$ on $\ulM/a\ulM$. Applying $\Dr_q$ to \eqref{EqTorsionExSeq2}, using the exactness of $\Dr_q$, and that $\Dr_q(\ulM/a\ulM)=\ulE[\Fp^e]$ and $\Dr_q(z^{e-n})=\phi_z^{e-n}$, proves $\Dr_q(\ulM/\Fp^n\ulM)=\ulE[\Fp^n]$. Conversely applying $\ulM_q$ to \eqref{EqTorsionExSeq}, using the exactness of $\ulM_q$, and that $\ulM/a\ulM=\ulM(\ulE[\Fp^e])$ and $z^{e-n}=\ulM_q(\phi_z^{e-n})$, proves $\ulM/\Fp^n\ulM=\ulM_q(\ulE[\Fp^n])$. This establishes \ref{ThmATorsion_C} in general.

\medskip\noindent
\ref{ThmATorsion_B} Let $R\cdot\gamma(\Fa)=R$, that is there are elements $a_1,\ldots,a_n\in\Fa$ and $b_1,\ldots,b_n\in R$ with $\sum_{i=1}^nb_i\gamma(a_i)=1$. Then the open subschemes $\Spec R[\tfrac{1}{\gamma(a_i)}]\subset\Spec R$ cover $\Spec R$ and it suffices to check that $\ulE[\Fa]$ is \'etale over $\Spec R[\tfrac{1}{\gamma(a_i)}]$ for each $i$. But there $\ulE[\Fa]$ is a closed subscheme of $\ulE[a_i]$ which is \'etale by Corollary~\ref{CorPhi_a}. This shows that $\ulE[\Fa]$ is unramified over $R$. Since it is flat by \ref{ThmATorsion_A}, it is \'etale as desired.

Conversely assume that $R\cdot\gamma(\Fa)\subset\Fm$ for a maximal ideal $\Fm\subset R$ and set $k=R/\Fm$. Over a field extension $k'$ of $k$  we have $E\times_Rk=\BG_{a,k'}^d=\Spec k'[x_1,\ldots,x_d]$. We will show that $\ulE[\Fa]\times_R k'$ is not \'etale over $k'$ by applying the Jacobi criterion \cite[\S 2.2, Proposition~7]{BLR}. Let $\Fa=(a_1,\ldots,a_n)$. Then $\ulE[\Fa]=\Spec k'[x_1,\ldots,x_d]/\bigl(\phi_{a_1}^*(x_1,\ldots,x_d)\colon j=1,\ldots,n\bigr)$. The Jacobi matrix is 
\[
\frac{\partial\phi_{a_j}^*}{\partial x_i}\;=\;\left(\begin{smallmatrix} \TS \Lie \phi_{a_1}\\[-1mm] \TS \vdots \\[1mm] \TS \Lie \phi_{a_n} \end{smallmatrix}\right)\;\in\;(k')^{nd\times d}.
\]
Since $\gamma(a_i)=0$ in $k'$ each $\Lie \phi_{a_i}$ is a nilpotent $d\times d$ matrix. Since $\phi_{a_i}\circ\phi_{a_j}=\phi_{a_ia_j}=\phi_{a_j}\circ\phi_{a_i}$ we have $\Lie \phi_{a_i}(\ker \Lie \phi_{a_j})\subset\ker \Lie \phi_{a_j}$. Therefore all $\ker \Lie \phi_{a_i}$ have a non-trivial intersection. This shows that the rank of the Jacobi matrix is less than $d$ and $\ulE[\Fa]\times_R k'$ is not \'etale over $k'$.
\end{proof}

\begin{proposition}\label{PropEtFundGp}
Let $\ulM=(M,\tau_M)$ be an $A$-motive over $R$ of rank $r$ and let $(0)\ne\Fa\subset A$ be an ideal with $R\cdot\gamma(\Fa)=R$, that is $\Fa+\CJ=A_R$. Let $\bar s=\Spec\Omega$ be a geometric base point of $\Spec R$. Then $\ulM/\Fa\ulM$ is an \'etale finite $\BF_q$-shtuka whose $\tau$-invariants $(\ulM/\Fa\ulM)^\tau(\Omega)$, see \eqref{EqTauInv}, form a free $A/\Fa$-module of rank $r$ which carries a continuous action of the \'etale fundamental group $\pi_1^\et(\Spec R,\bar s)$.
\end{proposition}

\begin{proof}
This result and its proof are due to Anderson~\cite[Lemma~1.8.2]{Anderson86} for $R$ a field. We let $G:=\Res_{A/\Fa|\BF_q}\GL_{r,A/\Fa}$ be the Weil restriction with $G(R')=\GL_r(A/\Fa\otimes_{\BF_q}R')$ for all $\BF_q$-algebras $R'$. Then $G$ is a smooth connected affine group scheme over $\BF_q$ by \cite[Proposition~A.5.9]{CGP}. Thus by Lang's theorem \cite[Corollary on p.~557]{Lang56} the Lang map $L\colon G\to G, g\mapsto g\cdot\sigma^*g^{-1}$ is finite \'etale and surjective (although not a group homomorphism if $r>1$ and $\Fa\ne A$).

Since $\Fa+\CJ=A_R$ the isomorphism $\tau_M\colon\sigma^*M|_{\Spec A_R\setminus\Var(\CJ)}\isoto M|_{\Spec A_R\setminus\Var(\CJ)}$ of $\ulM$ induces an isomorphism $\tau_{M/\Fa M}\colon \sigma^*M/\Fa M\isoto M/\Fa M$ and makes $\ulM/\Fa\ulM$ into a finite $\BF_q$-shtuka, which is \'etale. After passing to a covering of $\Spec R$ by open affine subschemes, we may assume that there is an isomorphism $\alpha\colon (A/\Fa)^r \otimes_{\BF_q}R\isoto M/\Fa M$ and then $\alpha^{-1}\circ\tau_{M/\Fa M}\circ\sigma^*\alpha$ is an element $b\in G(R)$ and corresponds to a morphism $b\colon\Spec R\to G$. The fiber product $\Spec R\underset{b,G,L}{\times}G$ is finite \'etale over $\Spec R$ and of the form $\Spec R'$. The projection onto the second factor $G$ corresponds to an element $c\in G(R')$ with $c\cdot\sigma^*c^{-1}=b$, that is $c=b\cdot\sigma^*c$. This implies $\alpha\circ c=\tau_{M/\Fa M}\circ\sigma^*(\alpha\circ c)$, and thus $\alpha\circ c$ is an isomorphism $(A/\Fa)^r\isoto (\ulM/\Fa\ulM)^\tau(R')\;:=\;\{\,m\otimes M/\Fa M \otimes_RR'\colon m=\tau_M(\sigma_M^*m)\,\}$. The proposition follows from this. 
\end{proof}

\begin{theorem}\label{ThmEtFundGp}
Let $\ulE$ be an abelian Anderson $A$-module over $R$ of rank $r$ and let $\ulM=\ulM(\ulE)$ be its associated effective $A$-motive. Let $(0)\ne\Fa\subset A$ be an ideal with $R\cdot\gamma(\Fa)=R$, that is $\Fa+\CJ=A_R$. Then for every $R$-algebra $R'$ such that $\Spec R'$ is connected, there is an isomorphism of $A/\Fa$-modules
\begin{eqnarray*}
\ulE[\Fa](R') & \isoto & \Hom_{A/\Fa}\bigl((\ulM/\Fa\ulM)^\tau(R')\,,\,\Hom_{\BF_q}(A/\Fa,\BF_q)\bigr)\,,\\[2mm]
P & \longmapsto & \qquad\qquad\quad\bigl[\;\ol m \qquad\longmapsto \qquad[\bar a\mapsto m\circ\phi_a(P)]\;\bigr]\,.
\end{eqnarray*}
In particular, if $\bar s=\Spec\Omega$ is a geometric base point of $\Spec R$, then $\ulE[\Fa](\Omega)$ is a free $A/\Fa$-module of rank $r$ which carries a continuous action of the \'etale fundamental group $\pi_1^\et(\Spec R,\bar s)$.
\end{theorem}

\begin{proof}
This result and its proof are due to Anderson~\cite[Proposition~1.8.3]{Anderson86} for $R$ a field. For general $R$ the proof was carried out in \cite[Lemma~2.4 and Theorem~8.6]{BoeckleHartl}. The last statement follows from Proposition~\ref{PropEtFundGp}.
\end{proof}

%%%%%%%%%%%%%%%%%%%%%%%%%%%%%%%%%%%%%%%%%%%%%%%%%%%%%%%%%%%%%%%%%%%%%%
%
%    Divisible local Anderson modules
%
%%%%%%%%%%%%%%%%%%%%%%%%%%%%%%%%%%%%%%%%%%%%%%%%%%%%%%%%%%%%%%%%%%%%%%

\section{Divisible local Anderson modules}\label{SectDivLocAndMod}
\setcounter{equation}{0}

In this section we consider the situation where $\epsilon\subset A$ is a maximal ideal and the elements of $\gamma(\epsilon)\subset R$ are nilpotent. Let $\hat q$ be the cardinality of the residue field $\BF_\epsilon=A/\epsilon$ and $f=[\BF_\epsilon:\BF_q]$, that is $\hat q=q^f$. We fix a uniformizing parameter $z\in\Quot(A)$ at $\epsilon$. It defines an isomorphism $\BF_\epsilon\dbl z\dbr\isoto \wh{A}_\epsilon:=\invlim A/\epsilon^n$. We consider the $\epsilon$-adic completion $\wh{A}_{\epsilon,R}:=\invlim A_R/\epsilon^n=(\BF_\epsilon\otimes_{\BF_q}R)\dbl z\dbr$. By continuity the map $\gamma$ extends to a ring homomorphism $\gamma\colon \wh{A}_\epsilon\to R$. We consider the ideals $\Fa_i=(a\otimes1-1\otimes\gamma(a)^{q^i}\colon a\in\BF_\epsilon)\subset \wh{A}_{\epsilon,R}$ for $i\in\BZ/f\BZ$. By the Chinese remainder theorem $\wh{A}_{\epsilon,R}$ decomposes
\[
\wh{A}_{\epsilon,R}\;=\;(\BF_\epsilon\otimes_{\BF_q}R)\dbl z\dbr\;=\;\prod_{i\in\BZ/f\BZ}\wh{A}_{\epsilon,R}/\Fa_i\,,
\]
and $\wh{A}_{\epsilon,R}/\Fa_i$ is the subset of $\wh{A}_{\epsilon,R}$ on which $a\otimes1$ acts as $1\otimes\gamma(a)^{q^i}$ for all $a\in\BF_\epsilon$. Each factor is canonically isomorphic to $R\dbl z\dbr$. The factors are cyclically permuted by $\sigma$ because $\sigma(\Fa_i)=\Fa_{i+1}$. In particular $\hat\sigma:=\sigma^f$ stabilizes each factor and acts on it via $\hat\sigma(z)=z$ and $\hat\sigma(b)=b^{\hat q}$ for $b\in R$. The ideal $\CJ:=(a\otimes1-1\otimes\gamma(a)\colon a\in A)\subset A_R$ decomposes as follows $\CJ\!\cdot\! \wh{A}_{\epsilon,R}/\Fa_0=(z-\gamma(z))$ and $\CJ\!\cdot\! \wh{A}_{\epsilon,R}/\Fa_i=(1)$ for $i\ne0$. In particular, $\wh{A}_{\epsilon,R}/\Fa_0$ equals the $\CJ$-adic completion of $A_R$, as $\gamma(z)$ is nilpotent in $R$; compare also \cite[Lemma~5.3]{AH_Local}. We also set $R\dpl z\dpr:=R\dbl z\dbr[\tfrac{1}{z}]$.

\begin{definition}\label{DefLocSht}
A \emph{local $\hat\sigma$-shtuka} (or \emph{local shtuka}) \emph{of rank} $r$ over $R$ is a pair $\ulHM=(\hat M,\tau_{\hat M})$ consisting of a locally free $R\dbl z\dbr$-module $\hat M$ of rank $r$, and an isomorphism $\tau_{\hat M}\colon\hat\sigma^\ast\hat M[\frac{1}{z-\gamma(z)}] \isoto\hat M[\frac{1}{z-\gamma(z)}]$. If $\tau_{\hat M}(\hat\sigma^*\hat M)\subset\hat M$ then $\ulHM$ is called \emph{effective}, and if $\tau_{\hat M}(\hat\sigma^*\hat M)=\hat M$ then $\ulHM$ is called \emph{\'etale}.

A \emph{morphism} of local shtukas $f\colon(\hat M,\tau_{\hat M})\to(\hat M',\tau_{\hat M'})$ over $R$ is a morphism of $R\dbl z\dbr$-modules $f\colon\hat M\to\hat M'$ which satisfies $\tau_{\hat M'}\circ \hat\sigma^\ast f = f\circ \tau_{\hat M}$.
\end{definition}

\begin{example}\label{ExLocShtuka}
Let $\ulM=(M,\tau_M)$ be an $A$-motive over $R$. We consider the $\epsilon$-adic completion $\ulM\otimes_{A_R}\wh{A}_{\epsilon,R}:=(M\otimes_{A_R}\wh{A}_{\epsilon,R}\,,\,\tau_M\otimes1)=\invlim\ulM/\epsilon^n\ulM$. We define the \emph{local $\hat\sigma$-shtuka at $\epsilon$ associated with $\ulM$} as $\ulHM_\epsilon(\ulM):=\bigl(M\otimes_{A_R}\wh{A}_{\epsilon,R}/\Fa_0\,,\,(\tau_M\otimes1)^f\bigr)$, where $\tau_M^f:=\tau_M\circ\sigma^*\tau_M\circ\ldots\circ\sigma^{(f-1)*}\tau_M$. It equals the $\CJ$-adic completion of $\ulM$ and therefore is effective if and only if $\ulM$ is effective, because of Proposition~\ref{PropAMotiveEffective}. Of course if $\BF_\epsilon=\BF_q$, and hence $\hat q=q$ and $\hat\sigma=\sigma$, we have $\wh{A}_{\epsilon,R}=R\dbl z\dbr$ and $\ulHM_\epsilon(\ulM)=\ulM\otimes_{A_R}\wh{A}_{\epsilon,R}$.

Also for $f>1$ the local shtuka $\ulHM_\epsilon(\ulM)$ allows to recover $\ulM\otimes_{A_R}\wh{A}_{\epsilon,R}$ via the isomorphism
\[
\bigoplus_{i=0}^{f-1}(\tau_M\otimes1)^i\mod\Fa_i\colon\Bigl(\bigoplus_{i=0}^{f-1}\sigma^{i*}(M\otimes_{A_R}\wh{A}_{\epsilon,R}/\Fa_0),\;(\tau_M\otimes1)^f\oplus\bigoplus_{i\ne0}\id\Bigr)\;\isoto\;\ulM\otimes_{A_R}\wh{A}_{\epsilon,R}\,,
\]
because for $i\ne0$ the equality $\CJ\!\cdot\! \wh{A}_{\epsilon,R}/\Fa_i=(1)$ implies that $\tau_M\otimes 1$ is an isomorphism modulo $\Fa_i$; see \cite[Example~2.2]{HartlKim} or \cite[Propositions~8.8 and 8.5]{BH1} for more details.
\end{example}

Let $\ulHM=(\hat M,\tau_{\hat M})$ be an effective local shtuka over $R$. Set $\ulHM_n:=(\hat M_n,\tau_{\hat M_n}):=(\hat M/z^n\hat M,\tau_{\hat M}\mod z^n)$ and $G_n:=\Dr_{\hat q}(\ulHM_n)$. Then $G_n$ is a finite locally free strict $\BF_\epsilon$-module scheme over $R$ and $\ulHM_n=\ulM_{\hat q}(G_n)$ by Theorem~\ref{ThmEqAModSch}. Moreover, $G_n$ inherits from $\ulHM_n$ an action of $\BF_\epsilon[z]/(z^n)$. The canonical epimorphisms $\ulHM_{n+1}\onto\ulHM_n$ induce closed immersions $i_n\colon G_n\into G_{n+1}$. The inductive limit $\Dr_{\hat q}(\ulHM):=\dirlim G_n$ in the category of sheaves on the big \fppf-site of $\Spec R$ is a sheaf of $\BF_\epsilon\dbl z\dbr$-modules that satisfies the following

\begin{definition}\label{DefZDivGp}
A \emph{$\name$-divisible local Anderson module over $R$} is a sheaf of $\BF_\epsilon\dbl z\dbr$-modules $G$ on the big \fppf-site of $\Spec R$ such that
\begin{enumerate}
\item \label{DefZDivGpAxiom1}
$G$ is \emph{$\name$-torsion}, that is $G = \dirlim G[z^n]$, where $G[z^n]:=\ker(z^n\colon G\to G)$,
\item \label{DefZDivGpAxiom2}
$G$ is \emph{$\name$-divisible}, that is $z\colon G \to G$ is an epimorphism,
\item \label{DefZDivGpAxiom3}
For every $n$ the $\BF_\epsilon$-module $G[z^n]$ is representable by a finite locally free strict $\BF_\epsilon$-module scheme over $R$ (Definition~\ref{DefStrict}), and
\item \label{DefZDivGpAxiom4}
there exist an integer $d \in \BZ_{\geq 0}$, such that $(z-\gamma(z))^d=0$ on $\omega_G$ where $\omega_G := \invlim\omega_{G[z^n]}$ and $\omega_{G[z^n]}=e^*\Omega^1_{G[z^n]/\Spec R}$ is the pullback under the zero section $e\colon\Spec R\to G[z^n]$.
\end{enumerate}
A \emph{morphism of $\name$-divisible local Anderson modules over $R$} is a morphism of \fppf-sheaves of $\BF_\epsilon\dbl z\dbr$-modules. The category of divisible local Anderson modules is $\BF_\epsilon\dbl z\dbr$-linear.
It is shown in \cite[Lemma~8.2]{HartlSingh} that $\omega_G$ is a finite locally free $R$-module and we define the \emph{dimension of $G$} as $\rk\omega_G$\,. A $\name$-divisible local Anderson module is called \emph{\'etale} if $\omega_G = 0$.  Since $\omega_G$ surjects onto each $\omega_{G[z^n]}$, this is the case if and only if all $G[z^n]$ are \'etale, see \cite[Lemma~3.7]{HartlSingh}.
%
%If $\dim G$ is constant and the integer $d:=\dim G$ satisfies $(z-\gamma(z))^d\cdot\omega_G=(0)$ globally on $R$ in axiom~\ref{DefZDivGpAxiom4}, we say that $G$ is \emph{bounded by $d$}.
\end{definition}

Conversely with a $\name$-divisible local Anderson module $G$ over $R$ one associates the local shtuka $\ulM_{\hat q}(G):=\invlim\ulM_{\hat q}(G[z^n])$. Multiplication with $z$ on $G$ gives $M_{\hat q}(G)$ the structure of an $R\dbl z\dbr$-module. In \cite[Theorem~8.3]{HartlSingh} we proved the following

\begin{theorem} \label{ThmEqZDivGps}
\begin{enumerate}
\item  \label{ThmEqZDivGps_A}
The two contravariant functors $\Dr_{\hat q}$ and $\ulM_{\hat q}$ are mutually quasi-inverse anti-equivalences between the category of effective local shtukas over $R$ and the category of $\name$-divisible local Anderson modules over $R$. 
\item  \label{ThmEqZDivGps_B}
Both functors are $\BF_\epsilon\dbl z\dbr$-linear and map short exact sequences to short exact sequences. They preserve \'etale objects.
\end{enumerate}
\medskip\noindent
Let $\ulHM=(\hat M,\tau_{\hat M})$ be an effective local shtuka over $S$ and let $G=\Dr_{\hat q}(\ulHM)$ be its associated $\name$-divisible local Anderson module. Then
\begin{enumerate}
\setcounter{enumi}{2}
\item  \label{ThmEqZDivGps_C}
$G$ is a formal Lie group if and only if $\tau_{\hat M}$ is topologically nilpotent, that is $\im(\tau_{\hat M}^n)\subset z\hat M$ for an integer $n$.
%\item  \label{ThmEqZDivGps_D}
%the height (see Proposition~\ref{PropTate}) and dimension of $G$ are equal to the rank and dimension of $\ulM$.
\item  \label{ThmEqZDivGps_E}
the $R\dbl z\dbr$-modules $\omega_{\Dr_{\hat q}(\ulHM)}$ and $\coker\tau_{\hat M}$ are canonically isomorphic.
\end{enumerate}
\end{theorem}

We now want to show that for an abelian Anderson $A$-module $\ulE$ over $R$ the local shtuka $\ulHM_\epsilon\bigl(\ulM(\ulE)\bigr)$ corresponds to the $\epsilon$-power torsion of $\ulE$ as in the following

\begin{definition}\label{DefDivLocAModOfAMod}
Let $\ulE$ be an abelian Anderson $A$-module over $R$ and assume that the elements of $\gamma(\epsilon)\subset R$ are nilpotent. We define $\ulE[\epsilon^\infty]:=\dirlim\ulE[\epsilon^n]$ and call it the \emph{$\name$-divisible local Anderson module} associated with $\ulE$. 
\end{definition}

This definition is justified by the following 

\begin{theorem}\label{ThmDivLocAModOfAMod}
Let $\ulE=(E,\phi)$ be an abelian Anderson $A$-module over $R$ and assume that the elements of $\gamma(\epsilon)\subset R$ are nilpotent. Then
\begin{enumerate}
\item \label{ThmDivLocAModOfAMod_A}
all $\ulE[\epsilon^n]$ are finite locally free strict $\BF_\epsilon$-module schemes,
\item \label{ThmDivLocAModOfAMod_B}
$\ulE[\epsilon^\infty]$ is a $\name$-divisible local Anderson module over $R$,
\item \label{ThmDivLocAModOfAMod_C}
If $\ulM=\ulM(\ulE)$ is the associated effective $A$-motive of $\ulE$ and $\ulHM:=\ulHM_\epsilon(\ulM)=\ulM\otimes_{A_R}\wh{A}_{\epsilon,R}/\Fa_0$ is the local $\hat\sigma$-shtuka at $\epsilon$ associated with $\ulM$, then there are canonical isomorphisms
\[
\begin{array}{rclcrcl}
\ulM_{\hat q}(\ulE[\epsilon^\infty]) & \cong & \ulHM_\epsilon(\ulM) & \quad \text{and} \quad & \ulE[\epsilon^\infty] & \cong & \Dr_{\hat q}\bigl(\ulHM_\epsilon(\ulM)\bigr)\,,\\[2mm]
\ulM_{q}(\ulE[\epsilon^\infty]) & \cong & \ulM\otimes_{A_R}\wh{A}_{\epsilon,R} & \text{and} & \ulE[\epsilon^\infty] & \cong & \Dr_{q}\bigl(\ulM\otimes_{A_R}\wh{A}_{\epsilon,R}\bigr)\,,\\[2mm]
\ulM_{\hat q}(\ulE[\epsilon^n]) & \cong & \ulHM/\epsilon^n\ulHM & \text{and} & \ulE[\epsilon^n] & \cong & \Dr_{\hat q}(\ulHM/\epsilon^n\ulHM)\,.
\end{array}
\]
\end{enumerate}
\end{theorem}

\begin{proof}
\newcommand{\mipo}{g}
\ref{ThmDivLocAModOfAMod_A} By Lemma~\ref{LemmaStrict} we may test strictness after applying a faithfully flat base change to $R$ and assume that $E=\BG_{a,R}^d=\Spec R[x_1,\ldots,x_d]=\Spec R[\ulX]$ and $\ulM(\ulE)=R\{\tau\}^{1\times d}$. We set $B:=\Gamma(\ulE[\epsilon^n],\CO_{\ulE[\epsilon^n]})$ and $I=\ker(R[\ulX]\onto B)$ and $I_0=(x_1,\ldots,x_d)$, and consider the deformation $B^\flat=R[\ulX]/I\!\cdot\!I_0$. The endomorphisms $\phi_a$ of $E$ for $a\in A$ satisfy $\phi_a^*(I)\subset I$ and $\phi_a^*(I_0)\subset I_0$. This defines a lift $A\to\End_{R\text{-algebras}}(B^\flat),\,a\mapsto[a]^\flat:=\phi_a^*$ compatible with addition and multiplication as in Definition~\ref{DefStrict}. 

Let $N\ge\dim\ulE$ be a positive integer which is a power of $\hat q$ such that $\gamma(a)^N=0$ for every $a\in\epsilon^n$. Choose $\lambda\in\BF_\epsilon$ with $\BF_\epsilon=\BF_q(\lambda)$ and let $\mipo$ be the minimal polynomial of $\lambda$ over $\BF_q$. Choose an element $t\in A$ with $t\mod\epsilon^n=\lambda$ in $A/\epsilon^n=\BF_\epsilon\dbl z\dbr/(z^n)$. Then $\mipo(t)\in\epsilon^n$, and hence $\gamma(\mipo(t))^N=0$. On $\Lie E$ the equation $\mipo(t^N)=\mipo(t)^N$ implies $\Lie \phi_{\mipo(t^N)}=\Lie \phi_{\mipo(t)}^N-\gamma(\mipo(t))^N=\bigl(\Lie \phi_{\mipo(t)}-\gamma(\mipo(t))\bigr)^N=0$. So $\phi_{\mipo(t^N)}\,\in\,\End_{R\text{-groups},\BF_q\text{\rm-lin}}(\BG_{a,R}^d)\,=\,R\{\tau\}^{d\times d}$ as a polynomial in $\tau$ has no constant term. This means that $\phi_{\mipo(t^N)}^*(x_i)\in I_0^q$. Moreover, since $\mipo(t)\in\epsilon^n$ we have $\phi_{\mipo(t)}=0$ on $\ulE[\epsilon^n]$ and hence $\phi_{\mipo(t)}^*(x_i)\in I$. Therefore $\phi_{\mipo(t^{\hat qN})}^*(I_0)=\phi_{\mipo(t)}^*\circ\phi_{\mipo(t^{\hat qN-N-1})}^*\circ\phi_{\mipo(t^N)}^*(I_0)\subset\phi_{\mipo(t)}^*(I_0^q)\subset\phi_{\mipo(t)}^*(I_0)^2\subset I\!\cdot\!I_0$. In other words $[\mipo(t^{\hat qN})]^\flat=[0]^\flat$ on $B^\flat$. This shows that the map $\BF_\epsilon=\BF_q[t^{\hat qN}]/(\mipo(t^{\hat qN}))\to\End_{R\text{-algebras}}(B^\flat)$ lifts the action of $\BF_\epsilon\subset\BF_\epsilon\dbl z\dbr/(z^n)$ on $\ulE[\epsilon^n]$ and is compatible with addition and multiplication. 

We compute the induced action on the co-Lie complex $\CoL{\CG/\Spec R}$ of $\CG=(\Spec B,\Spec B^\flat)$. In degree $0$ we have $\ell^{\,0}_{\CG/\Spec R}=\Omega^1_ {R[\ulX]/R} \otimes_{R[\ulX],\,e_{R[\ulX]}} R=\bigoplus_{i=1}^d R\!\cdot\!x_i=I_0/I_0^2$. From $t-\lambda\in\epsilon^n$ we obtain $\gamma(t^{\hat qN})-\gamma(\lambda)=\gamma(t-\lambda)^{\hat qN}=0$ in $R$. On $\Lie E$ this implies $\Lie \phi_{t^{\hat qN}}-\gamma(\lambda)=(\Lie \phi_t-\gamma(t))^{\hat qN}=0$ and therefore $\phi_{t^{\hat qN}}-\gamma(\lambda)\,\in\,\End_{R\text{-groups},\BF_q\text{\rm-lin}}(\BG_{a,R}^d)\,=\,R\{\tau\}^{d\times d}$ as a polynomial in $\tau$ has no constant term. This implies that $\bigl(\phi_{t^{\hat qN}}^*-\gamma(\lambda)\bigr)(I_0)\subset I_0^q\subset I_0^2$. We conclude that $t^{\hat qN}$ acts as the scalar $\gamma(\lambda)$ on $I_0/I_0^2$.

To compute the action of $t^{\hat qN}$ on $\ell^{\,-1}_{\CG/\Spec R}$ we use that by Theorem~\ref{ThmEqAModSch}\ref{ThmEqAModSch_F}, $\CoL{\CG/\Spec R}$ is homotopically equivalent to the complex $0\to\sigma^*M/\epsilon^n\sigma^*M \xrightarrow{\;\tau_M\,}M/\epsilon^n M\to0$ where $\ulM_q(\ulE[\epsilon^n])=\ulM/\epsilon^n\ulM$ and $\ulM=\ulM(\ulE)=(M,\tau_M)$; see Theorem~\ref{ThmATorsion}\ref{ThmATorsion_C}. Since $t^{\hat qN}-\gamma(\lambda)=(t\otimes1-1\otimes\gamma(t))^{\hat qN}=0$ on $\coker\tau_M$ there is an $A_R$-homomorphism $h\colon M\to\sigma^*M$ with $h\,\tau_M=\bigl(t^{\hat qN}-\gamma(\lambda)\bigr)\!\cdot\!\id_{\sigma^*M}$ and $\tau_M \,h=\bigl(t^{\hat qN}-\gamma(\lambda)\bigr)\!\cdot\!\id_M$. This means that $t^{\hat qN}$ is homotopic to the scalar multiplication with $\gamma(\lambda)$ on $0\to\sigma^*M/\epsilon^n\sigma^*M \xrightarrow{\;\tau_M\,}M/\epsilon^n M\to0$, and therefore also on $\CoL{\CG/\Spec R}$. Let $h'\colon I_0/I_0^2\to\ell^{\,-1}_{\CG/\Spec R}=:\ell^{\,-1}$ be this homotopy, that is $(t^{\hat qN}-\gamma(\lambda))|_{\ell^{\,-1}}=h'd$ and $(t^{\hat qN}-\gamma(\lambda))|_{I_0/I_0^2}=dh'$. But we must show that $t^{\hat qN}$ and $\gamma(\lambda)$ are not only homotopic on $\CoL{\CG/\Spec R}$, but equal. 

Since $0=\mipo(t^{\hat qN})=\prod_{i\in\BZ/f\BZ}(t^{\hat qN}-\gamma(\lambda)^{q^i})$ on $\CoL{\CG/\Spec R}$, we can decompose $\ell^{\,-1}=\prod_{i\in\BZ/f\BZ}(\ell^{\,-1})_i$ where $(\ell^{\,-1})_i:=\ker(t^{\hat qN}-\gamma(\lambda)^{q^i}\colon\ell^{-1}\to\ell^{\,-1})$. Since the differential $d$ of $\CoL{\CG/\Spec R}$ is an $R$-homo\-mor\-phism and equivariant for the action of $t^{\hat qN}$, it maps $(\ell^{\,-1})_i$ into $\ker(t^{\hat qN}-\gamma(\lambda)^{q^i}\colon I_0/I_0^2\to I_0/I_0^2)$ which is trivial for $i\ne0$. This shows that $0=h'd=t^{\hat qN}-\gamma(\lambda)=\gamma(\lambda^{q^i}-\lambda)$ on $(\ell^{\,-1})_i$, whence $(\ell^{\,-1})_i=(0)$ for $i\ne0$, because $\gamma(\lambda^{q^i}-\lambda)\in R\mal$. We conclude that $\ell^{\,-1}=(\ell^{\,-1})_0$ and $t^{\hat qN}$ acts as the scalar $\gamma(\lambda)$ on $\ell^{\,-1}$. This proves that $\ulE[\epsilon^n]$ is a finite locally free strict $\BF_\epsilon$-module scheme over $R$.

\medskip\noindent
\ref{ThmDivLocAModOfAMod_B} By construction $\ker(z^n\colon\ulE[\epsilon^\infty]\to\ulE[\epsilon^\infty])=\ulE[\epsilon^n]$ and $\ulE[\epsilon^\infty]$ is $\name$-torsion. Using the epimorphism $j_{n+1,n}\colon\ulE[\epsilon^{n+1}]\onto\ulE[\epsilon^n]$ from \eqref{EqTorsionExSeq} with $i_{n,n+1}\circ j_{n+1,n}=\phi_z$ we see that $\ulE[\epsilon^\infty]$ is $\name$-divisible. In \ref{ThmDivLocAModOfAMod_A} we saw that $\ulE[\epsilon^n]$ is representable by a finite locally free strict $\BF_\epsilon$-module scheme over $R$. It remains to verify condition \ref{DefZDivGpAxiom4} of Definition~\ref{DefZDivGp}. Since $\ulE[\epsilon^n]\into\ulE$ is a closed immersion, $\omega_{\ulE[\epsilon^n]}$ is a quotient of $\omega_\ulE=\Hom_R(\Lie \ulE,R)$. Since $A/\epsilon^n=\BF_\epsilon\dbl z\dbr/(z^n)$, there is an element $a\in A$ with $z-a\in\epsilon^n$, whence $\phi_a=\phi_z$ on $\ulE[\epsilon^n]$. Therefore $(\Lie \phi_a-\gamma(a))^d=0$ on $\Lie \ulE$ implies $(\phi_z-\gamma(z))^{N}=(\phi_a-\gamma(a))^N+\gamma(a-z)^N=0$ on $\omega_{\ulE[\epsilon^n]}$. It follows that $(\phi_z-\gamma(z))^N=0$ on $\omega_{\ulE[\epsilon^\infty]} := \invlim\omega_{\ulE[\epsilon^n]}$, and that $\ulE[\epsilon^\infty]$ is a $\name$-divisible local Anderson module over $R$.

\medskip\noindent
\ref{ThmDivLocAModOfAMod_C} We have $\ulM_{q}(\ulE[\epsilon^n])=\Hom_{R\text{-groups},\BF_q\text{\rm-lin}}(\ulE[\epsilon^n],\BG_{a,R})=\ulM/\epsilon^n\ulM$ and $\ulE[\epsilon^n]=\Dr_q(\ulM/\epsilon^n\ulM)$ by Theorem~\ref{ThmATorsion}\ref{ThmATorsion_C}. This implies 
\[
\ulM_q(\ulE[\epsilon^\infty])\;=\;\invlim\ulM_q(\ulE[\epsilon^n])\;=\;\invlim\ulM/\epsilon^n\ulM\;=\;\ulM\otimes_{A_R}\wh{A}_{\epsilon,R}
\]
and $\ulE[\epsilon^\infty]\,=\,\dirlim\Dr_{q}(\ulM/\epsilon^n\ulM)\,=\,\Dr_{q}(\invlim\ulM/\epsilon^n\ulM)\,=\,\Dr_{q}(\ulM\otimes_{A_R}\wh{A}_{\epsilon,R})$.

On $\ulE[\epsilon^n]$ every $\lambda\in\BF_\epsilon$ acts as $\phi_\lambda$ and on $\BG_{a,R}$ as $\gamma(\lambda)$. Therefore
\begin{eqnarray*}
\ulM_{\hat q}(\ulE[\epsilon^n]) & = & \Hom_{R\text{-groups},\BF_\epsilon\text{\rm-lin}}(\ulE[\epsilon^n],\BG_{a,R})\\[2mm]
 & = & \ulM_{q}(\ulE[\epsilon^n])/\Fa_0\ulM_{q}(\ulE[\epsilon^n])\\[2mm]
 & = & \ulM/\epsilon^n\ulM\otimes_{\wh{A}_{\epsilon,R}}\wh{A}_{\epsilon,R}/\Fa_0\\[2mm]
 & = & \ulHM/\epsilon^n\ulHM\,,
\end{eqnarray*}
where the second equality is due to the fact that $\wh{A}_{\epsilon,R}/\Fa_0$ is the summand of $\wh{A}_{\epsilon,R}$ on which $\lambda\otimes 1$ acts as $1\otimes\gamma(\lambda)$ for all $\lambda\in\BF_\epsilon$. This implies 
\[
\ulM_{\hat q}(\ulE[\epsilon^\infty])=\invlim\ulM/\epsilon^n\ulM\otimes_{\wh{A}_{\epsilon,R}}\wh{A}_{\epsilon,R}/\Fa_0=\ulM\otimes_{A_R}\wh{A}_{\epsilon,R}/\Fa_0=\ulHM_\epsilon(\ulM)=\ulHM\,.
\]
On the other hand, since $\ulE[\epsilon^n]$ is a finite locally free strict $\BF_\epsilon$-module by \ref{ThmDivLocAModOfAMod_A}, $\ulE[\epsilon^n]=\Dr_{\hat q}\bigl(\ulM_{\hat q}(\ulE[\epsilon^n])\bigr)=\Dr_{\hat q}(\ulHM/\epsilon^n\ulHM)$ by Theorem~\ref{ThmEqAModSch}\ref{ThmEqAModSch_E}, and so $\ulE[\epsilon^\infty]=\dirlim\Dr_{\hat q}(\ulHM/\epsilon^n\ulHM)=\Dr_{\hat q}\bigl(\ulHM_\epsilon(\ulM)\bigr)$.
\end{proof}

%%%%%%%%%%%%%%%%%%%%%%%%%%%%%%%%%%%%%%%%%%%%%%%%%%%%%%%%%%%%%%%%%%%%%%
%
%    Bibliography
%
%%%%%%%%%%%%%%%%%%%%%%%%%%%%%%%%%%%%%%%%%%%%%%%%%%%%%%%%%%%%%%%%%%%%%%

\vfill

\begin{minipage}[t]{0.5\linewidth}
\noindent
Urs Hartl\\
Universit\"at M\"unster\\
Mathematisches Institut \\
Einsteinstr.~62\\
D -- 48149 M\"unster
\\ Germany
\\[1mm]
www.math.uni-muenster.de/u/urs.hartl/
\end{minipage}
\begin{minipage}[t]{0.45\linewidth}
\noindent
\\
\\
\\[1mm]
\end{minipage}

\end{document}